\documentclass{amsart}
\usepackage{eurosym}
\usepackage[all]{xy}
\usepackage{a4wide,amssymb,mathrsfs}
\usepackage{algorithm}
\usepackage{algpseudocode}
\usepackage[unicode=true]{hyperref}
\usepackage{stmaryrd}
\usepackage{algorithm}
\usepackage{algpseudocode}
\usepackage{amsmath, amssymb} 
\usepackage{framed} 
\usepackage{enumitem} 

\newtheorem{theorem}{Theorem}[section]
\newtheorem{Co}[theorem]{Corollary}

\newtheorem{Not}[theorem]{Notation}
\newtheorem{De}[theorem]{Definition}
\newtheorem{Ex}[theorem]{Example}

\newtheorem{Le}[theorem]{Lemma}
\newtheorem{Rem}[theorem]{Remark}
\newtheorem{Pro}[theorem]{Proposition}

\def\mod{\textup{mod}}

\let\al\alpha

\let\ge\geqslant
\let\le\leqslant

\def\0{^{\to}}

\def\Div{\textup{Div}}
\def\Gal{\textup{Gal}}

\def\Ker{\textup{Ker}}
\def\Coker{\textup{Coker}}
\def\Res{\textup{Res}}
\def\Stab{\textup{Stab}}

\def\1{^{-1}}

\def\cref#1#2#3{\left(#2\right.\left|\ #3\right)_{#1}}

\def\C{\textup{C}}

\def\S{{\mathsf S}}

\begin{document}
	
\title{On the factorizations of integers via division algorithms for polynomials}
	
\author[Guram Donadze]{Guram Donadze}
\address{Institute of Cybernetics of Georgian Technical University, Zurab Anjaparidze Str. 5, Tbilisi 0186 Georgia}
\email{gdonad@gmail.com}

\author[Adrian Vasiu]{Adrian Vasiu}
\address{State University of Binghamton, 4400 Vestal Pkwy E, Binghamton, NY 13902 USA}
\email{avasiu@binghamton.edu}
\date{\today}
	
\maketitle

\begin{abstract}
We introduce and study several conditions related to the factorization problem of composite numbers. For this purpose, we employ cyclotomic polynomials, Sylvester resultants, and the Fermat equation. For instance, we show that for $m\in\mathbb N$ and distinct primes $p$ and $q$ with $p$ not dividing $m$, the existence of a solution to the Fermat equation $X^p+Y^p=Z^p$ in positive characteristic $q$ such that $X+Y\neq Z$ and $X, Y$ and $Z$ are $m$-th roots of unity implies the factorization of a composite natural number $N$ that is a multiple of $pq$ at the cost of $O\bigl( \phi(m)[m^3 + m^2(\log N)^2] M(\lfloor \log_2 N\rfloor +1) \bigr)$, where $\phi$ is the Euler's function and $M$ is the multiplication time function for $\mathbb Z$. We also show that such solutions do not exist for many semiprime integers $N$, provided that $m$ is required to have a fixed polynomial upper bound in $\log N$.
\end{abstract}



\section{Introduction}\label{S1}

Agrawal, Kayal, and Saxena introduced the first deterministic polynomial-time primality-proving algorithm in 2000, known as the AKS primality test, which does not rely on any hypotheses. The algorithm was published in a scientific journal four years later \cite{AKS}. Specifically, in \cite{AKS}, Theorem 4.1, it was proven that for a sufficiently large integer $n$, if in $\mathbb{Z}[X]$, with $X$ as an indeterminate, the polynomial congruence 
\[
(X+a)^n\equiv X^n+a \: (\text{mod}\: X^r-1, \: n) ,
\]
holds for an appropriately chosen $r\in\mathbb N$ bounded by a polynomial in $\log n$ and all $a\in\mathbb N$ with $a\le\lfloor\sqrt{\phi(r)}\log n\rfloor$, where $\phi:\mathbb N\rightarrow\mathbb N$ the Euler's (totient) function, then $n$ is a prime number.

Suppose $n$ is a composite number. Using the number field sieve, the conjectural cost (running time) to factor $n$ is $e^{\bigl(c+o(1)\bigr)(\log n)^{\frac{1}{3}}\bigl(\log (\log n)\bigr)^{\frac{2}{3}}}$ with $c=\bigl(\frac{64}{9}\bigr)^{\frac{1}{3}}=1.9229994...$ (see \cite{BLP}, Conjecture 11.2; see also \cite{St}, Sections 3 and 11 in which $c=\bigl(\frac{2}{9}\bigr)^{\frac{1}{3}}=1.5262856...$ but only for special composites $n$) and with $c=2\bigl(\frac{5+2\sqrt{6}}{18}\bigr)^{\frac{1}{3}}=1.6385862...$ (see \cite{Copp}, Section 4); this improved the rigorous or conjectural costs of prior algorithms such as a probabilistic one (see \cite{LP}, Theorem), the quadratic sieve algorithm, and the elliptic-curve factorization method (e.g., see \cite{P} and \cite{Coh}, Chapter 10). The best known upper bound on the cost of deterministic integer factorization algorithms is $O\Big(\frac{n^{1/5}\log^{16/5} n}{(\log \log n)^{3/5}}\Big)$ and is due to Harvey and Hittmeir (see \cite{HH}, Theorem 1.1). Their algorithm builds on techniques introduced by Hittmeir in \cite{Hi}, Theorem 1.1 and subsequently developed by Harvay in \cite{Ha}, Theorem 1.1 where the slightly weaker upper bound $O(n^{1/5}\log^{16/5} n)$ is obtained.

A natural question arises: {\it Is it possible to reveal a prime divisor of $n$ using some variation of the AKS algorithm?} Motivated by this question, we consider some conditions on $n$ under which one can find a proper divisor of it in polynomial time. Fortunately, these conditions seem to rarely hold for semi-prime integers, but of course the same ought to be taken into account. So Algorithm 2 we include in Section \ref{S8} on one side is a special-purpose algorithm and on another side, similar to some prior factorization algorithms which benefited from working not only over $\mathbb Z$ but over the rings of integers $O_K$s of number fields $K$ (such as the continued fraction factorization method which is a predecessor to the modern quadratic sieve algorithm; see also the number field sieve), is expected to become a general-purpose algorithm if the role of $\mathbb Z$ in this paper is taken over by a suitable $O_K$. We hope to generalize this paper to the $O_K$s in a future work. 

So currently our algorithm is competitive only in certain cases and in general it is outperformed by the existing integer factorization algorithms used in practice. Our aim is simply to propose a new approach to integer factorization, one that may offer further perspectives for future research.

Semiprimes and the difficulty of their factorizations play a central role in the security of the RSA cryptosystem, which is one of the most widely used public-key encryption schemes in practice. Introduced by Rivest, Shamir, and Adleman in 1977, see \cite{RSA}, RSA relies on the assumption that factoring large semiprimes is computationally infeasible. This assumption underpins the security of digital signatures, secure key exchange, and encryption protocols such as SSL/TLS and PGP. For an in-depth overview of RSA and public-key cryptography, see \cite{PPG}.

To mention one of the condition considered (see Definition \ref{D2}(3)), let $N$ be a natural number with at least two distinct prime divisors. Let $m>0$ be an integer, and suppose there exist two distinct prime divisors $p$ and $q$ of $N$ such that the Fermat equation 
\begin{equation}\label{EQ00}
X^p + Y^p =Z^p
\end{equation} 
has a solution $(X_0, Y_0, Z_0)$ in an algebraic closure $\overline{\mathbb F_q}$ of the finite field $\mathbb{F}_q$, where $X_0$, $Y_0$, and $Z_0$ are $m$-th roots of unity and satisfy $X_0 + Y_0 \neq Z_0$. In this situation, we show that one can compute a proper divisor of $N$ in polynomial time in $m$ (see Theorem \ref{TH3}). Consequently, if such a configuration occurs for some $m$ that is polynomially bounded in $\log N$, then $N$ can be factored in polynomial time. To the best of our knowledge, this is the first instance in which the (Fermat) Equation (\ref{EQ00}), or its simplified version 
\begin{equation}\label{EQ0}
X^p + Y^p =1
\end{equation} 
which is obtained when the triple $(X_0,Y_0,Z_0)$ is replaced by the pair $(\frac{X_0}{Z_0},\frac{Y_0}{Z_0})$, appears in connection with the integer factorization problem. So the role played by the ideals $(X^r-1,n)$ of $\mathbb Z[X]$ in the primality testing problem for $n$ (e.g., see \cite{AKS}) is played in what follows by the ideals $(X^m-1,Y^m-1,X^r+Y^r-a,N)$ of $\mathbb Z[X,Y]$ in factoring $N$, where $r$ is the residue of $p$ modulo $m$ and it is relatively prime to $m$.

Here is a more detailed description of the paper. 

In Section \ref{S2}, we study extensively Equation (\ref{EQ0}) and its generalization
\begin{equation}\label{EQ1.25}
X^p + Y^p =a,
\end{equation}
with $a\in\mathbb Z$, 
in all characteristics via cyclotomic polynomials, Sylvester resultants, and quotient rings of $\mathbb Z[X,Y]$ that depend on two or more entries of the quadruple $(N,m,r,a)$ in $\mathbb N^3\times\mathbb Z$, some of the main results being Lemma \ref{L2} and Theorem \ref{TH1}. Modulo a prime $p$ (i.e, when $N$ is a prime), our working condition that $m$ and $r$ are relatively prime implies that the quotient rings are isomorphic to finite quotients of $\mathbb F_p[X]$ and their dimensions as vector spaces over $\mathbb F_p$ are studied in Section \ref{S3}, the main results being Corollaries \ref{C2} and \ref{C3}. 

Section \ref{S4} proves (see Theorem \ref{TH0}) a congruence modulo products of distinct primes that generalizes the classical Frobenius Identity recalled in Lemma \ref{L1} and applies it (see Propositions \ref{P2} and \ref{P3}) to the quotient rings introduced. 

Section \ref{S5} formulates and studies two main conditions and variations of them that are used to study polynomial-time factorizations of integers (see Definitions \ref{D1} and \ref{D2}). 

The goal of Section \ref{S6} is threefold. First, we review the multiplication time function for $\mathbb Z$ and the modulo $N$ analog of it in order to include a Dichotomic Euclidean Algorithm in one variable modulo $N$ that involves a dichotomy (see the two parts of Proposition \ref{P8}): the output of the algorithm is either the usual greatest common divisor of two polynomials in one variable or is a divisor of $N$. Second, we introduce and study a third related condition (see Definition \ref{D3}). Third, we prove Theorem \ref{TH3}. 

In Section \ref{S7}, we show that Equation (\ref{EQ0}) has no solution among $m$-th roots of unity in $\overline{\mathbb F_q}$, when $6\nmid m$, provided $N\in pq\mathbb Z$ has no common prime factor with the Sylvester resultant of $X^m-1$ and $(1-X)^m-1$ (see Proposition \ref{P10}). A similar result is proved when $6\mid m$; for instance, if $6 \mid m$, then we show that Equation (\ref{EQ0}) has no solution $(X_1, Y_1)$ among $m$-th roots of unity in $\overline{\mathbb F_q}$ such that $X_1+Y_1\neq 1$, for arbitrary $N\in pq\mathbb Z$ that is coprime to the Sylvester resultant of $1+X^6+X^{12}+\cdots+X^{m-6}$ and $(1-X)^m-1$ and satisfies $\gcd(N,7m)=1$ (see Proposition \ref{P10}(5)). As a consequence, for the majority of semiprime numbers $N=pq$, if $m$ is polynomially bounded in $\log N$, then Equation (\ref{EQ0}) together with $X+Y\neq 1$ admits no solution among $m$-th roots of unity in $\overline{\mathbb F_q}$ (see Corollary \ref{C5}).

Section \ref{S8} presents two algorithms that relate to the factorization of $N$. 

{\bf Notations.} Let $(m, r)\in \mathbb{N}^2$ be a pair of natural numbers, and let $(a, b)\in \mathbb{Z}^2$ with $a\leq b$. We use the following notations:

\medskip
\noindent $\llbracket a, b \rrbracket:=\{a, a+1, \ldots, b\}$;

\noindent $\mathbb I_{m,r}:= \llbracket0, m-1\rrbracket\times \llbracket0, r-1\rrbracket$;

\noindent $\mathbb I_m:=\mathbb I_{m,m}$; 

\noindent $\mathbb T:=\mathbb N\setminus\{p^k|p\;\textup{is a prime},\;k\in\mathbb N\cup\{0\}\}$.

\medskip\noindent
So $\mathbb T$ is the set of natural numbers that have more than one prime divisor. 

\section{Sums of roots of unity via quotient rings and resultants}\label{S2}

To study Equation (\ref{EQ1.25}) we introduce several quotients of $\mathbb Z[X,Y]$ as follows.

For a quadruple $(m,r,N,a)\in\mathbb N^3\times\mathbb Z$, we consider the polynomial
$$\Psi_m(X):=X^m-1\in\mathbb Z[X]$$
and the following six ideals of $\mathbb Z[X,Y]$: 
$$I_{m, r;a}:=(X^m-1, Y^m-1, X^r+Y^r-a)=\bigl(\Psi_m(X),\Psi_m(Y),\Psi_r(X)+\Psi_r(Y)+2-a\bigr),$$
$$J_{m, r;a}:=(X^m-1, X^r+Y^r-a)=\bigl(\Psi_m(X),\Psi_r(X)+\Psi_r(Y)+2-a\bigr),$$ 
$$K_m:=(X^m-1,Y^m-1)=\bigl(\Psi_m(X),\Psi_m(Y)\bigr),\quad\quad\quad\quad$$ 
$$\,I_{m,r,N;a}:=(X^m-1, Y^m-1, X^r+Y^r-a,N)=\bigl(\Psi_m(X),\Psi_m(Y),\Psi_r(X)+\Psi_r(Y)+2-a,N\bigr),\quad$$ 
$$J_{m,r,N;a}:=(X^m-1, X^r+Y^r-a,N)=\bigl(\Psi_m(X),\Psi_r(X)+\Psi_r(Y)+2-a,N\bigr),\;\;\,\,$$
and
$$K_{m,N}:=(X^m-1,Y^m-1,N)=\bigl(\Psi_m(X),\Psi_m(Y),N\bigr).\quad\quad\quad\quad\quad$$ 
Between these ideals we have the following inclusions
\begin{equation}\label{EQ1.5}
\xymatrix@R=10pt@C=21pt@L=2pt{
J_{m,r;a} \ar@{}[r]|-*[@]{\subset} \ar@{}[d]|-*[@]{\cap} & I_{m,r;a}\ar@{}[r]|-*[@]{\supset} \ar@{}[d]|-*[@]{\cap} & K_m\ar@{}[d]|-*[@]{\cap}\\
J_{m,r,N;a} \ar@{}[r]|-*[@]{\subset} & I_{m,r,N;a}\ar@{}[r]|-*[@]{\supset} & K_{m,N}.\\}
\end{equation}

The six ideals define six quotient rings 
$$R_{m,r;a}:=\mathbb{Z}[X,Y]/I_{m,r;a},$$ 
$$R_{m,r,N;a}:=\mathbb{Z}[X,Y]/I_{m,r,N;a},$$
$$S_{m,r;a}:=\mathbb{Z}[X,Y]/J_{m,r;a},$$
$$S_{m,r,N;a}:=\mathbb{Z}[X,Y]/J_{m,r,N;a},$$ 
$$T_m:=\mathbb{Z}[X,Y]/K_m,$$
and
$$T_{m,N}:=\mathbb{Z}[X,Y]/K_{m,N}.$$
To Inclusions (\ref{EQ1.5}) correspond ring epimorphisms
\begin{equation*}\label{EQ2}
\xymatrix@R=10pt@C=21pt@L=2pt{
R_{m,r;a} \ar[r]\ar[d] & S_{m,r;a} \ar[d] & T_m\ar[l]\ar[d]\\
R_{m,r,N;a} \ar[r] & S_{m,r,N;a} & T_{m,N}.\ar[l]\\}
\end{equation*}

We have the following direct application of the Division Algorithm for monic polynomials.

\begin{Le}\label{L2}
Let $(m,r)\in\mathbb N^2$. Let $g\in\mathbb{Z}[X]$ be monic of degree $m$. Let $h\in\mathbb Z[X,Y]$ be such that as a polynomial in $Y$ it is monic of degree $r$. Let $Q:=\mathbb Z[X,Y]/(g,h)$. Then the following properties hold.

\medskip
{\bf (1)} As an abelian group $Q$ is free of rank $mr$ having $\{X^iY^j+(g,h)|(i,j)\in\mathbb I_{m,r}\}$ as a $\mathbb Z$-basis. 

\smallskip
{\bf (2)} For each $f\in\mathbb Z[X,Y]$ there exist unique integers $a_{ij}(f)$ indexed by
$(i,j)\in \mathbb I_{m,r}$ such that in $\mathbb Z[X,Y]$ we
have the congruence
\begin{equation}\label{EQ3}
f\equiv \sum_{(i,j)\in\mathbb I_{m,r}} a_{i,j}(f)X^iY^j \pmod {(g,h)}.
\end{equation}

\smallskip
{\bf (3)} Suppose there exists $l\in\mathbb{N}$ such that $f\in l\mathbb Z[X,Y]$. Then we have $a_{i,j}(f)\in l\mathbb Z$ for each $(i,j)\in\mathbb I_{m,r}$.

\smallskip
{\bf (4)} Let $a\in\mathbb Z$. The abelian group $S_{m,r;a}$ is free of rank $mr$ and $\{X^iY^j+J_{m,r;a}|(i,j)\in\mathbb I_{m,r}\}$ is a $\mathbb Z$-basis of it. Also, $S_{m,r,N;a}$ is a free $\mathbb Z/N\mathbb Z$-module of rank $mr$ and $\{X^iY^j+J_{m,r,N;a}|(i,j)\in\mathbb I_{m,r}\}$ is a $\mathbb Z/N\mathbb Z$-basis of it.

\smallskip
{\bf (5)} Let $a\in\mathbb Z$. For each $f\in\mathbb Z[X,Y]$ there exist unique integers $b_{i,j}(f)\in \llbracket0,N-1\rrbracket$ indexed by $(i,j)\in \mathbb I_{m,r}$ such that in $\mathbb Z[X,Y]$ we
have the congruence
\[
f\equiv \sum_{(i,j)\in\mathbb I_{m,r}} b_{i,j}(f)X^iY^j \pmod {J_{m,r,N;a}}.
\]

\smallskip
{\bf (6)} As an abelian group, $T_m$ is free of rank $m^2$ having $\{X^iY^j+K_m|(i,j)\in\mathbb I_m\}$ as a $\mathbb Z$-basis. Also, $T_{m,N}$ is a free $\mathbb Z/N\mathbb Z$-module of rank $m^2$ having $\{X^iY^j+K_{m,N}|(i,j)\in\mathbb I_m\}$ as a $\mathbb Z/N\mathbb Z$-basis.

\smallskip
{\bf (7)} There exist unique integers $c_{i,j}(f)\in \llbracket0,N-1\rrbracket$ indexed by $(i,j)\in \mathbb I_m$ such that in $\mathbb Z[X,Y]$ we
have the congruence
\[
f\equiv \sum_{(i,j)\in\mathbb I_m} c_{i,j}(f)X^iY^j \pmod {K_{m,N}}.
\]
\end{Le}
\begin{proof}
Let $P:=\mathbb Z[X]/(g)$. As $g(X)$ is a monic polynomial in $X$, the Division Algorithm applied to the polynomial ring $\mathbb{Z}[X]$ gives that for each $f_P\in\mathbb Z[X]$ there exists a unique $(m+1)$-tuple 
$$(q,a_0,\ldots,a_{m-1})\in\mathbb{Z}[X]\times \mathbb{Z}^m$$ 
such that $f_P=qg+\sum_{j=0}^{m-1} a_jX^j$. From this it follows that $P$, as an abelian group, is free of rank $m$ having $\{X^i+(g)|i\in \llbracket0,m-1\rrbracket\}$ as a $\mathbb Z$-basis. 

Let $\bar h$ be the image of $h$ in $P[Y]$. We have a natural identification $Q=P[Y]/(\bar h)$ by the Third Isomorphism Theorem. As $\bar h$ is a monic polynomial in $Y$, a similar argument based on the Division Algorithm applied to the polynomial ring $P[Y]$ gives that $P[Y]/(\bar h)$ as a $P$-module is free of rank $r$ having $\{Y^j+(\bar h)|j\in \llbracket0,r-1\rrbracket\}$ as a $P$-basis and hence as a $\mathbb Z$-module is free of rank $mr$ having $\{[X^i+(g)]Y^j+(\bar h)|(i,j)j\in\mathbb I_{m,r}\}$ as a $\mathbb Z$-basis. Thus, via the identification $Q=P[Y]/(\bar h)$, it follows that part (1) holds. 

Part (2) holds as it is only a reformulation of part (1).

For part (3), let $f_1\in\mathbb{Z}[X,Y]$ be such that $f=lf_1$. From the uniqueness in part (2) we get that $a_{i,j}(f)=la_{i,j}(f_1)\in l\mathbb Z$ for each $(i,j)\in\mathbb I_{m,r}$. So part (3) holds.

For parts (4) to (7) we take $g(X):=X^m-1$. 

For parts (4) and (5) we take $h(X,Y):=X^r+Y^r-a$; thus $(g,h)=J_{m,r;a}$ and $Q=S_{m,r;a}$. The first sentence of part (4) follows from part (1). The second sentence of part (4) follows from the mentioned first sentence by reduction modulo $N$. So part (4) holds.

Part (5) follows by remarking that $b_{i,j}(f)$ is the only integer in the set $\llbracket0, N-1\rrbracket$ such that we have $a_{ij}(f)\equiv b_{ij}(f) \pmod {N}$.

Parts (6) and (7) are proved similarly to parts (4) and (5) with $h(X,Y):=Y^m-1$. 
\end{proof}

\begin{Rem}\normalfont\label{R1}
{\bf (1)} The additive structures of the $\mathbb Z$-algebras $R_{m,r;a}$ 
and $R_{m,r,N;a}$ are more complicated. For instance, for each $l\in\mathbb N$ we have $R_{m,ml;1}=\{0\}$ and hence $R_{m,ml,N;1}=\{0\}$. Similarly, 
$$R_{m,1;a}=\mathbb Z[X,Y]/(X^m-1,Y^m-1,X+Y-a)\cong \mathbb Z[X]/\bigl(X^m-1,(a-X)^m-1\bigr)$$
and hence $R_{m,1,N;a}\cong (\mathbb Z/N\mathbb Z)[X]/\bigl(X^m-1,(a-X)^m-1\bigr)$.

\smallskip
{\bf (2)} Assume $r>m$ is not divisible by $m$. If $s\in \llbracket1, m-1\rrbracket$ is the residue of $r$ modulo $m$, then $I_{m,r;a}=I_{m,s;a}$ and hence $I_{m,r,N;a}=I_{m,s,N;a}$. So for the study of the $S_{m,r;a}$-algebra $R_{m,r;a}$ and of the $S_{m,r,N;a}$-algebra $R_{m,r,N;a}$ we can assume that $r\in\llbracket1, m-1\rrbracket$.

\smallskip
{\bf (3)} The homomorphisms from $\mathbb Z$ to $S_{m,r;a}$ or $S_{m,r,N;a}$ or $R_{m,r;a}$ or $R_{m,r,N;a}$ or $T_m$ or $T_{m,N}$ are integral as at the level of spectra are finite (the targets of the homomorphisms being finitely generated abelian groups by Lemma \ref{L2}(4) and (6)).
\end{Rem}

We consider the algebraic integer 
$$\zeta_m:=e^{\frac{2\pi i}{m}}=\cos\bigl(\frac{2\pi}{m}\bigr)+i\sin\bigl(\frac{2\pi}{m}\bigr)\in\mathbb C$$ 
and the $\mathbb Z$-subalgebra $\mathcal Z_m:=\mathbb Z[\zeta_m]$ of $\mathbb C$ generated by it. Let $\Phi_m\in\mathbb Z[X]$ be the $m$-th cyclotomic polynomial over $\mathbb Z$. Recall that $\Phi_m(X)$ is the minimal polynomial of $\zeta_m$ over $\mathbb Q$; so it is the unique monic irreducible polynomial in $\mathbb Z[X]$ with $\Phi_m(\zeta_m)=0$ and we have a $\mathbb Z$-algebra isomorphism
$$\mathcal Z_m\cong\mathbb Z[X]/\bigl(\Phi_m(X)\bigr)$$
that maps $\zeta_m$ to $X+\bigl(\Phi_m(X)\bigr)$.

Let $\Div(m)$ be the set of positive integers that divide $m$. In $\mathcal Z_m[X]$ we have identities
\begin{equation}\label{EQ4}
\Psi_m(X)=\prod_{i=1}^{m} (X-\zeta_m^i)=\prod_{d\in\Div(m)} \Phi_d(X)\;\;\textup{and}\;\;\Phi_m(X)=\prod_{i\in \llbracket1, m\rrbracket,\gcd(i,m)=1} (X-\zeta_m^i).
\end{equation}

As $\mathbb Z[X]$ is a unique factorization domain whose group of units is $\{-1,1\}$, in what follows the greatest common divisor of two nonzero polynomials $f$ and $g$ in $\mathbb Z[X]$ is assumed to have a positive leading coefficient and is denoted by $\gcd(f,g)$. Hence, if $f$ or $g$ is monic, then $\gcd(f,g)$ is monic and it is the unique monic polynomial of $\mathbb Z[X]$ such that we have an identity of ideals $(f,g)=\bigl(\gcd(f,g)\bigl)$ in $\mathbb Q[X]$. 

For $(c,d)\in\mathbb N^2$ and $a\in\mathbb Z$ we consider the monic polynomials
$$\Upsilon_{c,d;a}(X):=\gcd\bigl(\Phi_c(X),\Phi_d(a-X)\bigr)\in \mathbb Z[X]$$
and
$$\Upsilon_{c;a}(X):=\gcd\bigl(\Psi_c(X),\Psi_c(a-X)\bigr)\in\mathbb Z[X].$$

\begin{Le}\label{L2.5}
Let $(m,a)\in\mathbb N\times \mathbb Z$ and $(c,d)\in\mathbb N^2$. Then the following properties hold.

\medskip
{\bf (1)} If $a\notin\{-2,-1,0,1,2\}$, then $\Upsilon_{c,d;a}(X)=1$.

\smallskip
{\bf (2)} We have $\Upsilon_{1,1;2}(X)=X-1$. Moreover, $\Upsilon_{c,d;2}(X)=1$ if $(c,d)\neq (1,1)$.

\smallskip
{\bf (3)} We have $\Upsilon_{2,2;-2}(X)=X+1$. Moreover, $\Upsilon_{c,d;-2}(X)=1$ if $(c,d)\neq (2,2)$. 

\smallskip
{\bf (4)} Assume that $a=0$. If $c$ is odd (resp.\ even and $\frac{c}{2}$ is odd), then $\Upsilon_{c,d;0}(X)\neq 1$ iff $d=2c$ (resp.\ $d=\frac{c}{2}$), in which case $\Upsilon_{c,d;0}(X)=\Phi_c(X)$ (resp.\ $\Upsilon_{c,d;0}(X)=\Phi_{\frac{c}{2}}(X)$). If $4\mid c$, then $\Upsilon_{c,d;0}(X)\neq 1$ iff $d=c$, in which case $\Upsilon_{c,c;0}(X)=\Phi_c(X)$.

\smallskip
{\bf (5)} We have $\Upsilon_{6,6;1}(X)=X^2-X+1$. If $(c,d)\neq (6,6)$, then $\Upsilon_{c,d;1}(X)=1$.

\smallskip
{\bf (6)} We have $\Upsilon_{3,3; -1}(X)=X^2+X+1$. If $(c,d)\neq (3,3)$, then $\Upsilon_{c,d;-1}(X)=1$.

\smallskip
{\bf (7)} 
If $6\nmid m$ (resp.\ $6\mid m$), then $\Upsilon_{m;1}(X)=1$ (resp.\ $\Upsilon_{m;1}(X)=X^2-X+1$).

\smallskip
{\bf (8)} 
If $3\nmid m$ (resp.\ $3\mid m$), then $\Upsilon_{m;-1}(X)=1$ (resp.\ $\Upsilon_{m;-1}(X)=X^2+X+1$).

\smallskip
{\bf (9)} We have $\Upsilon_{m;2}(X)=X-1$. 

\smallskip
{\bf (10)} If $m$ is even (resp.\ odd), then $\Upsilon_{m;-2}(X)=X+1$ (resp.\ $\Upsilon_{m;-2}(X)=1$).

\smallskip
{\bf (11)} If $m$ is even (resp.\ odd), then $\Upsilon_{m;0}(X)=\Psi_m(X)$ (resp.\ $\Upsilon_{m;0}(X)=1$).
\end{Le}

\begin{proof}
Part (1) holds as for $|a|>2$, the polynomials $X^c-1$ and $(a-X)^d-1$ have no common complex root.
	
For part (2), $X=1$ is the unique common complex root of the polynomials $X^c-1$ and $(2-X)^d-1$. As $\Phi_c(1)=0$ iff $c=1$, part (2) holds.
	
For part (3), the polynomials $X^c-1$ and $(-2-X)^d-1$ have no common complex root when $2\nmid cd$, and if $c$ and $d$ are even, then $X=-1$ is the unique common complex root of $X^c-1$ and $(-2-X)^d-1$. As $\Phi_c(-1)=0$ iff $c=2$, part (3) holds.
	
For part (4), the complex roots of $\Phi_c(X)$ are simple and are precisely the roots of unity $\zeta$ of order $c$. If $c$ is odd, then $\text{ord}(-\zeta)=2c$ and $\Phi_d(-\zeta)=0$ iff $d=2c$. If $c$ is even and $\frac{c}{2}$ is odd, then $\text{ord}(-\zeta)=\frac{c}{2}$ and $\Phi_d(-\zeta)=0$ iff $d=\frac{c}{2}$. If $4\mid c$, then $\text{ord}(-\zeta)=c$ and $\Phi_d(-\zeta)=0$ iff $d=c$. So part (4) holds.

For part (5), the roots of $\Upsilon_{c,d;1}$ in an algebraic closure of $\mathbb Q$ are simple and are precisely the roots of unity $\zeta$ of order $c$ with the property that $1-\zeta$ is a root of unity of order $d$. But the only roots of unity in $\mathbb C$ such that $1$ minus them are also roots of unity in $\mathbb C$ are $\zeta_6=\cos(\frac{\pi}{3})+i\sin(\frac{\pi}{3})$ and $\overline{\zeta_6}=\zeta_6^5$, both of them being of order $6$. This implies that $\Upsilon_{c,d;1}(X)\neq 1$ iff $c=d=6$ and $\Upsilon_{6,6;1}(X)=(X-\zeta_6)(X-\zeta_6^5)=X^2-X+1$. So part (5) holds.

For part (6), the roots of $\Upsilon_{c,d;-1}$ in an algebraic closure of $\mathbb Q$ are simple and are precisely the roots of unity $\zeta$ of order $c$ with the property that $-1-\zeta$ is a root of unity of order $d$. But the only roots of unity in $\mathbb C$ such that $-1$ minus them are also roots of unity in $\mathbb C$ are $\zeta_6^2$ and $\overline{\zeta_6}^4$, both of them being of order $3$. This implies that $\Upsilon_{c,d;-1}(X)\neq 1$ iff $c=d=3$ and $\Upsilon_{3,3;-1}(X)=(X-\zeta_6^2)(X-\zeta_6^4)=X^2+X+1$. So part (6) holds.

For part (7), note that Equations (\ref{EQ4}) gives that for each $(c,d)\in\Div(m)$ we have divisibilities 
$$\Upsilon_{c,d;1}\mid\Upsilon_{m;1}\mid\prod_{(c,d)\in\Div(m)^2} \Upsilon_{c,d;1}.$$ 
From this and part (5) we get that $\Upsilon_{m;1}\in\{1,\Upsilon_{6,6;1}\}$ is such that $\Upsilon_{m;1}=\Upsilon_{6,6;1}$ iff $6\mid m$. So part (7) holds.

Part (8) is (resp.\ (9) and (10) are) proved similarly to part (7), with the reference to part (5) being replaced by a reference to part (6) (resp.\ parts (2) and (3)).

Part (11) follows from the identity $\Psi_m(X)=\Psi_m(-X)$ (resp.\ $\Psi_m(X)+\Psi_m(-X)=-2$).
\end{proof}

Considering the direct sum decomposition $S_{m,1;a}=\oplus_{i=0}^{m-1} \mathbb Z(X^i+J_{m,1;a})$ of free abelian groups of finite rank by Lemma \ref{L2}(4), we have an identity 
$$\Psi_m(a-X)+J_{m,1;a}=\sum_{i=0}^{m-1} c_{m,i;a}(X^i+J_{m,1;a}),$$ where $c_{m,i;a}:=(-1)^i\binom{m}{i}a^{m-i}\in\mathbb Z$ if $i\in \llbracket1, m-1\rrbracket$ and $c_{m,0;a}:=a^m+(-1)^m-1\in\mathbb Z$.

As $c_{m,i;1}=(-1)^i\binom{m}{i}$ if $i\in \llbracket1, m-1\rrbracket$ and $c_{m,0;1}=(-1)^m$, we have an identity
\begin{equation*}
2^{m-1}-1=(-1)^mc_{m,0;1}+\sum_{i=1}^{m-1} (-1)^ic_{m,i;1}
\end{equation*}
and $\max\{|c_{m,i;1}||i\in \llbracket0, m-1\rrbracket\}=\binom{m}{\lfloor\frac{m}{2}\rfloor}$.

We consider the additive endomorphism $L_{m;a}:S_{m,1;a}\rightarrow S_{m,1;a}$ which is the multiplication by $\Psi_m(a-X)+J_{m,1;a}$. With respect to the basis 
$\{X^i+J_{m,1;a}|i\in \llbracket0, m-1\rrbracket\}$ of $S_{m,1;a}$, the matrix representation of $L_{m;a}$ is
$$\mathbb M_{m;a}:=\begin{bmatrix}
c_{m,0;a} & c_{m,m-1;a} & c_{m,m-2;a} & ... & c_{m,2;a} & c_{m,1;a}\\
c_{m,1;a} & c_{m,0;a} & c_{m,m-1;a} & ... & c_{m,3;a} & c_{m,2;a}\\
c_{m,2;a} & c_{m,1;a} & c_{m,0;a} & ... & c_{m,4;a} & c_{m,3;a}\\
...& ... & ... & ... & ... & ... \\
c_{m,m-2;a} & c_{m,m-3;a} & c_{m,m-4;a} & ... & c_{m,0;a} & c_{m,m-1;a}\\
c_{m,m-1;a} & c_{m,m-2;a} & c_{m,m-3;a} & ... & c_{m,1;a} & c_{m,0;a}\\
\end{bmatrix}.$$
As the cokernel $\Coker(L_{m;a})$ of $L_{m;a}$ is $\mathbb Z[X]/\bigl(\Psi_m(X),\Psi_m(a-X)\bigr)$, from Remark \ref{R1}(1) we get that we have an isomorphism of abelian groups 
\begin{equation}\label{EQ6}
R_{m,1;a}\cong\Coker(L_{m;a}).
\end{equation}
Let 
$$\Delta_{m;a}:=-\det(\mathbb M_{m;a})\in\mathbb Z.$$

For $\bigr(f(X),g(X)\bigr)\in\mathbb Z[X]^2$, let $\Res\bigl(f(X),g(X)\bigr)\in\mathbb Z$ be the Sylvester resultant of $f(X)$ and $g(X)$. For $(c,d)\in\mathbb N^2$, we single out the following resultants
$$\Delta_{c,d;a}:=\Res\bigl(\Phi_c(X),\Phi_d(a-X)\bigr)\in\mathbb Z.$$ 

The following proposition first recalls properties of circular determinants in the context of the circular matrices $\mathbb M_{m;a}$ and second provides identities and estimates that pertain to the $\Delta_{m;a}$s and the $\Delta_{c,d;a}$ based on properties of Sylvester resultants.

\begin{Pro}\label{L3}
Let $m\in\mathbb N$. Then the following properties hold.

\medskip
{\bf (1)} The set of eigenvalues of $\mathbb M_{m;a}$ in $\mathbb C$ is $\{\Psi_m(a-\zeta_m^i)|i\in \llbracket1, m\rrbracket\}$. 

\smallskip
{\bf (2)} Over $\mathcal Z_m[\frac{1}{m}]$, the matrix $\mathbb M_{m;a}$ is conjugate to the diagonal matrix $\mathbb D_{m;a}$ whose entry on the $(i,i)$ position is $\Psi_m(a-\zeta_m^i)=\sum_{j=0}^{m-1} c_{m,j;a}\zeta_m^{ij}$.

\smallskip
{\bf (3)} We have $-\Delta_{m;a}=\prod_{i=1}^m \Psi_m(a-\zeta_m^i)=\Res\bigl(\Psi_m(X),\Psi_m(a-X)\bigr)$.

\smallskip
{\bf (4)} We have identities
$$-\Delta_{m;a}=\prod_{(c,d)\in\Div(m)^2} \Delta_{c,d;a}=\Big(\prod_{c\in\Div(m)} \Delta_{c,c;a}\Big)\Big(\prod_{(c,d)\in\Div(m)^2,\,c<d} \Delta_{c,d;a}^2\Big).$$

{\bf (5)} For $(c,d)\in\mathbb N^2$ we have $\Delta_{c,d;1}=0$ iff $c=d=6$. In particular, $\Delta_{m;1}=0$ iff $6\mid m$.

\smallskip
{\bf (6)} For $(c,d)\in\mathbb N^2$ we have $\Delta_{c,d;-1}=0$ iff $c=d=3$. In particular, $\Delta_{m;-1}=0$ iff $3\mid m$.

\smallskip
{\bf (7)} We have $\Delta_{1,1;1}=-1$ and for $c\in\mathbb N\setminus\{1,6\}$ we have $\Delta_{c,c;1}> 0$. In particular, $\Delta_{m;1}>0$ if $6\nmid m$.\footnote{The minus sign normalization in the definition of $\Delta_{m;a}$ is made in order to get this positivity property.}

\smallskip
{\bf (8)} For $(c,d)\in\mathbb N^2$ we have an inequality $|\Delta_{c,d;a}|\le (|a|+2)^{\phi(c)\phi(d)}$.

\smallskip
{\bf (9)} For $(c,d)\in\mathbb N^2$ and $|a|\ge 3$, we have an inequality $|\Delta_{c,d;a}|\ge (|a|-2)^{\phi(c)\phi(d)}$.
\end{Pro}

\begin{proof}
For $i\in \llbracket1, m\rrbracket$, we consider the column vector $v_i:=\begin{bmatrix}1\\
\zeta_m^{-i}\\
\zeta_m^{-2i}\\
...\\
\zeta_m^{-i(m-1)}\end{bmatrix}$. One computes 
$$\mathbb M_{m;a}v_i=\begin{bmatrix}\sum_{j=0}^{m-1} c_{m,j;a}\zeta_m^{-i(-j)}\\
\sum_{j=0}^{m-1} c_{m,j;a}\zeta_m^{-i(1-j)}\\
...\\
\sum_{j=0}^{m-1} c_{m,j;a}\zeta_m^{-i(m-1-j)}\end{bmatrix}=\Big(\sum_{j=0}^{m-1} c_{m,j;a}\zeta_m^{ij}\Big)v_i=\Psi_m(a-\zeta_m^i)v_i.$$ The determinant of the $m\times m$ matrix $\mathbb V_m:=\begin{bmatrix}v_1 & v_2 & \cdots & v_m\end{bmatrix}$ with entries in $\mathcal Z_m$ is the Vandermonde determinant $W(\zeta_m^{-1},\ldots,\zeta_m^{-m})$ which is nonzero in $\mathcal Z_m$ and thus also in $\mathbb C$. So $\{v_1,\ldots,v_m\}$ is a basis of the complex vector space $\mathbb C^m$ from which part (1) follows.

As $W(\zeta_m^{-1},\ldots,\zeta_m^{-m})$ is nonzero modulo each maximal ideal of $\mathcal Z_m[\frac{1}{m}]$, it is a unit of $\mathcal Z_m[\frac{1}{m}]$. Thus $\{v_1,\ldots,v_m\}$ is a basis of the $\mathcal Z_m[\frac{1}{m}]$-module $\mathcal Z_m[\frac{1}{m}]$ from which part (2) follows. 

From part (2) we get that 
$$-\Delta_{m;a}=\prod_{i=1}^m \Psi_m(a-\zeta_m^i)=\Res\bigl(\Psi_m(X),\Psi_m(a-X)\bigr),$$ 
where the last identity is a standard properties of resultants. So part (3) holds. 

We include a second proof of part (3) based on the (usual) definition of $\Res\bigl(\Psi_m(X),\Psi_m(a-X)\bigr)$ as the determinant of the $2m\times 2m$ matrix defined in $m\times m$ blocks by
$$\mathbb S_{m;a}:=\begin{bmatrix}
1_m & -1_m\\
\mathbb E_{m;a} & \mathbb M_{m;a}-\mathbb E_{m;a}\\
\end{bmatrix},$$
where $1_m$ is the identity matrix of size $m\times m$ and 
$$\mathbb E_{m;a}:=\begin{bmatrix}
(-1)^m & c_{m,m-1;a} & c_{m,m-2;a} & ... & c_{m,3;a} & c_{m,2;a} & c_{m,1;a}\\
0 & (-1)^m & c_{m,m-1;a} & ... & c_{m,4;a} & c_{m,3;a} & c_{m,2;a}\\
0 & 0 & (-1)^m & .... & c_{m,5;a} & c_{m,4;a} & c_{m,3;a}\\
...& ... & ... & ... & ... \\
0 & 0 & 0 & ... & 0 & (-1)^m & c_{m,m-1;a}\\
0 & 0 & 0 & ... & 0 & 0 & (-1)^m\\
\end{bmatrix}.$$
As we have a matrix identity $\mathbb S_{m;a}\begin{bmatrix}
1_m & 1_m\\
0_m & 1_m\\
\end{bmatrix}=\begin{bmatrix}
1_m & 0_m\\
\mathbb E_{m;a} & \mathbb M_{m;a}\
\end{bmatrix}$, where $0_m$ is the zero matrix of size $m\times m$, by taking determinants of this identity it follows that part (3) holds.

Part (4) follows from part (3) and standard properties of resultants of products of polynomials in $\mathbb Z[X]$ applied in the context of the following product decompositions $\Psi_m(X)=\prod_{c\in\Div(m)} \Phi_c(X)$ and $\Psi_m(a-X)=\prod_{d\in\Div(m)} \Phi_d(a-X)$ (cf.\ Equations (\ref{EQ4})).

Parts (5) (resp.\ (6)) follows from Lemma \ref{L2.5}(5) (resp.\ \ref{L2.5}(6)).

For part (7), as $\Delta_{1,1;1}=-1$ and $\Delta_{2,2;1}=\Phi_2\bigl(1-(-1)\bigr)=3$, it suffices to show that for each $c\in\mathbb N\setminus\{1,2,6\}$ we have $\Delta_{c,c;1}>0$. Recall that $\phi(c)$ is even, so there exists $s\in\mathbb N$ such that $\phi(c)=2s$ and the $\phi(c)$ primitive $c$-th roots of unity can be listed as $\{\eta_1,\eta_1^{-1},\ldots,\eta_s,\eta_s^{-1}\}$. So we compute
$$\Delta_{c,c;1}=\prod_{i=1}^{s} [(1-\eta_i)^m-1][(1-\eta_i^{-1})^m-1]=\prod_{i=1}^{s} |(1-\eta_i)^m-1|^2.$$
Thus $\Delta_{c,c;1} \geq 0$. As $\Delta_{c,c;1}\neq 0$ by part (5), we get that $\Delta_{c,c;1}>0$. Thus, if $6\nmid m$, then $\prod_{c\in\Div(m)} \Delta_{c,c;1}>0$. Also, $\prod_{(c,d)\in\Div(m)^2,\,c<d} \Delta_{c,d;a}^2>0$ by part (5). From the last two sentences and part (4) we get that, if $6\nmid m$, then $\Delta_{m;1}>0$. So part (7) holds.

Part (8) holds as Equations (\ref{EQ4}) and repeated applications of the triangle inequalities give
	$$
	\Delta_{c,d;a} = |\underset{(i,j)\in\mathbb I_{c,d},\gcd(i,c)=\gcd(j,d)=1}{\prod} \bigl(a-\zeta_c^i-\zeta_d^j\bigr)|\leq (|a|+2)^{\phi(c)\phi(d)}.$$
	
Part (9) is proved similarly to part (8) based on the inequality $|a-\zeta_c^i-\zeta_d^j|\ge |a|-2$ for each $(i,j)\in\mathbb I_m$. 	
\end{proof}

For a finite set $A$, let $|A|\in\mathbb N\cup\{0\}$ be the number of elements of $A$.

\begin{theorem}\label{TH1} Let $a\in\mathbb Z$ and $(m,r)\in\mathbb N^2$ with $m\ge r$ and $\gcd(m,r)=1$. Then the following properties hold.

\medskip
{\bf (1)} We have an isomorphism $R_{m,r;a}\cong R_{m,1;a}$ of nonzero rings. 

\smallskip
{\bf (2)} If $a$ is odd and $p$ is a prime divisor of $2^m-a^m$, then $R_{m,r,p;a}\neq 0$, i.e., $I_{m,r,p;a}\neq \mathbb Z[X,Y]$.

\smallskip
{\bf (3)} The ring $R_{m,r;a}$ is finite iff $\Upsilon_{m;a}(X)=1$ iff $\Delta_{m;a}\neq 0$ and iff one of the following disjoint conditions holds.

\medskip\noindent
{\bf (3.a)} We have $|a|\ge 3$.

\smallskip\noindent
{\bf (3.b)} We have $a=1$ and $6\nmid m$.

\smallskip\noindent
{\bf (3.c)} We have $a=-1$ and $3\nmid m$.

\smallskip\noindent
{\bf (3.d)} We have $a=0$ and $m$ is odd.

\smallskip\noindent
{\bf (3.e)} We have $a=-2$ and $m$ is odd.

\smallskip
{\bf (4)} If $6$ (resp.\ $3$) divides $m$, then $R_{m,r;1}$ (resp.\ $R_{m,r;-1}$) is a finitely generated abelian group of rank $2$. Also, $R_{m,r;2}$ has rank $1$ and, if $m$ is even, $R_{m,r;-2}$ has rank $1$ and $R_{m,r;0}$ has rank $m$.

\smallskip
{\bf (5)} Asume $R_{m,r;a}$ is finite. Then $|R_{m,r;a}|=|\Delta_{m;a}|$. Moreover, for a prime $p$ we have $p\mid\Delta_{m;a}$ iff $R_{m,r,p;a}\neq 0$.

\smallskip
{\bf (6)} If $m\in 2\mathbb N$ (resp.\ $m\in\mathbb N$), then $(a+1)^m-1$ (resp.\ $(a-1)^m-1$) is an eigenvalue of $\mathbb M_{m;1}$.

\smallskip
{\bf (7)} Let $\Lambda_{m;a}(X):=\sum_{i=0}^{m-1} X^i(a-X)^{m-1-i}\in\mathbb Z[X]$. Then the following properties hold.

\medskip\noindent
{\bf (7.a)} We have identities $(a-X)^m-X^m=(a-2X)\Lambda_{m;a}(X)$, $\Lambda_{m;a}(a)=a^{m-1}$, and 
$$\Delta_{m;a}=(2^m-a^m)\Res\bigl(X^m-1,\Lambda_{m;a}(X)\bigl)=(2^m-a^m)\prod_{i=1}^m \Lambda_{m;a}(\zeta_m^i).$$

\noindent
{\bf (7.b)} If $a\neq 0$ we have an inequality $|\Lambda_{m;a}(\zeta_m^i)|\le \frac{(1+|a|)^m-1}{|a|}$ for each $i\in \llbracket1, m-1\rrbracket$.

\smallskip\noindent
{\bf (7.c)} We have a divisibility $(a^m-2^m) \mid \Delta_{m;a}$ and an inequality $|\Lambda_{m;1}(\zeta_m^i)|\le 2^m-1$ for each $i\in \llbracket1, m-1\rrbracket$.

\smallskip
{\bf (8)} Assume $6\nmid m$. Let $l:=\lfloor\frac{m}{4}\rfloor\in\mathbb N\cup\{0\}$. Let $\epsilon_m$ be $-1$ if $m=4l$, be $0$ if $m=4l+1$ or $m=4l+2$, and be $1$ if $m=4l+3$. Then, if $m$ is even we have inequalities
$$(2^m-1)(\sqrt{2}^m-1)^{2l}<\Delta_{m;1}\le (2^m-1)(\sqrt{2}^m+1)^{2l+2\epsilon_m}(2^m+1)^{2l}$$ 
and if $m$ is odd we have inequalities
$$(\sqrt{2}^m-1)^{2l}<\Delta_{m;1}< (\sqrt{2}^m+1)^{2l+2\epsilon_m}(2^m+1)^{2l}.$$
\end{theorem}

\begin{proof}
Let $r'\in \llbracket1, m-1\rrbracket$ be such that $m\mid rr'-1$. The rule $(X,Y)\rightarrow (X^r,Y^r)$ defines an endomorphism $e_r:\mathbb Z[X,Y]\rightarrow\mathbb Z[X,Y]$ such that $e_r(K_m)\subset K_m$; so $e_r$ induces an endomorphism 
$$e_{r,m}:T_m\rightarrow T_m.$$ 
As $X^m-1$ divides $X^{rr'}-X$ and $Y^m-1$ divides $Y^{rr'}-Y$, $e_{r,m}\circ e_{r',m}=e_{r',m}\circ e_{r,m}$ is the identity automorphism of $T_m$. Hence $e_{r,m}$ and $e_{r',m}$ are automorphisms inverse to each other. As $e_{r,m}(X+Y-a+K_m)=X^r+Y^r-a+K_m$, $e_{r,m}$ induces an isomorphism $R_{m,1;a}\rightarrow R_{m,r;a}$. 

If $a$ is even, then the kernel of the epimorphism $\mathbb Z[X,Y]\rightarrow\mathbb F_2$ that maps $X$ and $Y$ to $1$ contains $I_{m,1;a}$ and hence $R_{m,1;a}\neq 0$.

If $a$ and $p$ are as in part (2), then the kernel of the epimorphism $\mathbb Z[X,Y]\rightarrow\mathbb F_p$ that maps $X$ and $Y$ to $a$ times the inverse of $2$ in $\mathbb F_p$ contains $I_{m,1,p;a}$ and in particular it induces an epimorphism $R_{m,1,p;a}\rightarrow\mathbb F_p$. So $R_{m,1,p;a}\neq 0$.

Parts (1) and (2) follow from the last two paragraphs.

For parts (3) and (4), from Remark \ref{R1}(1) we get that $R_{m,1;a}\otimes_{\mathbb Z} \mathbb Q=\mathbb Q[X]/(\Upsilon_{m;a})$. As the ring $R_{m,1;a}$ is finite iff $R_{m,1;a}\otimes_{\mathbb Z} \mathbb Q=0$ by Remark \ref{R1}(3), it follows that $R_{m,1;a}$ is finite iff $\Upsilon_{m;a}(X)=1$. Isomorphism (\ref{EQ6}) implies that $R_{m,1;a}$ is finite iff $\Delta_{m;a}\neq 0$. So to end the proof of part (3) it suffices to show that $\Upsilon_{m;a}(X)=1$ iff one of the conditions (3.a) to (3.e) holds. But this follows from Lemma \ref{L2.5}(1) and (7) to (11) and the fact that $\Upsilon_{m;a}(X)=1$ iff $\Upsilon_{c,d;a}(X)=1$ for each pair $(c,d)\in\Div(m)^2$. So part (3) holds.

For part (4) we can assume $r=1$ by part (1). For $r=1$, part (4) follows from the identity $R_{m,1;a}\otimes_{\mathbb Z} \mathbb Q=\mathbb Q[X]/(\Upsilon_{m;a})$ and the fact that $\Upsilon_{m;1}$ (resp.\ $\Upsilon_{m;-1}$) has degree $2$ by Lemma \ref{L2.5}(7) (resp.\ \ref{L2.5}(8)), $\Upsilon_{m;2}$ has degree $1$ by Lemma \ref{L2.5}(9), and, if $m$ is even, $\Upsilon_{m;-2}$ has degree $1$ by Lemma \ref{L2.5}(10) and $\Upsilon_{m;0}$ has degree $m$ by Lemma \ref{L2.5}(11).

For part (5), based on part (3) and Isomorphism (\ref{EQ6}) we have $|R_{m,r;a}|= |\Delta_{m;a}|$. Thus, for a prime $p$ we have $p\mid\Delta_{m;a}$ iff $R_{m,r,p; a}\neq 0$. So part (5) holds.

Part (6) holds as we have $\zeta_m^{\frac{m}{2}}=-1$ and $\Psi_m\bigl(a-(-1)\bigr)=(a+1)^m-1$ (resp.\ $\zeta_m^m=1$ and $\Psi_m(a-1)=(a-1)^m-1$) is an eigenvalue of $\mathbb M_{m;a}$ by Proposition \ref{L3}(1).

For part (7.a), the first identity follows from the formula for the difference of two $m$-th powers. We have $\Lambda_{m;a}(a)=\sum_{i=0}^{m-1} a^i(a-a)^{m-1-i}=a^{m-1}$. To check the third identity of part (7.a), we note that for each $i\in \llbracket0, m-1\rrbracket$ we have 
$$(a-\zeta_m^i)^m-1=(a-\zeta_m^i)^m-\zeta_m^{mi}+\zeta_m^{mi}-1=(a-2\zeta_m^i)\Lambda_{m;a}(\zeta_m^i).$$
From this and Proposition \ref{L3}(3) we get that
$$-\Delta_{m;a}=-\prod_{i=0}^{m-1} [(a-\zeta_m^i)^m-1]=-\prod_{i=0}^{m-1} (a-2\zeta_m^i)\Lambda_{m;a}(\zeta_m^i)=
-2^m\Res\bigl(x^m-1,\Lambda_{m;a}(X)\bigr)\prod_{i=0}^{m-1} (\frac{a}{2}-\zeta_m^i)$$
$$=-2^m\Res\bigl(x^m-1,\Lambda_{m;a}(X)\bigr)\left(\frac{a^m}{2^m}-1\right)=-(a^m-2^m)\Res\bigl(x^m-1,\Lambda_{m;a}(X)\bigr).$$

For part (7.b), if $i\in \llbracket1, m-1\rrbracket$, then repeated applications of the triangle inequality give that $|\Lambda_{m;a}(\zeta_m^i)|\le \sum_{j=0}^{m-1} (1+|a|)^j=\frac{(1+|a|)^m-1}{|a|}$. So part (7.b) holds.

Part (7.c) follows from parts (7.a) and (7.b).

For part (8), let $s:=\lfloor\frac{m}{2}\rfloor$. Based on Equations (\ref{EQ4}) and Proposition \ref{L3}(3) for $m$ even we have
\begin{equation}\label{EQ7}
\Delta_{m;1}=-\prod_{i=0}^{m-1} [(1-\zeta_m^i)^m-1]=(2^m-1)\prod_{i=1}^{s-1} |(1-\zeta_m^i)^m-1|^2
\end{equation}
and for $m$ odd we have
\begin{equation}\label{EQ8}
\Delta_{m;1}=-\prod_{i=0}^{m-1} [(1-\zeta_m^i)^m-1]=\prod_{i=1}^{s} |(1-\zeta_m^i)^m-1|^2.
\end{equation}
We only prove part (8) for $m=4l+3$ as the other cases are entirely similar. We have $l=2s+1$. For $i\in \llbracket1, l+1\rrbracket$ we have $|1-\zeta_m^i|< \sqrt{2}$ and hence $|(1-\zeta_m^i)^m-1|<\sqrt{2}^m+1$; as the algebraic integer $(1-\zeta_m^i)^m$ is a purely imaginary complex number, we also have $1\le|(1-\zeta_m^i)^m-1|$. For $i\in \llbracket l+2,2l+1\rrbracket$ we have $\sqrt{2}<|1-\zeta_m^i|<2$ and hence $\sqrt{2}^m-1< |(1-\zeta_m^i)^m-1|<2^m+1$. From these inequalities and Equation (\ref{EQ8}) we get that $(\sqrt{2}^m-1)^{2l}<\Delta_{m;1}< (\sqrt{2}^m+1)^{2l+2}(2^m+1)^{2l}$. So part (8) holds for $m=4l+3$.
\end{proof}

\begin{Ex}\normalfont\label{EX1}
Assume $(m,r)\in\mathbb N^2$ with $m$ odd, $m\ge r$, and $\gcd(m,r)=1$. We compute $\Delta_{m;0}=-\prod_{i=1}^m \Psi_m(-\zeta_m^i)=-\prod_{i=1}^m (-2)=(-1)^{m+1}2^m$ by Proposition \ref{L3}(3). So $R_{m,r,p;0}$ is nonzero iff $p=2$ by Theorem \ref{TH1}(5).
\end{Ex}

\section{Variation of degrees of polynomials over finite prime fields}\label{S3}

Let $R$ be a unital commutative ring and $a\in\mathbb Z$. We denote also by $a$ its image via the unique unital ring homomorphism $\mathbb Z\rightarrow R$; so $aR$ is the principal ideal of $R$ generated by $a\in R$.

If $p\nmid m$ and $\zeta$ is a primitive $m$-th root of $\overline{\mathbb F_p}$, let 
$$\Omega_{m,p,\zeta;a}:=\{(i,j)\in\mathbb I_m|\zeta^i+\zeta^j=a\}.$$

\begin{Le}\label{L4}
Let $a\in\mathbb Z$ and $N\in\mathbb N\setminus\{1\}$ be square free. Then the following properties hold.

\medskip
{\bf (1)} For every pair $(m,r)\in\mathbb N^2$ with $\gcd(m,r)=1$, there exists a unique polynomial $\theta_{m,N,r;a}(X)$ in $(\mathbb Z/N\mathbb Z)[X]$ such that the following properties hold.

\medskip\noindent
{\bf (1.a)} For each prime divisor $p$ of $N$, the reduction of $\theta_{m,N,r;a}(X)$ modulo $p$ is monic.

\smallskip\noindent
{\bf (1.b)} We have an identity $\bigl(\Psi_m(X)+(N),\Psi_m(a-X^r)+(N)\bigr)=\bigr(\theta_{m,N,r;a}(X)\bigr)$ between ideals of $(\mathbb Z/N\mathbb Z)[X]$.

\medskip
{\bf (2)} The degree of $\theta_{m,N,r;a}(X)$ does not depend on $r$.
\end{Le}

\begin{proof} As $\mathbb Z/N\mathbb Z\cong\prod_{p\mid N,p\,\textup{is a prime}} \mathbb F_p$, to prove part (1) we can assume that $N=p$ is a prime. So part (1) follows from the fact that $\mathbb F_p[X]$ is a principal ideal domain with the property that each nonzero ideal of it is generated by a unique monic polynomial.

As
$$\deg(\theta_{m,N,r;a})=\max\bigl(\deg(\theta_{m,p,r;a})|p\,\text{is a prime divisor of}\,N\bigr),$$ 
to prove part (2) we can assume that $N=p$ is a prime. 

We write $m=p^ls$ and $r=p^{l_r}s_r$, where the quadruple $(l,l_r,s,s_r)\in [\mathbb N\cup\{0\}]^2\times [\mathbb N]^2$ is such that $\gcd(p,s)=\gcd(p,s_r)=1$. As $\gcd(m,r)=1$, we have $ll_r=0$. Let $\zeta\in\overline{\mathbb F_p}$ be a primitive $s$-th root of unity. 

As $\Psi_m(X)+(p)=X^{p^ls}-1+(p)=(X^s-1)^{p^l}+(p)=\Psi_s(X)^{p^l}+(p)\in\mathbb F_p[x]$ and the reduction of $\Psi_s(X)$ modulo $p$ is a monic separable polynomial, we get that each root of $\Psi_m(X)$ in $\overline{\mathbb F_p}$ has multiplicity $p^l$. As $a\in\mathbb Z$, a similar argument gives that each root of $\Psi_m(a-X^r)$ in $\overline{\mathbb F_p}$ has multiplicity $p^{l+l_r}$). 

It follows that $\theta_{m,p,r;a}(X)=\theta_{s,p,r;a}^{p^l}(X)$ and 
$$\theta_{m,p,r;a}(X)=\prod_{(i,j)\in\mathbb I_m,\,(ri,j)\in\Omega_{s,p,\zeta;a}} (X-\zeta^i)^{p^l}.$$ 
As $r$ is relatively prime to $m$, the functions $\{(i,j)\in\mathbb I_m,\,(ri,j)\in\Omega_{s,p,\zeta;a}\}\rightarrow \llbracket0, m-1\rrbracket$ defined by the rule $(i,j)\mapsto j$ are injective and their images do not depend on $r$. In particular, we have an identity $|\{(i,j)\in\mathbb I_m,\,(ri,j)\in\Omega_{s,p,\zeta;a}\}|=|\Omega_{s,p,\zeta;a}|$. Therefore $\deg(\theta_{m,p,r;a})=p^l|\Omega_{s,p,\zeta;a}|$ and hence it is independent of $r$. So part (2) holds.
\end{proof}

For each $m\in\mathbb N$ and every prime $p$, let 
$$t_{m,p;a}:=\deg(\theta_{m,p,r;a})\in\mathbb N\cup\{0\}$$
with $r\in\mathbb Z$ arbitrary subject to being relatively prime to $m$ (cf.\ Lemma \ref{L4}(2)).

We have the following direct consequence of the definition of $t_{m,p;a}$ and Lemma \ref{L4}.

\begin{Co}\label{C1} 
Let $N\in\mathbb N\setminus\{1\}$ be square free, $(m,r)\in\mathbb N^2$ with $\gcd(m,r)=1$, $d\in\mathbb N\cup\{0\}$, and $a\in\mathbb Z$. Then the polynomial $\theta_{m,N,r;a}(X)\in (\mathbb Z/N\mathbb Z)[X]$ is monic of degree $d$ iff for each prime divisor $p$ of $N$ we have $t_{m,p;a}=d$. In particular, $\theta_{m,N,r;a}(X)$ is monic iff $\theta_{m,N,1;a}(X)$ is monic. 
\end{Co}

For $(a,m)\in\mathbb Z\times\mathbb N$, let $[a]_m:=a+m\mathbb Z\in\mathbb Z/m\mathbb Z$. We consider the left action 
$$\mathbb Z/m\mathbb Z\times\mathbb I_m\rightarrow\mathbb I_m$$ 
of the multiplicative monoid $\mathbb Z/m\mathbb Z$ on $\mathbb I_m$ defined by the rule $[l]_m\cdot (i,j):=(i_l,j_l)$
with $[li]_m=[i_l]_m$ and $[lj]_m=[j_l]_m$. Let $U_m$ be the group of units of this monoid (equivalently, of the ring $\mathbb Z/m\mathbb Z$). The induced action of $U_m$ on the subset 
$$\mathbb U_m:=\{(i,j)\in\mathbb I_m|1\in\{\gcd(i,m),\gcd(j,m)\}\}$$ of $\mathbb I_m$ is free. Let $\Stab_{m,p,\zeta;a}$ be the stabilizer of the subset $\Omega_{m,p,\zeta;a}$ of $\mathbb I_m$ with respect to the action $U_m\times\mathbb I_m\rightarrow\mathbb I_m$. 

Let $\pi:\mathbb I_m\rightarrow \llbracket0, m-1\rrbracket$ be the first projection. Let 
$$\omega_{m,p,\zeta;a}:=\{i\in \llbracket0, m-1\rrbracket|\exists\, j\in \llbracket0, m-1\rrbracket\,\textup{with}\,(i,j)\in\Omega_{m,p,\zeta;a}\}$$ 
be the image of $\Omega_{m,p,\zeta;a}$ via $\pi$. 

We have the following consequence of Lemma \ref{L4} and Theorem \ref{TH1}(4) and (5) which groups together the basic properties of the $t_{m,p;a}$s.

\begin{Co}\label{C2}
Let $a\in\mathbb Z$ and $m\in\mathbb N$. Let $p$ be a prime. We write $m=p^ls$ with the pair $(l,s)\in (\mathbb N\cup\{0\})\times\mathbb N$ such that $\gcd(p,s)=1$. Then the following properties hold.

\medskip
{\bf (1)}
We have $t_{m,p;a}\in\mathbb N$ iff $p\mid\Delta_{m;a}$.

\smallskip
{\bf (2)} We have an identity $\theta_{m,p,1;a}=\theta_{s,p,1;a}^{p^l}$. In particular, $t_{m,p;a}=p^lt_{s,p;a}$.

\smallskip
{\bf (3)} Assume that $l=0$ (equivalently, $m=s$). Let $\zeta\in\overline{\mathbb F_p}$ be a primitive $m$-th root of unity. Then $t_{m,p;a}=|\Omega_{m,p,\zeta;a}|$.

\smallskip
{\bf (4)} The integer $t_{m,p;a}$ is odd iff $p\mid 2^s-a^s$ and iff $p\mid 2^m-a^m$.

\smallskip
{\bf (5)} The function $\Omega_{m,p,\zeta;a}\rightarrow\omega_{m,p,\zeta;a}$ we get by restricting $\pi$ is a bijection. 

\smallskip
{\bf (6)} We have $0\notin \omega_{m,p,\zeta;1}$.

\smallskip
{\bf (7)} If $m$ is odd, then $t_{m,p;1}\le\frac{m-1}{2}$.

\smallskip
{\bf (8)} Assume that $m$ is odd. For each $i\in\omega_{m,p,\zeta;1}$ and every $[u]_m\in U_m$ with $[ui]_m=[m-i]_m$, we have $[u]_m\notin\Stab_{m,p,\zeta;1}$.\footnote{As $\gcd(m,i)=\gcd(m,m-i)$, such a unit $[u]_m$ always exists.} In particular, $\Stab_{m,p,\zeta;1}$ is a subgroup of $U_m$ of index at least $2$.\end{Co}

\begin{proof} 
For part (1), we have $t_{m,p;a}>0$ iff the ring $R_{m,1,p;a}$ is nonzero. If $\Delta_{m;a}=0$, then $R_{m,1;a}$ is infinite by Theorem \ref{TH1}(3) and hence it has positive rank and this implies, based on Remark \ref{R1}(3), that the ring $R_{m,1,p;a}$ is nonzero. If $\Delta_{m;a}\neq 0$, $R_{m,1;a}$ is finite by Theorem \ref{TH1}(3) and the ring $R_{m,1,p;a}$ is nonzero iff $p\mid\Delta_{m;a}$ by Theorem \ref{TH1}(4) and (5). Part (1) follows from the last two sentences.

Parts (2) and (3) follow from the proof of Lemma \ref{L4}. 

We have $2^s\equiv a^s\pmod{p}$ iff $2^{p^ls}\equiv a^{p^ls}\pmod{p}$. Thus $p\mid 2^s\equiv a^s$ iff $p\mid 2^m-a^m$. Based on this and part (2), to prove part (4) we can assume that $p$ is odd and $l=0$. We have an involution of $\Omega_{m,p,\zeta;a}=\Omega_{s,p,\zeta;a}$ which maps $(i,j)$ to $(j,i)$. This involution either has no fixed point, in which case $t_{m,p;a}$ is even, or it has a unique fixed point, in which case $t_{m,p;a}$ is odd. Moreover, the involution has a fixed point iff $\frac{a}{2}\in\mathbb F_p$ is an $m$-th roots of unity and thus iff $p\mid 2^m-a^m$. So part (4) holds. 

Part (5) holds as for each $j\in\llbracket0,m-1\rrbracket$, the value of $\zeta^j\in\overline{\mathbb F_p}$ determines $j$.
	
Part (6) holds as for each $j\in\llbracket0,m-1\rrbracket$, we have $\zeta^j\in\overline{\mathbb F_p}\setminus\{0\}$.

To prove part (7), let $(i,j)\in\Omega_{m,p,\zeta;1}$; we have $\{i,j\}\subset \llbracket1, m-1\rrbracket$ by part (6). As $\zeta^i+\zeta^j=1$, we have $1-\zeta^{m-i}=1-\zeta^{-i}=-\zeta^{j-i}$. As $m$ is odd, $-\zeta^{j-i}$ is not an $m$-th root of unity and hence $m-i\notin\omega_{m,p,\zeta;1}$. Thus the involution $\tau$ of $\llbracket1, m-1\rrbracket$ defined by the rule $\tau(l)=m-l$ is such that the union $\tau(\omega_{m,p,\zeta;1})\cup \omega_{m,p,\zeta;1}$ is disjoint. This implies that $2|\omega_{m,p,\zeta;1}|\le m-1$ and hence, based on parts (3) and (5), part (7) holds.

For part (8), let $j\in \llbracket1, m-1\rrbracket$ be such that $(i,j)\in\Omega_{m,p,\zeta;1}$. If $[u]_m\in\Stab_{m,p,\zeta;1}$, then $m-i\in\omega_{m, p,\zeta;1}$, a contradiction to the last paragraph. Thus $[u]_m\notin\Stab_{m,p,\zeta;1}$ and part (8) holds.
\end{proof}

For an integer $N\ge 2$ let $\mathcal P(N)=\mathcal P(-N)\in\Div(N)$ be the largest prime divisor of $N$. If $m\ge 2$ with $6\nmid m$, then we have an inequality $\mathcal P(2^m-1)\le\mathcal P(\Delta_{m;1})$ by Theorem \ref{TH1}(7) which in general is strict. 

\begin{Ex}\normalfont\label{EX2}
{\bf (1)} The prime divisors of $\Delta_{11;1} =101832157445630503$ are $23$, $67$, $89$, and $199$. A direct computation gives that 
	$$\theta_{11,23,1;1}(X)=X^5+9X^4+8X^3+2X^2+4X+1\in\mathbb F_{23}[X],$$ 
	$$\theta_{11,67,1;1}(X)=X^2-X-5\in\mathbb F_{67}[X],\;\;\;\theta_{11,89,1;1}(X)=X+45\in\mathbb F_{89}[X],$$ 
	and 
	$$\theta_{11,199,1;1}(X)=X^2-X-78\in\mathbb F_{199}[X].$$ 
	Therefore $t_{11,23;1}=5$, $t_{11,67;1}=t_{11,199;1}=2$, and $t_{11,89;1}=1$. As $t_{11,23;1}=5=\frac{11-1}{2}$, Corollary \ref{C2}(7) is optimal in general.

\smallskip
{\bf (2)} For $m=11$ we have 
$$\mathcal P(2^{11}-1)=\mathcal P(2047)=89<199=\mathcal P(\Delta_{11;1})$$ and for $m=15$ we have 
$\mathcal P(2^{15}-1)=\mathcal P(32767)=151$ and
$$151<\mathcal P(\Delta_{15;1})=\mathcal P(\Delta_{3,5;1})=\mathcal P(\Delta_{3,15;1})=\mathcal P(\Delta_{5,15;1})=\mathcal P(\Delta_{15,15;1})=271.$$

{\bf (3)} It is easy to check that Theorem \ref{TH1}(8) implies that there exist universal constants $C_1$ and $C_2$ in the interval $(0,1)$ such that $C_1<C_2$ and for each $m\in\mathbb N\setminus 6\mathbb N$ we have inequalities $$2^{C_1m^2}\le\Delta_{m;1}\le 2^{C_2m^2}.$$

{\bf (4)} Assume $a\ge 3$. Then $\Delta_{1,1;a}=a-2>0$, $\Delta_{2,2;a}=a+2>0$ and the proof of Proposition \ref{L3}(5) adapts to give that $\Delta_{m;a}<0$. The positive integers $-\Delta_{m;a}$ grow very fast in terms of $m$; like, $-\Delta_{3;3}=14896$, $-\Delta_{5;3}=803800668451$, and $-\Delta_{7;3}=237162507224588197994176$.
\end{Ex}

\begin{theorem}\label{TH2}
Assume that $6\nmid m$. Let $\mathbb K_m$ be the set of pairs $(i,j)\in\mathbb I_{m-1}$ such that $ij>0$, $i\neq j$, and the stabilizer of $\zeta_m^i+\zeta_m^j$ in the Galois group $\Gal\bigl(\mathbb Q(\zeta_m)/\mathbb Q\bigr)$ is trivial. 

For a prime $p$ we consider the following statements (conditions).
	
\medskip
{\bf (1)} We have inequalities $5\le m<2^m\le p\le 3^{\phi(m)}$ and $m$ is odd.

\smallskip
{\bf (2)} We have $\mathcal Z_m/p\mathcal Z_m\cong\mathbb F_p^{\phi(m)}$ and $m\mid p-1$.

\smallskip
{\bf (3)} We have $t_{m,p;1}\in 2\mathbb N$. 

\smallskip
{\bf (4)} Let $\zeta\in\mathbb F_p\setminus\{0\}$ be a primitive $m$-th root of unity. Then for each $[l]_m\in U_m\setminus\{[1]_m\}$ we have $\Omega_{m,p,\zeta;1}\cap ([l]_m\Omega_{m,p,\zeta;1})=\emptyset$.

\smallskip
{\bf (5)} For each $(i,j)\in\Omega_{m,p,\zeta;1}$, the stabilizer of $\zeta_m^i+\zeta_m^j$ in the Galois group $\Gal\bigl(\mathbb Q(\zeta_m)/\mathbb Q\bigr)$ is trivial.

\smallskip
{\bf (6)} We have $|\mathbb K_m|\ge 2t_{m,p;1}[\phi(m)-1]$ and $t_{m,p;1}[\phi(m)-1]\le (m-1)(m-3)$.

\medskip
If $p\mid\Delta_{m;1}$ and $p>\min\left(2^m-1,3^{\phi(m)}\right)$, then statements (1) to (6) are true.
\end{theorem}

\begin{proof} As $\Delta_{1;1}=1$, $\Delta_{2;1}=3$, and $\Delta_{3;1}=7$, we have $m\ge 4$. So $p\ge 2^m>m$ and hence $p\nmid m$. Let $n\in \mathbb{Z}$ be such that $p n =\Delta_{m;1}$. We have $n\in\mathbb N$ by Proposition \ref{L3}(7).
	
By Proposition \ref{L3}(3), using Equations (\ref{EQ4}) we have
\begin{equation}\label{EQ9}
-pn=-\Delta_{m;1}=\prod_{i=0}^{m-1} [(1-\zeta_m^i)^m -1]=\prod_{i=0}^{m-1}\prod_{j=0}^{m-1} (1-\zeta_m^i - \zeta_m^j)=\prod_{(i,j)\in\mathbb I_m} (1-\zeta_m^i - \zeta_m^j). 
\end{equation}
Applying the Galois norm $N_{\mathbb{Q}(\zeta_m)/\mathbb{Q}}$ to the previous equality, we get:
$$
N_{\mathbb{Q}(\zeta_m)/\mathbb{Q}}(-pn) = \prod_{(i,j)\in\mathbb I_m} N_{\mathbb{Q}(\zeta_m)/\mathbb{Q}}(1-\zeta_m^i - \zeta_m^j). 
$$	
As $N_{\mathbb{Q}(\zeta_m)/\mathbb{Q}}(-pn)=(-pn)^{\phi(m)}$, and as $N_{\mathbb{Q}(\zeta_m)/\mathbb{Q}}(1-\zeta_m^i - \zeta_m^j)\in \mathbb{Z}$ for all $(i,j)\in\mathbb I_m$, there exist a pair $(i_0,j_0)\in\mathbb I_m$ such that $p$ divides $N_{\mathbb{Q}(\zeta_m)/\mathbb{Q}}(1-\zeta_m^{i_0} - \zeta_m^{j_0})$. As $p$ does not divide $N_{\mathbb{Q}(\zeta_m)/\mathbb{Q}}(-1)=(-1)^{\phi(m)}$, we have $(i_0,j_0)\neq (0,0)$. Thus $i_0j_0>0$.
As
$$
N_{\mathbb{Q}(\zeta_m)/\mathbb{Q}}(1-\zeta_m^{i_0} - \zeta_m^{j_0})=\prod_{k\in \llbracket0, m-1\rrbracket,\gcd(k,m)=1} (1-\zeta_m^{ki_0} - \zeta_m^{kj_0}),
$$
$N_{i_0,j_0}:=|N_{\mathbb{Q}(\zeta_m)/\mathbb{Q}}(1-\zeta_m^{i_0} - \zeta_m^{j_0})|$ is a positive integer bounded from above by
$$
\prod_{k\in \llbracket0, m-1\rrbracket,\gcd(k,m)=1} 1+|\zeta_m^{ki_0}|+|\zeta_m^{kj_0}|= 3^{\phi (m)}.
$$
Thus $p\le 3^{\phi (m)}$. As $p>\min(2^m-1,3^{\phi(m)})$ it follows that $m<2^m\le p\le 3^{\phi(m)}$.

If $m$ is even, then $\phi(m)\le\frac{m}{2}$ and hence $3^{\phi(m)}\le \sqrt{3}^m<2^m$, a contradiction. Thus $m\ge 5$ is odd. So statement (1) is true.

Let $q\in\mathbb N$ be the order of $p$ modulo $m$; so $r:=\frac{\phi(m)}{q}\in\mathbb N$. As $p\nmid m$, $\mathcal Z_m/p\mathcal Z_m\cong\mathbb F_{p^q}^r$. So there exist exactly $r$ prime (equivalently, maximal) ideals $\mathfrak p_1,\ldots,\mathfrak p_r$ of $\mathcal Z_m$ that contain $p$; they are permuted transitively by the Galois group $\mathcal G_m:=\Gal\bigl(\mathbb Q(\zeta_m)/\mathbb Q\bigr)$. Let $s_0\in\mathbb N$ be the number of these maximal ideals that contain $1-\zeta_m^{i_0} - \zeta_m^{j_0}$. It follows that $N_{i_0,j_0}$ is divisible by $p^{\frac{s_0\phi(m)}{r}}=p^{s_0q}$. 

If $s_0q\ge 2$, then $p^2\le N_{i_0,j_0}\le 3^{\phi(m)}\le 3^{m-1}$, hence $p<\sqrt{3}^m<2^m$, a contradiction.

Thus $s_0q=1$ and hence $s_0=q=1$. Thus $m\mid p-1$ and $r=\phi(m)$, so statement (2) is true. 

By reindexing we can assume that the image of $\zeta_m$ in $\mathcal Z_m/\mathfrak p_1\cong\mathbb F_p$ is $\zeta$. 

As $p>2^m-1\ge m$, Corollary \ref{C2}(4) implies that statement (3) is true and that $i\neq j$.

If $(i,j)\in\Omega_{m,p,\zeta;1}$, hence $1-\zeta_m^i-\zeta_m^j\in\mathfrak p_1$, by taking $(i_0,j_0):=(i,j)$, for $l\in \llbracket2,m-1\rrbracket$ relatively prime to $m$, $1-\zeta_m^{li}-\zeta_m^{lj}\in\mathfrak p_k$ with $k\in \llbracket2,r\rrbracket$ and hence from $s_0=1$ it follows that $[l]_m(i,j)\notin\Omega_{m,p,\zeta;1}$. So statement (4) is true.

As $r=\phi(m)$ and $s=1$ we get that the stabilizer of $1-\zeta_m^{i_0} - \zeta_m^{j_0}$, equivalently of $\zeta_m^{i_0} +\zeta_m^{j_0}$, in $\mathcal G_m$ is the trivial subgroup. From this, if $(i,j)\in\Omega_{m,p,\zeta;1}$, by taking $(i_0,j_0):=(i,j)$ we get that statement (5) is true and that $i+j\neq m$.

Based on statements (4) and (5) we have a disjoint union 
$$\sqcup_{(i,j)\in\Omega_{m,p,\zeta;1}} \{(i_l,j_l)|l\in \llbracket2,m-1\rrbracket,\gcd(m,l)=1\}$$
of subsets of $\mathbb K_m$ and hence also of subsets of $\{(i,j)\in\mathbb I_{m-1}|ij>0,\;i\neq j,\;i+j\neq m\}$. As the last set has $(m-1)^2-2(m-1)=(m-1)(m-3)$ elements, statement (6) is true. 
\end{proof}

\begin{Co}\label{C3}
Let $m\in \mathbb N\setminus 6\mathbb N$. Let $p$ be a prime such that $p>\min\bigl(2^m-1,3^{\phi (m)}\bigr)$. Assume that one of the conditions (1) to (6) of Theorem \ref{TH2} does not hold (e.g., $m\nmid p-1$). Then there exist no $m$-th roots of unity $x$ and $y$ in $\overline{\mathbb F_p}$ such that $x+y=1$. 
\end{Co}

\begin{proof}
We show that the assumption that such $x$ and $y$ exist leads to a contradiction. There exists a nonzero ring homomorphism $\rho:R_{m,1,p;1}\rightarrow\overline{\mathbb F_p}$ such that $\rho(X+I_{m,r,p;1})=x$ and $\rho(Y+I_{m,r,p;1})=y$. Thus $p\mid\Delta_{m;1}$ by Theorem \ref{TH1}(3) and (5) and this contradicts Theorem \ref{TH2}.
\end{proof}

\begin{Not}\normalfont\label{N1}
Let $p$ be a prime and $a\in\mathbb Z$. 

\medskip
{\bf (1)} For $\alpha\in\overline{\mathbb F_p}\setminus\{0\}$, let $o(\alpha)\in\mathbb N$ be its order. Moreover, for $\alpha\in\overline{\mathbb F_p}\setminus\{0,a\}$ we also define $\ell_a(\alpha):=\textup{lcm}\bigl(o(\alpha),o(a-\alpha)\bigr)$. We note that $\ell_a(\alpha)=\ell_a(a-\alpha)$.  

\smallskip
{\bf (2)} If $\gamma\in\mathbb F_p\setminus\{0\}$ is the reduction of an integer $\beta$ modulo $p$, let $o_p(\beta):=o(\gamma)$.

\smallskip
{\bf (3)} We define subsets $SO_{p;a}:=\{\ell_a(\alpha)|\alpha\in\mathbb F_p\setminus\{0,a\}\}$, $O_{p;a}:=\{\ell_a(\alpha)|\alpha\in\overline{\mathbb F_p}\setminus\{0,a\}\}$, $MSO_{p;a}:=\cup_{\ell\in SO_{p;a}} \ell\mathbb N$, and $MO_{p;a}:=\cup_{\ell\in O_{p;a}} \ell\mathbb N$ of $\mathbb N$. 

\smallskip
{\bf (4)} If we have an inclusion $MO_{p;1}\subset 6\mathbb N$, then let $m_{p;1}:=0$. If the set difference $MO_{p;1}\setminus 6\mathbb N$ is non-empty, then let $m_{p;1}:=\min(MO_{p;1}\setminus 6\mathbb N)\in\mathbb N$.

\smallskip
{\bf (5)} If we have an inclusion $MO_{p;-1}\subset 3\mathbb N$, then let $m_{p;-1}:=0$. If the set difference $MO_{p;-1}\setminus 3\mathbb N$ is non-empty, then let $m_{p;-1}:=\min(MO_{p;-1}\setminus 3\mathbb N)\in\mathbb N$.

\smallskip
{\bf (6)} If we have an inclusion $MO_{p;-2}\subset 2\mathbb N$, then let $m_{p;-2}:=0$. If the set difference $MO_{p;-2}\setminus 2\mathbb N$ is non-empty, then let $m_{p;-2}:=\min(MO_{p;-2}\setminus 2\mathbb N)\in\mathbb N$.

\smallskip
{\bf (7)} If $|a|\ge 3$, let $m_{p;a}:=\min(MO_{p;a})\in\mathbb N$.
\end{Not}

From Corollary \ref{C3} we get directly the following consequence.

\begin{Co}\label{C4}
Let $p$ be a prime such that $m_{p;1}>0$. If one of the conditions (1) to (6) of Theorem \ref{TH2} does not hold for $m_{p;1}$ (e.g., $m_{p;1}\nmid p-1$), then we have an inequality $p\le\min\bigl(2^{m_{p;1}}-1,3^{\phi(m_{p;1})}\bigr)$. 
\end{Co}

\begin{Le}\label{L5}
Let $p$ be a prime. Then the following properties hold.

\medskip
{\bf (1)} If $6\nmid p-1$, then $m_{p;1}>0$.

\smallskip
{\bf (2)} If $6\mid p-1$, then there exists $\alpha\in\mathbb F_p\setminus\{0,1\}$ such that $6\nmid o(\alpha)$ and $6\nmid o(1-\alpha)$.
\end{Le}

\begin{proof}
If $p=2$, then for $\alpha\in\mathbb F_4\setminus\mathbb F_2$ we have $o(\alpha)=o(1-\alpha)=3$, hence $3\in MO_{2;1}$. If $p=3$, then $o_3(2)=2\in MO_{3;1}$. If $6\mid p-5$, then $o_p(2)\in MO_{p;1}\setminus 6\mathbb N$ is a divisor of $p-1$. So part (1) holds.

For part (2), let $s\in\mathbb N$ be such that $p-1=6s$. The number of elements in $\mathbb F_p\setminus\{0\}$ that have order divisible by $6$ is $$C:=\sum_{i\in\Div(s)} \phi(6i)<\frac{3}{2}\sum_{i\in\Div(s)}\phi(6)\phi(i)=3\sum_{i\in\Div(s)}\phi(i)=3s.$$ 
As $C<3s=\frac{p-1}{2}$, $2C+1\le p-2$ and thus there exists $\alpha\in\mathbb F_p\setminus\{0,1\}$ such that $6\nmid o(\alpha)$ and $6\nmid o(1-\alpha)$. So part (2) holds.
\end{proof}

\begin{Rem}\normalfont\label{R2}
{\bf (1)} Let $p$ be a prime, $m\in\mathbb N$, and $a\in\mathbb Z$. There exists a nonzero (equivalently, surjective) ring homomorphism $R_{m,1,p;a}\rightarrow\mathbb F_p$ iff $m\in MSO_{p;a}$. Similarly, there exists a nonzero ring homomorphism $R_{m,1,p;a}\rightarrow F$ with $F$ a finite field of characteristic $p$ iff $m\in MO_{p;a}$. 

\smallskip
{\bf (2)} From part (1), Theorem \ref{TH1}(4), and Remark \ref{R1}(3) we get the following inclusions and identities $6\mathbb N\subset MO_{p;1}$, $3\mathbb N\subset MO_{p;-1}$, $2\mathbb N\subset MO_{p;0}\cap MO_{p;-2}$, and $\mathbb N=MO_{p;2}$ for each prime $p$. Moreover, $\mathbb N=MO_{2;a}$ if $a\in 2\mathbb Z$.

\smallskip
{\bf (3)} Let $p$ be an odd prime. Then $\ell_1(\frac{1}{2})=o_p(2)\in SO_{p;1}$. If $6\nmid o_p(2)$ (e.g., this holds if $p\not\equiv 1 \pmod{6}$), then $m_{p;1}\le o_p(2)$ with $p\mid 2^{o_p(2)}-1$. In general $m_{p;1}<o_p(2)$ and hence $p\nmid 2^{m_{p;1}}-1$. For instance, if $p=11$, then $o_{11}(2)=10$ but $m_{11;1}\le 5$ as $SO_{11;1}=\{5,10\}$; we have $5\in SO_{11;1}$ as $o_{11}(3)=o_{11}(9)=5$.

\smallskip
{\bf (4)} We have $\min(MO_{2;1})=3$.

\smallskip
{\bf (5)} We have $t_{m,p;a}\in\mathbb N$ iff $\theta_{m,p,1;a}(X)$ has a root in $\overline{\mathbb F_p}$ and iff $m\in MO_{p;a}$. Similarly, $\theta_{m,p,1;a}(X)$ has a root in $\mathbb F_p$ iff $m\in MSO_{p;a}$.

\smallskip
{\bf (6)} We have $SO_{13;1}\subset 6\mathbb N$ with $o_{13}(5)=4$ and $o_{13}(9)=3$ not divisible by $6$.
\end{Rem}

\section{Frobenius Identity}\label{S4}

For elements $x, y$ in a unital commutative ring $R$ and $n\in \mathbb{N}$, we write $x\equiv y \pmod{n}$ if $x-y\in nR$. The following well-known identity is used repeatedly throughout this section.

\begin{Le}[Frobenius Identity]\label{L1}
	Let $p$ be a prime number. Then, for each $x, y\in R$, we have the congruence
	\[
	(x + y)^p \equiv x^p + y^p \pmod{p}.
	\]
\end{Le}
\begin{proof}
	This follows from the binomial expansion and the fact that for each $k\in \llbracket1,p-1\rrbracket$ the binomial coefficient is divisible by $p$ and hence we have $\binom{p}{k}x^ky^{p-k}\in pR$.
\end{proof}

We aim to prove the following generalized form of Frobenius Identity. 

\begin{theorem}\label{TH0}
	If $N=p_1^{e_1}\cdots p_l^{e_l}$, where $l\in\mathbb N$, $p_1, \ldots, p_l$ are distinct primes, and $e_1, \ldots e_l$ are natural numbers, then for each $x, y\in R$ 
	the following congruence holds
	\begin{align}\label{formula}
		(x+y)^N \equiv 
	\end{align}
	\begin{align*}
		\overset{l}{\underset{s=1}{\sum}}
		\underset{\{j_1, \ldots , j_{l-s}\}=\llbracket1, l\rrbracket \setminus \{i_1, \ldots , i_s\} }{\underset{\{i_1, \ldots , i_s\}\subseteq \llbracket1, l\rrbracket}
			{\sum}}(-1)^{s+1}(x^{p_{i_1}^{e_{i_1}}\cdots p_{i_s}^{e_{i_s}}}+
		y^{p_{i_1}^{e_{i_1}}\cdots p_{i_s}^{e_{i_s}}})^{p_{j_1}^{e_{j_1}}\cdots p_{j_s}^{e_{j_{l-s}}}} \: (\text{mod}\: p_1\cdots p_l).
	\end{align*}
\end{theorem}

Although we apply this congruence only in the case when $l=2$ and $e_1=e_2=1$, we include and prove it in its full generality with the hope that it may be useful in other contexts. As such, and to make the proof of Congruence (\ref{formula}) easier to follow, we first consider this case.

\begin{Pro}\label{P1}
	Let $p$ and $q$ be distinct prime numbers. If $N=pq$, then we have the congruence
	\[
	(x + y)^N \equiv (x^p + y^p)^q + (x^q + y^q)^p - x^N - y^N \pmod{N}.
	\] 
\end{Pro}
\begin{proof} As $p$ is a prime, it follows from Lemma \ref{L1} that
	\[
	(x + y)^{N} \equiv (x^p + y^p)^q \pmod{p}
	\]
	and 
	\[
	(x^q + y^q)^p - x^N - y^N \equiv 0 \pmod{p}.
	\]
	Therefore,
	\[
	(x + y)^N \equiv (x^p + y^p)^q + (x^q + y^q)^p - x^N - y^N \pmod{p}.
	\]
	
	\noindent The same argument applies modulo $q$. Therefore, the two sides of the identity are congruent modulo both $p$ and $q$, and, as $N = pq$, the result follows by the Chinese Remainder Theorem applied to the coprime principal ideals $pR$ and $qR$ of $R$.\end{proof}

We now prove the general case of Congruence \eqref{formula}.

\medskip
{\it Proof of Theorem \ref{TH0}.} We begin by showing Congruence \eqref{formula} for $p_1$ in place of $p_1\cdots p_l$.

By Lemma \ref{L1} applied $e_1$ times, we have the identity $(x+y)^{p_1^{e_1}} \equiv x^{p_1^{e_1}} + y^{p_1^{e_1}} \pmod{p_1}.$
Thus, 
\[
(x+y)^N \equiv (x^{p_1^{e_1}} + y^{p_1^{e_1}})^{p_2^{e_2} \cdots p_l^{e_l}} \pmod{p_1}.
\]

Thus we need to prove the congruence
\begin{equation}\label{EQ9.5}
\overset{l}{\underset{s=1}{\sum}}\sum_{\substack{
		\{1\} \ne \{i_1, \dots, i_s\} \subseteq \llbracket1, l\rrbracket \\
		\{j_1, \dots, j_{l-s}\} = \llbracket1, l\rrbracket \setminus \{i_1, \dots, i_s\}
}} 
(-1)^{s+1}
\left(
x^{p_{i_1}^{e_{i_1}} \cdots p_{i_s}^{e_{i_s}}} +
y^{p_{i_1}^{e_{i_1}} \cdots p_{i_s}^{e_{i_s}}}
\right)^{p_{j_1}^{e_{j_1}} \cdots p_{j_{l-s}}^{e_{j_{l-s}}}}
\equiv 0 \pmod{p_1}.
\end{equation}

\noindent Using a substitution $s=t-1$ while denoting the elements $i_2$ to $i_t$ if $1\notin\{i_1,\ldots,i_s\}$ and a substitution $s=t$ while denoting the elements $i_1$ to $i_t$  if $1\in\{i_1,\ldots,i_s\}$, the left hand side of Congruence (\ref{EQ9.5}) is congruent to
$$
	\overset{l}{\underset{t=2}{\sum}}\sum_{\substack{
			\{i_2, \dots, i_t\} \subseteq \llbracket2, l\rrbracket \\
			\{j_2, \dots, j_{l-t+1}\} = \llbracket2, l\rrbracket \setminus \{i_2, \dots, i_t\}
	}} 
	(-1)^{t}
	\left( 
	x^{p_{i_2}^{e_{i_2}} \cdots p_{i_t}^{e_{i_t}}} + 
	y^{p_{i_2}^{e_{i_2}} \cdots p_{i_t}^{e_{i_t}}}
	\right)^{p_1^{e_1} p_{j_2}^{e_{j_2}} \cdots p_{j_{l-t+1}}^{e_{j_{l-t+1}}}}$$
$$+\overset{l}{\underset{t=2}{\sum}}\sum_{\substack{
			\{i_2, \dots, i_t\} \subseteq \mathbb \llbracket2, l\rrbracket\\
			\{j_2, \dots, j_{l-t+1}\} = \mathbb \llbracket2, l\rrbracket \setminus \{i_2, \dots, i_t\}
	}} 
	(-1)^{t+1}
	\left( 
	x^{p_1^{e_1} p_{i_2}^{e_{i_2}} \cdots p_{i_t}^{e_{i_t}}} + 
	y^{p_1^{e_1} p_{i_2}^{e_{i_2}} \cdots p_{i_t}^{e_{i_t}}}
	\right)^{p_{j_2}^{e_{j_2}} \cdots p_{j_{l-t+1}}^{e_{j_{l-t+1}}}} \equiv 0 \pmod{p_1}.$$

\noindent Here we have used the fact that, in characteristic \( p_1 \), Lemma \ref{L1} applied $e_1$ times gives
\[
\left(
x^{p_{i_2}^{e_{i_2}} \cdots p_{i_t}^{e_{i_t}}} + 
y^{p_{i_2}^{e_{i_2}} \cdots p_{i_t}^{e_{i_t}}}
\right)^{p_1^{e_1}} \equiv
x^{p_1^{e_1} p_{i_2}^{e_{i_2}} \cdots p_{i_t}^{e_{i_t}}} + 
y^{p_1^{e_1} p_{i_2}^{e_{i_2}} \cdots p_{i_t}^{e_{i_t}}}
\pmod{p_1}.
\]

By symmetry, the same holds for all other primes $ p_i $, and thus the identity holds modulo $p_1 \cdots p_l$ by the Chinese Remainder Theorem applied to the pairwise coprime principal ideals $p_1R,\ldots, p_lR$ of $R$. Thus Theorem \ref{TH0} holds.

\smallskip
Let $(a, b)\in \mathbb{Z}^2$. Next we show that if $(X+Y+b)^N-(b+a)^N\not \in I_{m, r, N; a}$, 
then for a specific $r$ the coset $(X+Y+b)^N-(b+a)^N + I_{m, r, N; a}$ contains a polynomial whose coefficients are divisible by a proper divisor of $N$.

\begin{Pro}\label{P2}
	Let $p$ be a prime divisor of $N$. Let $m\in\mathbb N\cap\llbracket2,p-1\rrbracket$. Let $(n,r)\in\mathbb N
\times\llbracket1,m-1\rrbracket$ be such that $p = m n + r$. 
	Let $(a,b)\in \mathbb{Z}^2$. Then the following properties hold.
	
	\medskip
	{\bf (1)} There exist unique integers $d_{i,j}\in p \llbracket0,\frac{N-p}{p}\rrbracket$ indexed by $(i,j)\in\mathbb I_{m,r}$ such that in $\mathbb Z[X,Y]$ we have the congruence
	\begin{equation}\label{EQ10}
		\left( X + Y +b \right)^N - \bigl(X^{p} + Y^{p}+(b+a)^p-a\bigr)^{\frac{N}{p}}\equiv \sum_{(i,j)\in\mathbb I_{m,r}} d_{i,j} X^i Y^j \pmod{J_{m,r,N; a}}.
	\end{equation}
	
	\smallskip
	{\bf (2)} In $\mathbb Z[X,Y]$ we also have the congruence
	\begin{equation}\label{EQ11}
		(X + Y+b)^N -(b+a)^N \equiv \sum_{(i,j)\in\mathbb I_{m,r}} d_{i,j} X^iY^j \pmod {I_{m,r,N; a}}.
	\end{equation}
	
	\smallskip
	{\bf (3)} 
	If $\left( X + Y +b \right)^N-(b+a)^N \notin I_{m,r,N; a}$, then there exists a pair $(i_0,j_0)\in \mathbb I_{m,r}$ such that $d_{i_0,j_0}\neq 0$ (so $\gcd(d_{i_0,j_0},N)$ 
	is a proper divisor of $N$).
\end{Pro}
\begin{proof}
	By Lemma \ref{L1}, we observe that
	\[
	\left(X + Y+b \right)^N = \left(X + Y+b\right)^{p (\frac{N}{p})} \equiv \left(X^{p} + Y^{p}+b^p\right)^{\frac{N}{p}} \pmod{p}.
	\]
	Moreover, $b^p \equiv (b+a)^p-a^p \equiv (b+a)^p-a \pmod{p}$. Thus
	\[
	(X + Y+b)^N - (X^{p} + Y^{p}+(b+a)^p-a)^{\frac{N}{p}}\in p\mathbb Z[X,Y].
	\]
	By Lemma \ref{L2}(3) and (5) applied to $(f,l)=\bigl((X + Y+b)^N - [X^{p} + Y^{p}+(b+a)^p-a]^{\frac{N}{p}},p\bigr)$ it follows that there exist unique integers 
	$d_{i,j}\in p \llbracket0,\frac{N-p}{p}\rrbracket$ indexed by $(i,j)\in\mathbb I_{m,r}$ such that Congruence (\ref{EQ10}) holds. So part (1) holds.

	As $p = mn + r$, we can write:
	\[
	X^{p} + Y^{p}+(b+a)^p-a=X^{mn + r} + Y^{mn + r}+(b+a)^p-a \equiv X^{r} + Y^{r}+(b+a)^p-a \equiv (b+a)^p \pmod{I_{m, r; a}}.
	\]
	Substituting this in the analog of Congruence (\ref{EQ10}) for $I_{m, r,N; a}$ instead of $J_{m, r,N; a}$ (recall that we have an inclusion $J_{m, r,N; a}\subseteq I_{m, r,N; a}$), 
	it follows that Congruence (\ref{EQ11}) holds. So part (2) holds.
	
	If $\left( X + Y +b \right)^N-(b+a)^N \notin I_{m,r,N; a}$, then there exists $(i_0,j_0)\in \mathbb I_{m,r}$ such that $d_{i_0,j_0}\neq 0$ by part (2). So part (3) holds.
\end{proof} 

\begin{Rem}\normalfont\label{R3}
	Although the nonvanishing of the polynomial $(X + Y+b)^N -(b+a)^N$ modulo $I_{m, r, N; a}$ ensures within the coset $(X + Y+b)^N -(b+a)^N+I_{m, r, N; a}$ the existence of a polynomial
	whose at least one coefficient is a proper divisor of $N$, for some $r<m$, we currently lack a polynomial-time algorithm to identify such a polynomial, i.e., a pair $(a,b)$ for which $m$ and $r$ would exist. 
	Nevertheless, this observation led to the idea that if $(X + Y -a)^N \not\equiv 0 \pmod{I_{m, r, N; a}}$, then we can find a proper divisor of $N$
	in polynomial time through the use of single-variable polynomials rather than bi-variate ones. 
	This possibility is explored in the next section.
\end{Rem}

We derive the following result in the special case of square-free semiprimes.

\begin{Pro}\label{P3}
	Let $N = pq$, where $p$ and $q$ are distinct primes with $p\ge 3$. Let $(m, a)\in\mathbb N\times \mathbb{Z}$ be such that $1<m<p$. Let $r\in \llbracket1, m-1\rrbracket$ be the residue of $p$ modulo $m$.
	Then the following properties hold.
	
	\medskip
	{\bf (1)} There exist unique integers $c_{i,j}\in q\mathbb Z$ indexed by $(i,j)\in\mathbb I_{m,r}$ such that in $\mathbb Z[X,Y]$ we have the congruence 
	\[
	(X+Y)^N-(X^q + Y^q)^p \equiv \sum_{(i,j)\in\mathbb I_{m,r}} c_{i,j}X^iY^j \pmod {J_{m,r; a}}.
	\]
	
	\smallskip
	{\bf (2)} If $\left( X + Y \right)^N - a^q\in I_{m,r,N; a}$, then in $\mathbb Z[X,Y]$ we have the congruence 
	\[
	(X + Y)^N - X^N - Y^N \equiv \sum_{(i,j)\in\mathbb I_{m,r}} c_{i,j} X^i Y^j \pmod{I_{m,r,N; a}}.
	\]
	
	\smallskip
	{\bf (3)} If $\left( X + Y \right)^N - a^q\in I_{m,r,N; a}$ and $(X + Y)^N - X^N - Y^N \notin I_{m, r,N; a}$, then there exists a pair $(i_0,j_0)\in \mathbb I_{m,r}$ such that $\gcd(c_{i_0,j_0}, N)=q$.
\end{Pro}
\begin{proof}
We have $(X+Y)^N-(X^q + Y^q)^p\in q\mathbb{Z}[X, Y]$ by Lemma \ref{L1} applied to the prime $q$. From this and Lemma \ref{L2}(2) and (3) it follows that there exist unique integers $c_{ij}\in q\mathbb{Z}$ indexed by $(i,j)\in\mathbb I_{m,r}$ such that part (1) holds.

	For part (2), as $X^m\equiv Y^m\equiv 1 \pmod{I_{m,r; a}}$, it follows that $X^p+Y^p\equiv X^r+Y^r\equiv a \pmod{I_{m,r; a}}$. Based on this and Proposition \ref{P1}, we have:
	\[
	(X + Y)^N - a^q \equiv (X^q + Y^q)^p - X^N - Y^N \pmod{I_{m,r,N; a}}.
	\]
	As $\left( X + Y \right)^N - a^q \in I_{m,r,N; a}$, we incur that
	\[
	(X^q + Y^q)^p - X^N - Y^N \equiv 0 \pmod{I_{m,r,N; a}}.
	\]
	As $J_{m,r; a} \subset I_{m, r, N; a}$, part (2) follows from part (1) and the last congruence.
	
	For part (3), as $(X + Y)^N - X^N - Y^N \notin I_{m,r,N; a}$, it follows that
	\[
	\sum_{(i,j)\in\mathbb I_{m,r}} c_{i,j}X^iY^j \not\equiv 0\pmod {I_{m,r,N;a}}.
	\]
	This implies that there exists $(i_0, j_0)\in\mathbb I_{m,r}$ such that $N$ does not divide $c_{i_0, j_0}$, so $\gcd(c_{i_0, j_0}, N)=q$ as $N=pq$ and $c_{i_0, j_0}\in q\mathbb{Z}$. So part (3) holds.
\end{proof}


\section{On properties ($\flat$) and ($\natural$)}\label{S5}

We introduce the following property on $N\in\mathbb T$.

\begin{De}\label{D1}
	Let $N\in\mathbb T$, $a\in\mathbb Z$, and $m\in \mathbb{N}\setminus \{1, 2\}$. We say that $N$ has property $(\flat_{m;a})$ if there exist $r\in \llbracket2,m-1\rrbracket$ such that $(X+Y-a)^N \not\equiv 0 \pmod{I_{m,r,N; a}}$ and $r$ is the residue of a prime divisor of $N$ modulo $m$.
\end{De}

\begin{Le}\label{L6} Let the pairs $(m, N)\in \mathbb{N}^2$ and $(r, r')\in \llbracket1, m-1\rrbracket^2$ be such that $\gcd (r, m)=1$ and $m\mid rr'-1$. Then there exists an isomorphism 
$$R_{m, r, N; a}=\mathbb{Z}[X, Y]/I_{m, r, N; a}\cong \mathbb{Z}[X]/\bigl(\Psi_m(X), \Psi_m(a-X^r), N\bigr)$$ 
induced by the homomorphism $\mathbb{Z}[X, Y]\rightarrow\mathbb{Z}[X]$ given by the rule $(X, Y)\rightarrow\bigl(X, (a-X^r)^{r'}\bigr)$. 
\end{Le}
\begin{proof}
	By Theorem \ref{TH1}(1) we have an isomorphism $R_{m, 1, N; a}\cong R_{m, r, N; a}$ induced by the endomorphism of $\mathbb Z[X,Y]$ given by the rule $(X,Y)\rightarrow (X^r,Y^r)$. Note that its inverse $R_{m, r, N; a}\cong R_{m, 1, N; a}$ is the isomorphism induced by the endomorphism of $\mathbb Z[X,Y]$ given by the rule $(X, Y)\rightarrow (X^{r'}, Y^{r'})$. By Remark \ref{R1} we have an isomorphism
	$$R_{m, 1, N; a} \cong \mathbb{Z}[X]/\bigl(\Psi_m(X), \Psi_m(a-X), N\bigr)$$ 
	induced by the homomorphism $\mathbb{Z}[X, Y]\rightarrow\mathbb{Z}[X]$ given by the rule $(X, Y)\rightarrow (X, a-X)$.
	
	Thus, we obtain an isomorphism $R_{m, r, N; a}\cong \mathbb{Z}[X]/\bigl(\Psi_m(X), \Psi_m(a-X), N\bigr)$ induced by the homomorphism $\mathbb{Z}[X, Y]\rightarrow\mathbb{Z}[X]$ given by the rule $(X, Y)\rightarrow (X^{r'}, (a-X)^{r'})$.
	
	We also have an isomorphism 
	\begin{equation}\label{EQ11.1}
		\mathbb{Z}[X]/\bigl(\Psi_m(X),\Psi_m(a-X),N\bigr) \cong \mathbb{Z}[X]/\bigl(\Psi_m(X), \Psi_m(a-X^r), N\bigr)
	\end{equation} 
	induced by the endomorphism of $\mathbb Z[X]$ given by the rule $X\rightarrow X^r$. So the homomorphism $\mathbb Z[X,Y]\rightarrow\mathbb Z[X]$ given by the rule $(X, Y)\rightarrow\bigl(X^{r'r}, (a-X^r)^{r'}\bigr)$ defines an isomorphism 
	$$R_{m, r, N; a}\cong \mathbb{Z}[X]/\bigl(\Psi_m(X), \Psi_m(a-X^r), N\bigr).$$
	As $X^{r'r}\equiv X \pmod{\Psi_m(X)}$, the lemma holds.
\end{proof}

\begin{Pro}\label{P4} Let $N\in\mathbb{T}$, $a\in \mathbb{Z}$, and $m\in \mathbb{N}\setminus \{1, 2\}$ be such that $\gcd(m, N)=1$. Then $N$ has property $(\flat_{m; a})$ iff there exists a prime $p\in\Div(N)\setminus\Div(m)$ such that for $r\in \llbracket2,m-1\rrbracket$ with $m\mid p-r$ and $r'\in \llbracket1,m-1\rrbracket$ with $m\mid rr'-1$, in $\mathbb Z[X,Y]$ we have the noncongruence 
\begin{equation}\label{EQ11.2}
\bigl(X+(a-X^r)^{r'}-a\bigr)^N\not\equiv 0 \pmod{(\Psi_m(X), \Psi_m(a-X^r), N)}.
\end{equation}
\end{Pro}
\begin{proof} By definition, $N$ has property $(\flat_{m; a})$ iff there exists a prime divisor $p$ of $N$ such that for $r\in \llbracket 2, m-1\rrbracket$ with $m\mid p-r$ we have $(X+Y-a)^N\not\equiv 0\pmod{I_{m, r, N;a}}$. Applying the isomorphism of Lemma \ref{L6}, we get that Noncongruence (\ref{EQ11.2}) holds. 
\end{proof}

In the next proposition we provide a concrete criterion, based on the existence of nontrivial solutions to Fermat-type equations over finite fields, on when $N$ has property $(\flat_{m;a})$. This allows us to connect certain algebraic properties of finite prime fields to the factorization of integers.

\begin{Pro}\label{P5}
Let $N\in\mathbb{T}$, $a\in \mathbb{Z}$, and $m\in \mathbb{N}\setminus \{1, 2\}$. Let $p$ be a prime divisor of $N$ such that $p=mn+r$ with $r\in \llbracket1, m-1\rrbracket$ and $n\in\mathbb N\cup\{0\}$. 
	Let $q$ be a prime divisor of $N$ which does not divide $m$. Let $F$ be a finite field of characteristic $q$ which contains a primitive $m$-th root of unity and let $G$ be the subgroup of the group of units of $F$ generated by it. Assume that Equation (\ref{EQ1.25}) has a solution $(X_1, Y_1)\in G^2$ such that $X_1 + Y_1 \neq a$. Then
	\[
	\left( X + Y - a \right)^N \not\equiv 0 \pmod{I_{m,r,N; a}}
	\]
and thus $N$ has property $(\flat_{m;a})$.	
\end{Pro}
\begin{proof} Let $I^F_{m,r; a}$ be the ideal of $F[X, Y]$ generated by $X^m-1$, $Y^m-1$ and $X^{r}+Y^{r}-a$. As the $\mathbb F_q$-algebra $\mathbb F_q[X]/(X^m-1)$ is \'etale, its nilpotent radical is $\{0\}$. The same holds for quotients of it. Based on this and Lemma \ref{L6} it follows that the $F$-algebra $F[X,Y]/(I^F_{m,r; a})$ is also \'etale and hence $I^F_{m,r; a}$ is a radical ideal of $F[X,Y]$.

As $q$ divides $N$, to prove the proposition it suffices to show that 
	$(X+Y-a)^N \notin I^F_{m,r; a}$. 
	
	Let $V(I^F_{m,r;a}):=\{(X, Y)\in F\times F \mid f(X, Y)=0 \text{ for all } f \in I^F_{m,r; a}\}$ be the null set in $F^2$ of the ideal $I^F_{m,r;a}$. We have $(X_1, Y_1)\in V(I^F_{m,r; a})$. As $X_1+Y_1\neq a$, $(X_1, Y_1)$ does not belong to the null set in $F^2$ of $(X + Y-a)^N$. So $(X + Y - a)^N \notin \sqrt{I^F_{m,r; a}}=I^F_{m,r; a}$ and the proposition holds.
\end{proof}

Now we define a property $(\natural_{m; a})$ which is stronger than property $(\flat_{m; a})$.

\begin{De}\label{D2}
	Let $(a, N)\in \mathbb{Z}\times \mathbb T$. 
	
	\medskip
	{\bf (1)} Let $p$ and $q$ be distinct prime divisors of $N$ and $m\in\mathbb N \setminus\{1, 2\}$. We say that $N$ has property $(\natural)$ with respect to $(p,q,m;a)$ if $p\nmid m$ and Equation (\ref{EQ1.25}) has a solution $(X_1, Y_1)$ formed by $m$-th roots of unity in $\overline{\mathbb F_q}$ with $X_1+ Y_1 \neq a$.
	
	\smallskip
	{\bf (2)} Let $p$ be a prime divisor of $N$ and $m\in\mathbb N \setminus\{1, 2\}$. We say that $N$ has property $(\natural)$ with respect to $(p,m; a)$ if there exists a prime divisor $q$ of $N$ 
	which is distinct from $p$ and such that $N$ has property $(\natural)$ with respect to $(p,q,m; a)$.
	
	\smallskip
	{\bf (3)} We say that $N$ has property $(\natural)$ if there exists $a\in \mathbb{Z}$ and $m\in\mathbb N\setminus\{1, 2\}$ such that $N$ has property $(\natural)$ with respect to $(p,q,m; a)$ for a pair $(p,q)$ of distinct prime divisors of $N$. In this case, we shall also refer to $N$ as having property $(\natural_{m; a})$. 
\end{De}

In the next section we show that for a fixed $m\in \mathbb{N}$ one can find a proper divisor of $N$ in polynomial time in $m$, $a$, and $\log N$ for all $N\in\mathbb T$ that have property $(\natural_{m;a})$.

There exist many studies on vanishing sums of roots of unity in the literature (e.g., see \cite{LL1}, Main Theorem over $\mathbb C$ and see \cite{LL}, Theorem 1.3 over finite fields) and one hopes to use some of their ideas for studying properties $(\flat_{m; a})$ and $(\natural_{m; a})$.

For $N\in\mathbb N$, if $N$ has property $(\natural_{m; a})$, then $N$ also has property $(\flat_{m; a})$ by Proposition \ref{P5}. 

Recall $U_m$ is the multiplicative group of units of $\mathbb Z/m\mathbb Z=\{[0]_m,\ldots,[m-1]_m\}$. 

\begin{Pro}\label{P6}
Let $a\in\mathbb Z$, $m\in\mathbb N$, and $q$ a prime number such that $q\nmid m$ and $q\mid\Delta_{m;a}$. Let $\zeta\in\overline{\mathbb F_q}$ be a primitive $m$-th roots of unity. Recall $\Stab_{m,q,\zeta;a}$ is the stabilizer of the subset $\Omega_{m,q,\zeta;a}$ of $\mathbb I_m$ with respect to the action $U_m\times\mathbb I_m\rightarrow\mathbb I_m$. Then the following properties hold.
	
\medskip
	{\bf (1)} Let $N\in q\mathbb N$ be such that it has a prime divisor $p$ with $p\nmid mq$. Then $N$ has property $(\natural)$ with respect to $(p,q,m;a)$ iff $[p]_m\notin\Stab_{m,q,\zeta;a}$.
	
\smallskip
	{\bf (2)} Assume $q\mid 2^m-1$, i.e., q is odd and $o_q(2)\mid m$. Then $\Stab_{m,q,\zeta;1}$ is a subgroup of the kernel $\Ker_{m,o_q(2)}$ of the reduction modulo $o_q(2)$ epimorphism $U_m\rightarrow U_{o_q(2)}$. In particular, we have that the order $|\Stab_{m,q,\zeta;1}|$ divides $|\Ker_{m,o_q(2)}|=\frac{\phi(m)}{\phi\bigl(o_q(2)\bigl)}$. 
	
\smallskip
	{\bf (3)} Assume $q$ is odd and $m=o_q(2)$. Then $\Stab_{m,q,\zeta;1}=\{[1]_m\}$ and for each prime $p$ with $p>m$ and $m\nmid p-1$, every $N\in\mathbb N$ with $pq \mid N$ has property $(\natural)$ with respect to $(p, q, m;1)$. 
	
	\smallskip
	{\bf (4)} Suppose that $t_{m,q;1}\in\llbracket1,\phi(m)-1\rrbracket$ and the intersection $\Omega_{m,q,\zeta;1}\cap\mathbb U_m$ is nonempty. Let $$d_{m,q}:=\max\bigl(\Div(\phi(m))\cap [1,t_{m,q;1}]\bigr)\in\mathbb N.$$ 
	Then $d_{m,q}\in \Big[1,\frac{\phi(m)}{2}\Big]$ and there exist distinct elements $[r_1]_m,\ldots,[r_{\phi(m)-d_{m,q}}]_m\in U_m$ such that if $[p]_m=[r_i]_m$ for some 
	$i\in \llbracket1, \phi(m)-d_{m,q}\rrbracket$, then $N$ has property $(\natural)$ with respect to $(p,q,m;1)$.	
\end{Pro}

\begin{proof}
	Recall that $p\mid\Delta_{m;a}$ implies that $|\Omega_{m,q,\zeta;a}|=t_{m,q;a}\ge 1$ by Corollary \ref{C2}(1) and (3).
	
	For part (1) we use the contrapositive. So, $N$ does not have property $(\natural)$ with respect to $(p,q,m;a)$ iff for each solution of the equation $X_1^p + Y_1^p = a$ formed by $m$-th roots of unity in $\overline{\mathbb F_q}$ we have $X_1 + Y_1 = a$. So using $X_1=\zeta^i$ and $Y_1=\zeta^j$ with $(i,j)\in\mathbb I_m$ we get that $N$ does not have property $(\natural)$ with respect to $(p,q,m;a)$ iff for $(i,j)\in\mathbb I_m$, $[p]_m\cdot (i,j)\in\Omega_{m,q,\zeta;a}$ implies that $(i,j)\in\Omega_{m,q,\zeta;a}$, and hence iff $[p]_m^{-1}\cdot\Omega_{m,q,\zeta;a}\subset \Omega_{m,q,\zeta;a}$. As this is equivalent to $[p]_m^{-1}\cdot\Omega_{m,q,\zeta;a}=\Omega_{m,q,\zeta;a}$ and hence to $[p]_m\cdot\Omega_{m,q,\zeta;a}=\Omega_{m,q,\zeta;a}$, we conclude that $N$ does not have property $(\natural)$ with respect to $(p,q,m;a)$ iff $[p]_m\in\Stab_{m,q,\zeta;a}$. So part (1) holds.
	
	For part (2), let $s\in\mathbb N$ be such that $m=o_q(2)s$. We can assume that $\zeta$ is such that $\zeta^{-s}$ is the image of $2$ in $\mathbb F_q$. Thus $(s,s)\in\Omega_{m,q,\zeta;1}$. As $\zeta^s$ is the only solution of the equation $X+X=1$ in $\mathbb F_q$, we have $\Omega_{m,q,\zeta;1}\cap\{(i,i)|i\in \llbracket0,m-1\rrbracket=\{(s,s)\}$. Thus $\Stab_{m,q,\zeta;1}$ stabilizes $\{(s,s)\}$ and hence $\Stab_{m,q,\zeta;1}$ is a subgroup of 
	$$\{[u]_m\in U_m|m\mid us-s\}= \{[u]_m\in U_m|o_q(2)\mid u-1\},$$ i.e., of $\Ker_{m,o_q(2)}$. From this and Lagrange Theorem we get that part (2) holds.
	
	For part (3), we have $s=1$ and thus $|\Stab_{m,q,\zeta;1}|\mid 1$ by part (2). So $\Stab_{m,q,\zeta;1}=\{[1]_m\}$. From this and part (1) we get that part (3) holds.
	
	For part (4), as $|\Omega_{m,q,\zeta;1}|=t_{m,q;1}$ by Corollary \ref{C2}(3), we have $1\le |\Omega_{m,q,\zeta;1}|<\phi(m)$. As the action of $U_m$ on $\mathbb U_m$ is free and $\Omega_{m,q,\zeta;1}\cap\mathbb U_m\neq\emptyset$, we have $|\Stab_{m,q,\zeta;1}|\le t_{m,q;1}$. As $|\Stab_{m,q,\zeta;1}|\mid\phi(m)$ by Lagrange Theorem, we have $|\Stab_{m,q,\zeta;1}|\in\Div\bigl(\phi(m)\bigr)\cap [1,t_{m,q;1}]$; thus $|\Stab_{m,q,\zeta;1}|\le d_{m,q}$. As the largest divisor of $\phi(m)$ in $\llbracket1,\phi(m)-1\rrbracket$ is at most $\frac{\phi(m)}{2}$, we have $d_{m,q}\le \frac{\phi(m)}{2}$.  Thus $|U_m\setminus\Stab_{m,q,\zeta;1}|\ge\phi(m)-d_{m,q}\ge \frac{\phi(m)}{2}$; so there exist distinct elements $[r_1]_m,\ldots,[r_{\phi(m)-d_{m,q}}]_m\in U_m$ such that $\{[r_1]_m,\ldots,[r_{\phi(m)-d_{m,q}}]_m\}\subset U_m\setminus\Stab_{m,q,\zeta;1}$.  
	
	If $[p]_m=[r_i]_m$ for some index $i\in \llbracket1, \phi(m)-d_{m,q}\rrbracket$, then $[p]_m\notin\Stab_{m,q,\zeta;1}$, hence $N$ has property $(\natural)$ with respect to $(p,q,m;1)$ by part (1). So part (4) holds.
\end{proof}

Next we exemplify Proposition \ref{P6}(1) and (2).

\begin{Ex}\normalfont\label{EX3}
	Let $(k,l)\in\mathbb N^2$. With $m:=kl+1$, let $q$ be a prime divisor of $2^m-1$. Let $p$ be a prime with $p\equiv k \pmod{m}$. Let $N\in \mathbb{N}$ be such that $pq \mid N$. Let $x$ be the image of $2^l$ in $\mathbb F_q$. The triple $(x, x, 1)\in \mathbb{F}_q^3$ satisfies Equation (\ref{EQ00}), i.e., we have 
	$$(2^l)^k+(2^l)^k=2^{kl}+2^{kl}=2^{kl+1}=2^m\equiv 1 \pmod{q}.$$ 
	If $2^{l+1}\not\equiv 1 \pmod{m}$, then this triple does not satisfy the equation $X+Y\neq Z$ and thus $N$ has property $(\natural)$ with respect to $(p, q, m; 1)$. From this and Proposition \ref{P6}(1) we get that we have $[k]_m\notin\Stab_{m,q,\zeta;1}$ for each primitive $m$-th root of unity $\zeta\in\overline{\mathbb F_q}$. Note that, as 
	$$(2^{l+1})^k=2^{kl+k}=2^{m-1+k}\equiv 2^{k-1} \pmod{m},$$ 
	if $2^{k-1}\not\equiv 1 \pmod{m}$, then $2^{l+1}\not\equiv 1 \pmod{m}$. 
\end{Ex}	

Next we prove a converse of Proposition \ref{P5} for a square free $N\in\mathbb T$.

\begin{Pro}\label{P7}
	Let $N\in \mathbb{T}$ be square free, $a\in \mathbb{Z}$, and $m\in \mathbb{N}\setminus\{1, 2\}$.
	Then $N$ has property $(\natural_{m; a})$ iff $N$ has property $(\flat_{m; a})$.
\end{Pro}
\begin{proof}
	The `only if' part was proved in Proposition \ref{P5}. We now proof the `if' part.
	
	Suppose that $(X+Y-a)^N\not\equiv 0 \pmod{I_{m, N, r; a}}$, where $r\in \llbracket2, m-1\rrbracket$ is the residue of a prime divisor $p$ of $N$ modulo $m$.
	By Proposition \ref{P4} we have $(X+(a-X^r)^{r'}-a)^N\not\equiv 0 \pmod{(\Psi_m(X), \Psi_m(a-X^r), N)}$, where $r'\in \llbracket2, m-1\rrbracket$ is such 
	that $m\mid rr'-1$. As $N$ is square free, by the Chines Remainder Theorem we have an isomorphism of rings
	$$(\mathbb{Z}/N\mathbb{Z})[X]\cong (\mathbb{Z}/p_1\mathbb{Z})[X] \times \cdots \times (\mathbb{Z}/p_k\mathbb{Z})[X]=\prod_{i=1}^k \mathbb F_{p_i}[X],$$ where $p_1, \ldots, p_k$
	are the distinct prime divisors of $N$. Therefore, there exists $q\in \{p_1, \ldots, p_k\}$ such that 
	$(X+(a-X^r)^{r'}-a)^N\not\equiv 0 \pmod{(\Psi_m(X), \Psi_m(a-X^r), q)}$. In $\mathbb{F}_q[X]$ we have that $\theta_{m,q,r;a}(X)\mid X^m-1$, $\theta_{m,q,r;a}(X)\mid (a-X^r)^m-1$ and 
	$\theta_{m,q,r;a}(X)\nmid (X+(a-X^r)^{r'}-a)^N$. Thus, there exists a root $X_1$ of $\theta_{m,q,r;a}(X)$ in $\overline{\mathbb F_q}$ such that $X_1+(a-X_1^r)^{r'}-a\neq 0$. 
	The mentioned divisibilities give that $X_1^m-1=(a-X_1^r)^m-1=0$. Thus, $X_1$ and $a-X_1^r$ are $m$-th roots
	of unity in $\overline{\mathbb{F}_q}$ such that $X_1+(a-X_1^r)^{r'}\neq a$. 
	
	We are left to show that 
	$X_1^p+[(a-X_1^r)^{r'}]^p=a$ and $p\neq q$. Indeed, as $p\equiv r \pmod{m}$ and $rr'\equiv 1\pmod{m}$, we have 
	$X_1^p+[(a-X_1^r)^{r'}]^p=X_1^r+(a-X_1^r)^{r'r}=X_1^r+a-X_1^r=a$. If $q=p$, then as $X_1+(a-X_1^r)^{r'}-a\neq 0$, by Frobenius isomorphism we get
	that 
	$$0\neq (X_1+(a-X_1^r)^{r'}-a)^q =X_1^q+[(a-X^r)^{r'}]^q-a=X_1^p+[(a-X^r)^{r'}]^p-a=0,$$ a contradiction.
\end{proof}

\section{Existence of a polynomial-time algorithm under property $(\natural)$}\label{S6}

Let $M:\mathbb N\to (0,\infty)$ be a multiplication time function for $\mathbb Z$ (cf.\ \cite{GG}, Definition 8.26 and \cite{HvdH}, Section 1). Similarly, let $M_{\mod}:\mathbb N\to (0,\infty)$ be a multiplication time modulo function for $\mathbb Z$. So for $n\in\mathbb N$, two integers of bit length at most $n$ can be multiplied using at most $M(n)$ word operations and for each 
$N\in \mathbb N$ one multiplication in $\mathbb Z/N\mathbb Z$ can be performed using $M_{\mod}(N)$ word operations. 

We have identities $M(n)=O (n\log n)$ (see \cite{HvdH}, Theorem 1.1) and 
$$M_{\mod}(N)=O\bigl(M(\lfloor \log_2 N\rfloor +1)\bigr)$$
(see \cite{GG}, Corollary 9.9).

Let $b\in \llbracket2, N-1\rrbracket$. Using the Euclidean Algorithm one can find $\gcd(b, N)$. By Lamé Theorem (see \cite{L}) the number of division steps in the Euclidean Algorithm is bounded by 
$\frac{\log(\sqrt{5}N)}{\log \frac{1+\sqrt{5}}{2}}$. As division of two integers bounded by $N$ costs $O\bigl(M(\lfloor \log_2 N\rfloor +1)\bigr)$ (e.g., see \cite{GG}, Theorem 9.8),
the cost of finding $\gcd(b, N)$ is $O\bigl(M(\lfloor \log_2 N \rfloor +1) \log N\bigr)$. By the same argument, if $\gcd(b, N)=1$, then using the Extended Euclidean Algorithm we can find the inverse of $a$ modulo $N$ at the cost of $O\bigl(M(\lfloor \log_2 N \rfloor +1) \log N\bigr)$. 

We have the following estimate of costs in Division Algorithms in one variable modulo $N$.

\begin{Le}\label{L7}
Let $(c,d)\in\mathbb N^2$ with $c\ge d$. Let 
$$D_{c,d}:=d(c+1-d).$$ 
Let $f(X)\in (\mathbb{Z}/N\mathbb{Z})[X]$ be a polynomial of degree $c$. Let $g(X)\in (\mathbb{Z}/N\mathbb{Z})[X]$ be a monic polynomial of degree $d$. Let $\bigl(h(X), g_1(X)\bigr)\in (\mathbb{Z}/N\mathbb{Z})[X]^2$ be the unique pair with the properties that we have $f(X)=h(X)g(X)+g_1(X)$ and $\deg(g_1) < d$. Then the cost of finding $g_1(X)$ is bounded by
$$D_{c,d}\left[M_{\mod}(N)+O(\log_2 N)\right].$$
\end{Le}

\begin{proof}
First we need to consider $f_1(X):=f(X)-a_0X^{c-d}g(X)$, where $a_0$ is the leading coefficient of $f(X)$. 
The cost of the multiplication $a_0X^{c-d}g(X)$ is bounded by $dM_{\mod}(N)$, and the cost of subtraction $f(X)-a_0X^{c-d}g(X)$ is
$dO\bigl(\log_2 N\bigr)$ (cf.\ the multi-precision package described in \cite{Coh}, Chapter 1, Page 3). Thus the cost of finding $f_1(X)$ is bounded by $d\left[M_{\mod}(N)+O(\log_2 N)\right]$.
If $\deg(f_1) < d$, then $g_1(X)=f_1(X)$. 

If $\deg(f_1)\ge d$, then let $f_2(X):=f_1(X)-a_1X^{\deg(f_1)-d}g(X)$, where $a_1$ is the leading coefficient of $f_1(X)$.
By the same argument the cost of finding $f_2(X)$ is bounded by $d\left[M_{\mod}(N)+O(\log_2 N)\right]$. If $\deg(f_2) < d$, then $g_1(X)=f_2(X)$. 

If $\deg(f_2)\ge d$, then let $f_3(X):=f_2(X)-a_2X^{\deg(f_2)-d}g(X)$, where $a_2$ is the leading coefficient of $f_2(X)$, and so on. 

As $c>\deg(f_1)>\deg(f_2)>\cdots $, this process terminates after at most $c+1-d$ steps and the lemma follows.
\end{proof}

\begin{Pro}[Dichotomic Euclidean Algorithm in one variable modulo $N$]\label{P8}
Let the pair $(c,d)\in\mathbb N^2$ be such that $c\ge d$. Let $g_{-1}(X)\in (\mathbb{Z}/N\mathbb{Z})[X]$ 
	be a polynomial of degree $c$ and let $g_0(X)\in (\mathbb{Z}/N\mathbb{Z})[X]$ be a monic polynomial of degree $d$. Let $\iota\in \llbracket1, d\rrbracket$ be the unique integer such that there exists a $\iota$-tuple $(g_1,\ldots,g_{\iota})\in (\mathbb{Z}/N\mathbb{Z})[X]^{\iota}$ with the property that the leading coefficients of $g_1,\ldots,g_{\iota-1}$ are units of $\mathbb Z/N\mathbb Z$, either $g_{\iota}=0$ or the leading coefficient of $g_{\iota}$ is not a unit of $\mathbb Z/N\mathbb Z$, and for each $i\in \llbracket1,\imath \rrbracket$, $g_i$ is the remainder of the division of $g_{i-2}(X)$ by $g_{i-1}(X)$ (i.e., there exists a unique $h_i\in (\mathbb{Z}/N\mathbb{Z})[X]$ such that $g_{i-2}(X)=g_{i-1}(X)h_i(X)+g_i(X)$ and $\deg(g_i)<\deg(g_{i-1})$). For $e\in\mathbb N$ with $e\ge d$ let 
	$$B_{e,d,\iota}:=d\left(e+\iota-\frac{d(\iota+1)}{2\iota}\right)-\frac{\iota (\iota-1)}{2}\in [d(e+1-d),de]\cap\mathbb Q.$$ Then the following properties hold. 
	
	\medskip
	{\bf (1)} If $g_{\iota}=0$, then $\bigl(g_{-1}(X), g_0(X)\bigr)=\bigl(g_{\iota-1}(X)\bigr)$ and the cost of finding $g_{\iota-1}$ is bounded by 
	$$\left(B_{c+\iota-1,d,\iota}-\frac{\iota (\iota-1)}{2}\right)M_{\mod}(N)+B_{c,d,\iota}O(\log_2 N)+O\bigl(2(\iota-1) M(\lfloor \log_2 N \rfloor +1) \log N\bigr).$$
		
		\smallskip
	{\bf (2)} If $g_{\iota}\neq 0$, then the cost of finding a proper divisor of $N$ is bounded by
	$$\left(B_{c+\iota-1,d,\iota}-\frac{\iota (\iota-1)}{2}\right)M_{\mod}(N)+B_{c,d,\iota}O(\log_2 N)+O\bigl((2\iota-1) M(\lfloor \log_2 N \rfloor +1) \log N\bigr).$$
\end{Pro}
\begin{proof} As $\deg(g_1)<d$ and $g_{-1}(X)=g_0(X)h_1(X)+g_1(X)$, by Lemma \ref{L7} the cost of finding $g_1(X)$ is bounded by 
	$D_{c,d}\left[M_{\mod}(N)+O(\log_2 N)\right]$. 
	
	If $g_1(X)=0$, then $\bigl(g_{-1}(X), g_0(X)\bigr)=\bigl(g_0(X)\bigr)$.
	
	Assume that $g_1(X)\neq 0$. Let $a_1\in \llbracket1, N-1\rrbracket$ be such that $a_1+N\mathbb{Z}$ is the leading coefficient of $g_1(X)$. If $\gcd(a_1, N)>1$, then $\gcd(a_1, N)$ 
	is a proper divisor of $N$; the cost of finding $\gcd(a_1, N)$ is $O\bigl(M(\lfloor \log_2 N \rfloor +1) \log N\bigr)$.
	If $\gcd(a_1, N)=1$, then $a_1+N\mathbb{Z}$ is invertible in $\mathbb{Z}/N\mathbb{Z}$ and we can find its inverse $(a_1+N\mathbb{Z})^{-1}$ 
	at the cost $O\bigl(M(\lfloor \log_2 N \rfloor +1) \log N\bigr)$. 
	Then we calculate $(a_1+N\mathbb{Z})^{-1}g_1(X)$ to get a monic polynomial at the cost bounded by $\deg(g_1)M_{\mod}(N)$. 
	As $\deg(g_2)<\deg(g_1)$ and $g_0(X)=g_1(X)h_2(X)+g_2(X)$, by Lemma \ref{L7} the cost of finding $g_2(X)$, after having the monic polynomial 
	$(a_1+N\mathbb{Z})^{-1}g_1(X)$, is bounded by $D_{d,\deg(g_1)}\left[M_{\mod}(N)+O(\log_2 N)\right]$.
	
	If the leading coefficient of $g_2(X)$ is invertible, then we can continue the process.
	
	In this way we get uniquely the sequence of polynomials $g_1(X), g_2(X), \ldots, g_{\imath}(X)$ in $(\mathbb{Z}/N\mathbb{Z})[X]$.
	
	If $g_{\iota}(X)=0$, then by induction on $i\in \llbracket0, \iota -1\rrbracket$ we get that 
	$$\bigl(g_0(X), g_1(X)\bigr)=\bigl(g_i(X),g_{i+1}(X)\bigr)$$ 
	and hence $\bigl(g_{-1}(X), g_0(X)\bigr)=\bigl(g_{\iota-1}(X)\bigr)$. 
	If $g_{\iota}(X)\neq 0$, then if $a_{\iota}\in\mathbb \llbracket1, N-1\rrbracket$ is such that $a_{\iota}+N\mathbb{Z}$ is the leading coefficient of $g_{\iota}(X)$, then $\gcd(a_{\iota}, N)$ is a proper divisor of $N$.
	
	Thus, by denoting
	$$
	E_{N,g_{-1},g_0}:=\overset{\iota-1}{\underset{i=0}{\sum}}D_{\deg(g_{i-1}),\deg(g_i)}\left[M_{\mod}(N)+O(\log_2 N)\right] +
	M_{\mod}(N)\overset{\iota-1}{\underset{i=1}{\sum}}\deg(g_i),$$
	if $g_{\iota}=0$ then the total cost is bounded by 
	$E_{N,g_{-1},g_0}+O\bigl(2(\iota-1)M(\lfloor \log_2 N \rfloor +1) \log N\bigr)$, and if $g_{\iota}\neq 0$ then the total cost is bounded by $E_{N,g_{-1},g_0}+O\bigl((2\iota-1)M(\lfloor \log_2 N \rfloor +1) \log N\bigr)$. As for $\iota=1$ we have $E_{N,g_{-1},g_0}=D_{c,d}\left[M_{\mod}(N)+O(\log_2 N)\right]$, parts (1) and (2) hold for $\iota=1$.
	
	We are left to prove parts (1) and (2) for $\iota\in \llbracket2, d\rrbracket$. Denoting 
	$$A_{g_{-1},g_0}:=\overset{\iota-1}{\underset{i=0}{\sum}}D_{\deg(g_{i-1}),\deg(g_i)}=\overset{\iota-1}{\underset{i=0}{\sum}}[\deg(g_{i-1})+1-\deg(g_i)]\deg(g_i),$$
	$$B_{g_{-1},g_0}:=\overset{\iota-1}{\underset{i=1}{\sum}}\deg(g_i),$$
	$$C_{g_{-1},g_0}:=A_{g_{-1},g_0}+B_{g_{-1},g_0},$$
we have an identity
\begin{equation}\label{EQ12}
	E_{N,g_{-1},g_0}=A_{g_{-1},g_0}O(\log_2 N)+C_{g_{-1},g_0}M_{\mod}(N).
\end{equation}	
Denoting $$D_{g_{-1},g_0}:=\overset{\iota-1}{\underset{i=1}{\sum}}[\deg(g_{i-1})-\deg(g_i)]\deg(g_i),$$
we have an identity 
\begin{equation}\label{EQ13}
	A_{g_{-1},g_0}=
	d(c+1-d)+B_{g_{-1},g_0}+D_{g_{-1},g_0}.
\end{equation}
So to estimate $A_{g_{-1},g_0}$ and $C_{g_{-1},g_0}$ it suffices to estimate $B_{g_{-1},g_0}$ and $D_{g_{-1},g_0}$.

As $\deg(g_i)\le d-i$ for each $i\in \llbracket1, \iota\rrbracket$, we estimate
\begin{equation}\label{EQ14}
	 B_{g_{-1},g_0}\leq \overset{\iota-1}{\underset{i=1}{\sum}}(d-i)
	=d(\iota-1) - \frac{\iota (\iota-1)}{2}.
\end{equation}

Defining $b_1:=\deg(g_{\iota-1})\in\mathbb N\cup\{0\}$ and $b_{i+1}:=\deg(g_{\iota-i-1})-\deg(g_{\iota-i})\in\mathbb N$ for $i\in \llbracket1, \iota -1\rrbracket$, we have an identity $\overset{\iota}{\underset{i=1}{\sum}} b_i=\deg(g_0)=d$ and we estimate
\begin{equation}\label{EQ15}
D_{g_{-1},g_0}=\overset{\iota-1}{\underset{i=1}{\sum}} b_{i+1}(\sum_{j=1}^i b_j)=\frac{(\overset{\iota}{\underset{i=1}{\sum}} b_i)^2-\overset{\iota}{\underset{i=1}{\sum}} b_i^2}{2}=\frac{d^2-\overset{\iota}{\underset{i=1}{\sum}} b_i^2}{2}\le \frac{d^2(1-\frac{1}{\iota})}{2}
\end{equation}
by Cauchy--Schwartz' inequality. 

Equation (\ref{EQ13}) and Inequalities (\ref{EQ14}) and (\ref{EQ15}) give that
\begin{equation}\label{EQ16}
A_{g_{-1},g_0}\le d(c+1-d)+d(\iota-1) - \frac{\iota (\iota-1)}{2}+\frac{d^2(1-\frac{1}{\iota})}{2}=d\left[c+\iota-\frac{d(\iota+1)}{2\iota}\right]-\frac{\iota (\iota-1)}{2}=B_{c,d,\iota}.
\end{equation}
From Inequalities (\ref{EQ14}) and (\ref{EQ16}) we get that
\begin{equation}\label{EQ17}
	C_{g_{-1},g_0}=A_{g_{-1},g_0}+B_{g_{-1},g_0}\le B_{c,d,\iota}+d(\iota-1) - \frac{\iota (\iota-1)}{2}=B_{c+\iota-1,d,\iota}-\frac{\iota (\iota-1)}{2}.
\end{equation}
Equation (\ref{EQ12}) and Inequalities (\ref{EQ16}) and (\ref{EQ17}) give that parts (1) and (2) hold for $\iota\in \llbracket2, d\rrbracket$.\footnote{As $B_{g_{-1},g_0}=\overset{\iota-1}{\underset{i=1}{\sum}} b_i(\iota-i)$, the expression $B_{g_{-1},g_0}+D_{g_{-1},g_0}-\frac{d^2}{2}$ (resp.\ $2B_{g_{-1},g_0}+D_{g_{-1},g_0}-\frac{d^2}{2}$) is equal to $\overset{\iota-1}{\underset{i=1}{\sum}} b_i(\iota-i)-\frac{1}{2}\overset{\iota}{\underset{i=1}{\sum}} b_i^2$ (resp.\ $\overset{\iota-1}{\underset{i=1}{\sum}} 2b_i(\iota-i)-\frac{1}{2}\overset{\iota}{\underset{i=1}{\sum}} b_i^2$) which, though difficult to estimate in a compact form, can be used to refine Inequality (\ref{EQ16}) (resp.\ Inequality (\ref{EQ17})) in many cases.}
\end{proof}

\begin{De}\label{D3}
Let $(N,m,a)\in\mathbb N^2\times \mathbb{Z}$. We say that $N$ has property $(\sharp_{m; a})$ if there exists no monic polynomial in $(\mathbb Z/N\mathbb Z)[X]$ which generates the ideal $\bigl(\Psi_m(X)+(N),\Psi_m(a-X)+(N)\bigr)$ of $\mathbb (\mathbb Z/N\mathbb Z)[X]$. 
\end{De}

Note that for $(N,m, a)\in\mathbb N^2\times \mathbb{Z}$, $N$ has property $(\sharp_{m;a})$ iff in Proposition \ref{P8} for the reductions $g_{-1}(X)$ and $g_0(X)$ of $\Psi_m(a-X)$ and $\Psi_m(X)$ (respectively) modulo $N$, we have $g_{\iota}\neq 0$. Also, for pairs $(N,m)\in\mathbb N^2$ and $(a,b)\in\mathbb Z^2$ with $a\equiv b\pmod{N}$, $N$ has property $(\sharp_{m;a})$ iff $N$ has property $(\sharp_{m;b})$.

\begin{Ex}\normalfont\label{EX4}
Let $N\in\mathbb T$ be square free and let $m\in\mathbb N$. Then $N$ has property $(\sharp_{m;a})$ iff the polynomial $\theta_{m,N,1;a}$, defined in Lemma \ref{L4}, is not monic and hence by Corollary \ref{C1}, iff there exists a pair $(p,q)$ of prime divisors of $N$ such that $t_{m,p;a}<t_{m,q;a}$. In particular, if there exists a pair $(p,q)$ of prime divisors of $N$ such $p\mid\Delta_{m;a}$ and $q\nmid\Delta_{m;a}$, then $t_{m,p;a}=0<t_{m,q;a}$ by Corollary \ref{C2}(1) and therefore $N$ has property $(\sharp_{m;a})$.
\end{Ex}

Let $N\in\mathbb N\setminus\{1\}$. For $s\in \llbracket1, N-1\rrbracket$ with $\gcd(N,s)=1$ we consider the set of orders
$$O_{N,s}:=\{o_p(s)|p\,\textup{is a prime divisor of}\; N\}\subset\mathbb N.$$
For $a\in\mathbb Z\setminus\{0,2\}$, let $b:=6$ if $a=1$, let $b:=3$ if $a=-1$, let $b:=2$ if $a=-2$, and let $b:=0$ if $|a|\ge 3$; we consider the set of minimal elements 
$$M_{N;a}:=\{m_{p;a}|p\,\textup{is a prime divisor of}\; N\}\subset (\mathbb N\cup \{0\})\setminus b\mathbb N.$$

\begin{Pro}\label{P8.1}
Let $N\in \mathbb N\setminus\{1\}$. Then the following properties hold.

\medskip
{\bf (1)} Let $a\in\mathbb Z$ be such that $\gcd(N,a)=1$. Assume $N=2l-1$ is odd and $O_{N,al}$ has at least two elements. Then $N$ has property $(\sharp_{o;a})$ for each $o\in O_{N,al}\setminus\{\max(O_{N,al})\}$.

\smallskip
{\bf (2)} Let $a\in\mathbb Z\setminus\{0,2\}$. Assume $M_{N;a}$ has at least two elements. Let $m_0:=\min(M_{N;a})$. If $m_0>0$ then $N$ has property $(\sharp_{m_0;a})$, and if $m_0=0$ then $N$ has property $(\sharp_{m;a})$ for each $m\in M_{N;a}\cap\mathbb N$.

\smallskip
{\bf (3)} Let $m\in\mathbb N\setminus\{1\}$. Assume $N$ is odd and all its prime divisors are greater than $m$. Let $r\in \llbracket2, m-1\rrbracket$ be relatively prime to $m$. Let 
$s\in \llbracket1, N-1\rrbracket$ be such that $N\mid r-2s$. If $O_{N,s}$ has at least two elements, then $N$ has property $(\sharp_{o;r})$ for each $o\in O_{N,s}\setminus\{\max(O_{N,s})\}$.
\end{Pro}

\begin{proof}
For part (1), we consider prime divisors $p$ and $q$ of $N$ such that $o_p(al)=o<o_q(al)$. So $p\mid 1-(al)^o$ but $q\nmid 1-(al)^o$. As $2l\equiv 1\pmod{N}$, multiplying by $2^o$ we get that $p\mid 2^o-a^o$ but $q\nmid 2^o-a^o$. Thus $t_{o,p;a}$ is odd and $t_{o,q;a}$ is even by Corollary \ref{C2}(4). Hence $t_{o,p;a}\neq t_{o,q;a}$. Thus, if $\bar N\in\mathbb N\setminus\{1\}$ is the largest square free natural number that divides $N$, then $\theta_{m,\bar N,1;a}$ is not monic by Corollary \ref{C1}. Thus $\bar N$ and $N$ have property $(\sharp_{o;a})$ and part (1) holds.

For part (2), let $b\in\{0,2,3,6\}$ be as above and we consider prime divisors $p$ and $q$ of $N$ such that $m_{p;a}=m_0<m_{q;a}$. Thus $m_{q;a}\in\mathbb N\setminus b\mathbb N$. If $m_0=0$, then $b>0$ and $MO_{p;a}\subset b\mathbb N$ and hence $m_{q;a}\in MO_{q;a}\setminus MO_{p;a}$; from Remark \ref{R2}(5) we get that $t_{m_{q;a},p;a}=0$ and $t_{m_{q;a},q;a}\in\mathbb N$ and thus $t_{m_{q;a},p;a}\neq t_{m_{q;a},q;a}$ and as in the prior paragraph we argue that $N$ has property $(\sharp_{m_{q;a};a})$. If $m_0>0$, then $m_0\in \mathbb N\setminus b\mathbb N$ and hence, as $m_0<m_{q;a}$, we get that $m_0\in MO_{p;a}\setminus MO_{q;a}$; so $t_{m_0,p;a}>0=t_{m_0,q;a}$ by Remark \ref{R2}(5) and thus $t_{m_0,p;a}\neq t_{m_0,q;a}$ and as in the prior paragraph we argue that $N$ has property $(\sharp_{m_0;a})$. So part (2) holds.

Part (3) is proved entirely similar to part (1).
\end{proof}

For a prime $p$, an integer $d\in\Div(p-1)$, let $g_d(p)\in\llbracket1,p-1\rrbracket$ be the smallest such $o_p\bigl(g_d(p)\bigr)=d$. Thus $g(p):=g_{p-1}(p)\in\llbracket1,p-1\rrbracket$ is the smallest primitive root modulo $p$. 

\begin{Co}\label{C4.5} Let $p$ and $q$ be two odd primes with $p<q$. Let $N:=pq$. We consider the set 
$$\S_{p,q}:=\{(o,s)\in \bigl(\Div(p-1)\cup\Div(q-1)\bigr)\times \llbracket1,N\rrbracket|[s]_N\in U_N,o_p(s)\neq o_q(s),o=\min\bigl(o_p(s),o_q(s)\bigr)\}$$
and the natural number 
$$\mu_{p,q}:=\min(\{o+2s|(o,s)\in\S_{p,q}\}).$$
Then the following properties hold.

\medskip
{\bf (1)} For each pair $(o,s)\in\S_{p,q}$, $N$ has property $(\sharp_{o;2s})$.

{\bf (2)} We have an inclusion of sets 
$$\{\bigl(1,p+1\bigr),\bigl(2,p-1\bigr)\}\subset\S_{p,q}.$$ 

\smallskip
{\bf (3)} Let $d\in\Div(p-1)\triangle\Div(q-1)$ (e.q., $d$ could be $q-1$). If $d$ divides $p-1$ (resp.\ $q-1$ and $p\nmid g_d(q)$), then $\bigl(\min(d,o_q(g_d(p))),g_d(p)\bigr)\in\S_{p,q}$ (resp.\ $\bigl(\min(d,o_p(g_d(q)),g_d(q)\bigr)\in\S_{p,q}$).

\smallskip
{\bf (4)} We have inequalities $\mu_{p,q}\le 2p\le 2(\sqrt{N+1}-1)$.

\smallskip
{\bf (5)} If $p\nmid g(q)$, then $\mu_{p,q}\le o_p\bigl(g(q)\bigr)+2g(q)\le p-1+2g(q)<p-1+2q^{\frac{5}{8}}$.
\end{Co}

\begin{proof}
Part (1) follows from Proposition \ref{P8.1}(1) applied to $a:=2s$: if $l\in\mathbb N$ is such that $N=2l-1$, then $al=2sl\equiv s\pmod{N}$.

As $o_p(p+1)=1$ but $o_q(p+1)\ge 2$ because $q\ge p+2$ and $o_p(p-1)=2$ but $o_q(p-1)>2$ because $q\nmid (p-1)^2-1=p(p-2)$, part (2) holds.

For part (3), if $d$ divides $p-1$ (resp.\ $q-1$), then $o_p\bigl(g_d(p)\bigr)=d$ while $o_q\bigl(g_d(p)\bigr)\neq d$ as $d\notin\Div(q-1)$ (resp.\ $o_p\bigl(g_d(q)\bigr)\neq d$ as $d\notin\Div(p-1)$ while $o_q\bigl(g_d(q)\bigr)=d$), so part (3) holds.

Part (4) follows from part (2).

As $q\ge 5$, we have $g(q)<q^{\frac{5}{8}}$ by \cite{MS}, Theorem 1.3. Based on this, part (5) follows from part (3) applied to $d=q-1\in\Div(q-1)\setminus\Div(p-1)$ and $g_d(q)=g(q)$.
\end{proof}

\begin{Pro}\label{P9}
Let $(N,m,a)\in\mathbb T\times\mathbb N\times \mathbb{Z}$ be such that $\gcd(m, N)=1$ and $N$ does not have properties $(\sharp_{m;a})$ and $(\flat_{m;a})$. Let $p$ be a prime divisor of $N$ and $(r, r')\in \llbracket1, m-1\rrbracket^2$ be such that $p\equiv r \pmod{m}$ and $m\mid rr'-1$. Then the following properties hold.

\medskip
{\bf (1)} There exists a monic polynomial $\Theta_{m, r, N;a}(X)\in \mathbb{Z}[X]$ such that we have an identity 
$\bigl(\Theta_{m, r, N;a}(X), N\bigr)=\bigl( \Psi_m(X), \Psi_m(a-X^r), N\bigr)$ between ideals of $\mathbb Z[X]$. 

\smallskip
{\bf (2)} If $N$ is square free, then we have $(X+Y)^N-a^N\in I_{m, r, N;a}$.

\smallskip
{\bf (3)} If $N=pq$ is a semiprime and $q\mid a^p-a$ (e.g., this holds if either $a\in\{0,1\}$ or $a=-1$ and $p$ is odd), then there exist integers 
$c_{i,j}\in q\mathbb Z \cap \llbracket1, N-1\rrbracket$ 
indexed by $(i,j)\in\mathbb I_{m,r}$ such that 
$$\bigl(X + (a-X^r)^{r'}\bigr)^N - X^N - (a-X^r)^{r'N} \equiv \sum_{(i,j)\in\mathbb I_{m,r}} c_{i,j} X^i (a-X^r)^{r'j} \pmod{(\Theta_{m, r, N;a}, N)}.$$

\smallskip
{\bf (4)} Let $d:= \deg(\Theta_{m, r, N;a})$. If $N=pq$ is a semiprime and $q\mid a^p-a$, then there exists a unique $d$-tuple 
$(c_0,\ldots,c_{d-1})\in\{0,q,2q,\ldots,N-q\}^{d}$ such that 
$$(X + (a-X^r)^{r'})^N - X^N - (a-X^r)^{r'N} \equiv \overset{d-1}{\underset{i=0}{\sum}}c_{i}X^i \pmod{(\Theta_{m, r, N;a}, N)}.$$
\end{Pro}
\begin{proof}
	Part (1) follows from Isomorphism (\ref{EQ11.1}) and the fact that $N$ does not have property $(\sharp_{m;a})$.
	
	As $N$ does not have property $(\flat_{m;a})$, we have 
	$$\bigl(X+(a-X^r)^{r'}-a\bigr)^N\equiv 0 \pmod{(\Theta_{m, r, N;a}(X), N)}$$ 
	by Proposition \ref{P4}. Let $q$ be a prime divisor of $N$; so $\theta_{m, r, q;a}(X)$ is the reduction of $\Theta_{m, r, N;a}(X)$ modulo $q$. Then in $\mathbb{F}_q[X]$ we have that
	$\theta_{m, r, q;a}(X) \mid \bigl(X+(a-X^r)^{r'}-a\bigr)^N$. As we have $\theta_{m, r, q;a}(X)\mid X^m-1$ and $q\nmid m$, the polynomial $\theta_{m, r, q;a}(X)$
	is separable. Therefore we have $\theta_{m, r, q;a}(X) \mid \bigl(X+(a-X^r)^{r'}-a\bigr)$. Hence
	$\theta_{m, r, q;a}(X) \mid \bigl(X+(a-X^r)^{r'}\bigr)^N -a^N$. Thus 
	$$\bigl(X+(a-X^r)^{r'}\bigr)^N -a^N \equiv 0 \pmod{(\Theta_{m, r, N;a}(X), q)}$$ 
	for each prime divisor $q$ of $N$. As $N$ is square free, we get that 
	$$\bigl(X+(a-X^r)^{r'}\bigr)^N -a^N \equiv 0 \pmod{(\Theta_{m, r, N;a}(X), N)}$$ 
	by the Chinese Reminder Theorem. Taking into account the isomorphism of Lemma \ref{L6}
	we get that $(X+Y)^N-a^N \in I_{m, r, N;a}$, so part (2) holds. 
	
	For part (3), note that $N\in\mathbb T$ and $N=pq$ imply that $p\neq q$. As $q\mid a^p-a$, we get that $N\mid a^N-a^q$. From this, part (2), and Proposition \ref{P3}(2) it follows that there exist unique integers $c_{i,j}\in q\mathbb Z \cap \llbracket0, N-1\rrbracket$ indexed by $(i,j)\in\mathbb I_{m,r}$ such that we have the following congruence $(X + Y)^N - X^N - Y^N \equiv \sum_{(i,j)\in\mathbb I_{m,r}} c_{i,j} X^i Y^j \pmod{I_{m,r,N;a}}$.
	Applying the isomorphism of Lemma \ref{L6} to this congruence, we get that part (3) holds. 
	
	Part (4) follows from part (3) and from the Division Algorithm in $(\mathbb Z/N\mathbb Z)[X]$ applied to the division of the polynomial 
	$[X+(a-X^r)^{r'}]^N-X^N-(a-X^r)^{r'N}+(N)$ by $\Theta_{m,r,N;a}(X)+(N)$.
\end{proof}

\begin{Le}\label{L9}
	Let a prime divisor $p$ of $N$ and a triple $(m, n, r, r')\in\mathbb{N}^4$ be such that $p=mn+r$, $r\in\llbracket1,m-1\rrbracket$, and $m\mid rr'-1$. 
	Let $a\in \mathbb{Z}$. We assume that there exists a monic polynomial $\Theta_{m,r,N; a}(X)\in \mathbb{Z}[X]$ such that 
	$$\bigl(\Psi_m(X),\Psi_m(a-X^r), N\bigr)=\bigl(\Theta_{m,r,N; a}(X), N\bigr).\footnote{For instance, this holds if $N$ is square free and does not have property $(\sharp_{m;a})$ by Corollary \ref{C1}, as $\Theta_{m,r,N;a}(X)$ can be any monic polynomial whose reduction modulo $N$ is $\theta_{m,r,N;a}(X)$.}$$ 
	Let $d:=t_{m,p;a}=\deg(\Theta_{m,r,N;a})\in\mathbb N\cup\{0\}$. Then the following properties hold.
	
	\medskip
	{\bf (1)} 
	There exists a unique $d$-tuple $(b_0,\ldots,b_{d-1})\in \llbracket0, N-1\rrbracket^d$ such that 		
	
	$$
	[X+(a-X^r)^{r'}-a]^N \equiv \overset{d-1}{\underset{i=0}{\sum}}b_{i}X^i \pmod{(\Theta_{m,r,N;a}(X), N)}.
	$$
	
	{\bf (2)} If
	$$
	[X+(a-X^r)^{r'}-a]^N \not \equiv 0\pmod{(\Theta_{m,r,N;a}(X),N)}, 
	$$
	then there exists $i\in \llbracket0, d\rrbracket$ such that $\gcd(b_i, N)$ is a proper divisor of $N$ divisible by $p$.
\end{Le}
\begin{proof}
	Part (1) follows from the Division Algorithm in $(\mathbb Z/N\mathbb Z)[X]$ applied to the division of the polynomial $[X+(a-X^r)^{r'}-a]^N+(N)$ by $\Theta_{m,r,N;a}(X)+(N)$. 
	
	By Lemma \ref{L1} we have
	$$
	[X+(a-X^r)^{r'}-a]^N = [X+(a-X^r)^{r'}-a]^{p(\frac{N}{p})}=[X^p+(a-X^r)^{pr'}-a+pF(X)]^{(\frac{N}{p})},
	$$
	where $F(X)\in \mathbb{Z}[X]$. Let $k\in\mathbb N$ be such that $rr'=1+mk$.
	Moreover, as we have congruences $X^{mr} \equiv 1 \pmod{\Psi_m(X)}$ and $(a-X^r)^{mnr'+km} \equiv 1 \pmod{\Psi_m(a-X^r)}$, we compute
	$$
	X^p+(a-X^r)^{pr'}=X^{mn+r}+(a-X^r)^{(mn+r)r'}=X^{mn+r}+(a-X^r)^{mnr'+1+km} 
	$$
	$$
	\equiv X^r + (a-X^r) \pmod{(\Psi_m(X),\Psi_m(a-X^r))}\equiv a \pmod{(\Psi_m(X),\Psi_m(a-X^r))}. 
	$$
	As the ideal $\bigl(\Psi_m(X),\Psi_m(a-X^r)\bigr)$ is contained in $\bigl(\Theta_{m,r,N;a}(X), N\bigr)$, we get that
	$$
	[X+(a-X^r)^{r'}-a]^N \equiv [pF(X)]^{(\frac{N}{p})} \pmod{(\Theta_{m,r,N;a}(X), N)} . 
	$$
	Division Algorithm similarly gives that there exists a unique $d$-tuple 
	$$(c_0,\ldots,c_{d-1})\in\{0,p,2p,\ldots,N-p\}^{d}$$ 
	such that
	$$
	 [pF(X)]^{(\frac{N}{p})} \equiv \overset{d-1}{\underset{i=0}{\sum}}c_{i}X^i \pmod{(\Theta_{m,r,N;a}(X), N)}.
	$$
	As $[X+(a-X^r)^{r'}-a]^N \not \equiv 0\pmod{(\Theta_{m,r,N;a}(X), N)}$, there exists $i\in \llbracket0, d\rrbracket$ such that $b_i\neq 0$. This and the congruence $b_i\equiv c_i \pmod{N}$ imply that $p\mid\gcd(b_i, N)\neq N$. So part (2) holds.
\end{proof}

\begin{Rem}\normalfont\label{R4}
	For a fixed integer $m$, as the prime divisors of $N$ are not directly observable, the residue $r$ modulo $m$ of any such prime divisor remains inaccessible. 
	Because of this, in many applications we perform computations for each integer $r$ such that $r\in\llbracket2,m-1\rrbracket$ and $\gcd(m,r)=1$ (e.g., see proof of Theorem \ref{TH3} in which $l$ is used and Algorithm 2 of Section \ref{S8}).
\end{Rem}

\begin{Le}\label{L10} 
	Let $(N,a)\in\mathbb T\times \mathbb{Z}$. Suppose there exist distinct prime divisors $p$ and $q$ of $N$ and $m\in\mathbb{N}\setminus\{1,2\}$ such that $N$ has property $(\natural)$ with respect to $(p,q,m; a)$. 
	Let $(r,r')\in\llbracket1,m-1\rrbracket^2$ be such that $m\mid rr'-1$. Suppose that there exists a monic polynomial $\Theta_{m,N,r;a}(X)\in \mathbb{Z}[X]$ such that 
	$$\bigl(\Psi_m(X),\Psi_m(a-X^r), N\bigr)=\bigl(\Theta_{m,N,r;a}(X), N)\bigr).$$
Then $[X+(a-X^r)^{r'}-a]^N \not \equiv 0\pmod{(\Theta_{m,N,r;a}(X), N)}$. 
\end{Le}
\begin{proof} 
	This follows from Propositions \ref{P7} and \ref{P4} applied to the square free $pq\in\mathbb T$.\end{proof}

	
	


\begin{theorem}\label{TH3}
Let $(N, m, a)\in \mathbb{T} \times \mathbb{N}\times \mathbb{Z}$ be such that $N$ has property $(\natural_{m;a})$. Then we can find a proper divisor of $N$ at the cost of 
$$
O\bigl( \phi(m)[m^3 + m^2(\log N)^2] M(\lfloor \log_2 N\rfloor +1) \bigr).
$$
\end{theorem}
\begin{proof} By replacing $a$ with $b\in\llbracket0,N-1\rrbracket$ such that $a\equiv b\pmod{N}$, for computations it is convenient to assume that $a\in\llbracket0,N-1\rrbracket$. Suppose that $N$ has property $(\natural)$ with respect to $(p,q,m;a)$ for a pair $(p,q)$ of distinct prime divisors of $N$. Let $r\in \llbracket1, m-1\rrbracket$ be the residue of $p$ modulo $m$. 

We first assume that there exists a monic polynomial $\Theta_{m,N,r;a}(X)\in \mathbb{Z}[X]$ such that we have an identity $\bigl(\Theta_{m,N,r;a}(X),N\bigr)=\bigl(\Psi_m(X),\Psi_m(a-X^r),N\bigr)$. Then
	by Lemma \ref{L10} we have that
	$$
	[X+(a-X^r)^{r'}-a]^N\not\equiv 0 \pmod{(\Theta_{m,N,r;a}(X), N)},
	$$
	where $(r, r')\in \mathbb{N}^2$, $r\equiv p \pmod{m}$ and $r r' \equiv 1 \pmod{m}$. If $r=1$ then by taking $r'=1$ we get that $0\not\equiv 0 \pmod{(\Theta_{m,N,r;a}(X), N)}$, a contradiction; so $r\ge 2$. A similar argument gives that $t_{m,p;a}=\deg(\Theta_{m,N,r;a})>0$.
	
We can find a proper divisor of $N$ by Lemma \ref{L9}(2) and we compute the cost below. 
	
	As $\bigl(\Theta_{m,N,r;a}(X),N\bigr)=\bigl(\Psi_m(X),\Psi_m(a-X^r),N\bigr)$, we have $t_{m,p;a}\le m$. 

Let $l\in \llbracket1, m-1\rrbracket$ with $\gcd(l, m)=1$. If there exists no monic polynomial $\Theta_{m,N,l;a}(X)$ in $\mathbb{Z}[X]$ such that $\bigl(\Theta_{m,N,l;a}(X),N\bigr)=\bigl(\Psi_m(X),\Psi_m(a-X^l),N\bigr)$, 
 then by Proposition \ref{P8}(2) applied to $(c,d)=(m, m)$ and 
 $\bigl(g_{-1}(X), g_0(X)\bigr)=\bigl(\Psi_m(a-X^l), \Psi_m(X)\bigr)$,
there exists $\iota\in \llbracket1, m\rrbracket$ such that we can find a proper divisor of $N$ at the cost bounded by
	$$\left(B_{m+\iota-1, m,\iota}-\frac{\iota (\iota-1)}{2}\right)M_{\mod}(N)+B_{m,m,\iota}O(\log_2 N)+O\bigl((2\iota-1) M(\lfloor \log_2 N \rfloor +1) \log N\bigr).$$
We have:
	$$
	B_{m+\iota-1,m,\iota}=m\left(m+\iota-1+\iota-\frac{m(\iota+1)}{2\iota}\right)-\frac{\iota (\iota-1)}{2}
	$$
	and
	$$
	B_{m,m,\iota}=m\left(m+\iota-\frac{m(\iota+1)}{2\iota}\right)-\frac{\iota (\iota-1)}{2}.
	$$	
Viewing these as functions of $\iota$ in the real interval $[1,m]$, the derivatives are easily see to be positive in the mentioned interval and hence we get the following upper bounds
$$B_{m+\iota-1,m,\iota}\le B_{2m-1,m,m}=\frac{5}{2}m^2-2m-\frac{1}{2}<\frac{5}{2}m^2$$
and
$$B_{m,m,\iota}\le B_{m,m,m}=m^2.$$
Thus, if there exists no monic polynomial $\Theta_{m,N,l;a}(X)\in \mathbb{Z}[X]$ as above, then we can find a proper divisor of $N$ at the cost of
	$O\bigl(m^2 M_{\mod}(N)+m^2 \log_2 N + m M(\lfloor \log_2 N \rfloor +1) \log N \bigr)$ and hence, as $M_{\mod}(N)=O\bigl(M(\lfloor \log_2 N\rfloor +1)\bigr)$ by \cite{GG}, Corollary 9.9, at the cost of
	\begin{align}\label{bound1}
	O\bigl( (m^2+m\log N)M(\lfloor \log_2 N\rfloor +1)\bigr) . 	
	\end{align}
	
	Assume now that there exists a monic polynomial $\Theta_{m,N,l;a}(X)\in \mathbb{Z}[X]$ such that we have $\bigl(\Theta_{m,N,l;a}(X),N\bigr)=\bigl(\Psi_m(X),\Psi_m(a-X^l),N\bigr)$. 
	Then by Proposition \ref{P8}(1) applied to the quadruple $(c,d,g_{-1},g_0)$ being $\bigl(m, m,\Psi_m(a-X^l),\Psi_m(X)\bigr)$, there exists $\iota\in \llbracket1, m\rrbracket$ such that we can find $\Theta_{m,N,l;a}(X)$ at the cost bounded by
	$$\left(B_{m+\iota-1,m,\iota}-\frac{\iota (\iota-1)}{2}\right)M_{\mod}(N)+B_{m,m,\iota}O(\log_2 N)+O\bigl(2(\iota-1) M(\lfloor \log_2 N \rfloor +1) \log N\bigr).$$
	This expression differs from the previous one only by a replacement of $2\iota-1$ by the smaller $2(\iota-1)$. Thus, we can find $\Theta_{m,N,l;a}(X)$ at the cost of
	$O\bigl( (m^2+m\log N)M(\lfloor \log_2 N\rfloor +1)\bigr)$. 
	
	From Lemma \ref{L4}(2) applied to the reduction modulo the largest square free natural number that divides $N$, we get that $\deg\bigl(\Theta_{m,N,l;a}(X)\bigr)=t_{m,p;a}>0$. 

 Let $l'\in \llbracket2, m-1\rrbracket$ be such that $ll'\equiv 1 \pmod{m}$. 
 We know that the cost of computing $l'$ is $O(M(\lfloor \log_2 m \rfloor +1)\log m )$.
 By Lemma \ref{L9}(1) applied to $f(X)=\Theta_{m,N,l;a}(X)$ and $d=t_{m,p;a}$, there exists a unique 
 $t_{m,p;1}$-tuple $(b_{0,l},\ldots,b_{t_{m,p;1}-1,l})\in \llbracket0, N-1\rrbracket^{t_{m,p;a}}$ such that 		
	$$
	[X+(a-X^l)^{l'}-a]^N \equiv \overset{t_{m,p;a}-1}{\underset{i=0}{\sum}}b_{i,l}X^i \pmod{(\Theta_{m,N,l;a}, N)}.
	$$
 We show that one can find the tuple $(b_{0,l},\ldots,b_{t_{m,p;a}-1,l})$ at the cost of 
 $$O\bigl(m^2(\log N)^2 M(\lfloor \log_2 N\rfloor +1)\bigr).$$ 
 Let $\theta_{m,N,l;a}(X)\in (\mathbb{Z}/N\mathbb{Z})[X]$ be $\Theta_{m,N,l;a}(X)$ modulo $N$.
 As $\theta_{m, N, l; a}(X)$ is a monic polynomial of degree $t_{m,p;a}$, there exists a unique polynomial $g_l(X)\in (\mathbb{Z}/N\mathbb{Z})[X]$ such that $\deg(g_l)< t_{m,p;a}$
 and $X+(a-X^l)^{l'}-a \equiv g_l(X) \pmod{\theta_{m,N,l;a}(X)}$. Here $X+(a-X^l)^{l'}-a$ is viewed as a polynomial in $(\mathbb{Z}/N\mathbb{Z})[X]$.
 By Lemma \ref{L7} applied to the pairs $\bigl(f(X), g(X)\bigr)=\bigl(X+(a-X^l)^{l'}-a, \theta_{m,N,l;a}(X)\bigr)$ and $(c, d)=(ll', t_{m,p;a})$, we can find $g_l(X)$ at the cost of
 \begin{align}\label{bound2}
 t_{m,p;a}(ll'+1-t_{m,p;a})\left[M_{\mod}(N)+O(\log_2 N)\right]=O\bigl(m^3 M(\lfloor \log_2 N\rfloor +1)\bigr).
 \end{align}
 
We have that $\sum_{i=0}^{t_{m,p;a}-1} b_{i,l}X^i -g_l(X)^N \in \bigl(\theta_{m,N,l;a}(X)\bigr)$. Here $\sum_{i=0}^{t_{m,p;a}-1} b_{i,l}X^i$ is viewed as a polynomial
 in $(\mathbb{Z}/N\mathbb{Z})[X]$. Thus, we can find $(b_{0,l},\ldots,b_{t_{m,p;a}-1,l})$
 using the coefficients of the residue of $g_l(X)^N$ in $(\mathbb{Z}/N\mathbb{Z})[X]/\bigl(\theta_{m,N,l;a}(X)\bigl)$. 
 As the cost of one multiplication of two polynomials in $(\mathbb{Z}/N\mathbb{Z})[X]$ with degree
 less than $t_{m,p;a}$ is $O\bigl( t_{m,p;a}^2 M_{\mod}(N)\bigr)$, by \cite{GG}, Corollary 9.7, the cost of one multiplication of residues in
 $(\mathbb{Z}/N\mathbb{Z})[X]/\bigl(\theta_{m,N,l;a}(X)\bigr)$ is $O\bigl( t_{m,p;a}^2 M_{\mod}(N)\bigr)$. Using the fast powering method, the computation of $g_l(X)^N$ requires $O\bigl((\log N)^2\bigr)$ multiplications. Thus, we can find
 $(b_{0,l},\ldots,b_{t_{m,p;a}-1,l})$ at the cost of $O\bigl( t_{m,p;a}^2 M_{\mod}(N)(\log N)^2\bigr)$. As $t_{m,p;a}<m$, we can find
 $(b_{0,l},\ldots,b_{t_{m,p;a}-1,l})$ at the cost of
 \begin{align}\label{bound3}
 	O\bigl( m^2(\log N)^2 M(\lfloor \log_2 N\rfloor +1)\bigr) . 
 \end{align}
 
 The next step is to compute $\gcd(b_{i,l}, N)$ for each $i\in \llbracket0, t_{m,p;a}-1\rrbracket$. This can be done at the cost of
 $O\bigl(M(\lfloor \log_2 N \rfloor +1) \log N\bigr)$ for each $i\in \llbracket0, t_{m,p;a}-1\rrbracket$. 
 Thus, for a fixed $l$, using Display (\ref{bound1}), Equation (\ref{bound2}), and Display (\ref{bound3}), 
 we get that all integers $\gcd(b_{i,l}, N)$, $i\in \llbracket0, t_{m,p;a}-1\rrbracket$, can be computed at the cost of
 $$
 O\bigl([m^2+m\log N +m^3 + m^2(\log N)^2] M(\lfloor \log_2 N\rfloor +1)\bigr)= O\bigl([m^3 + m^2(\log N)^2] M(\lfloor \log_2 N\rfloor +1)\bigr) .
 $$	
 For a given $l$ it is possible that $\gcd(b_{i,l}, N)=1$ for all $i\in \llbracket0, t_{m,p;a}-1\rrbracket$. But as $N$ has property $(\natural)$
 with respect to $(p, q, m; a)$, by Lemma \ref{L9}(2) and Lemma \ref{L10} we have that
 the above procedure gives a proper divisor of $N$ for $l=r\equiv p \pmod{m}$. 
 If we perform the above procedure for all $l\in\llbracket2,m-1\rrbracket$ with $\gcd(l, m)=1$, 
 then the total cost is bounded by
	$$
	O\bigl( \phi(m)[m^3 + m^2(\log N)^2] M(\lfloor \log_2 N\rfloor +1) \bigr),
	$$
	and the theorem is proved. 
\end{proof}

\section{On property $(\natural)$}\label{S7}

Let $(m, a)\in \mathbb{N}\times\mathbb Z$. We define $\tilde\Psi_{m;a}(X)\in\mathbb Z[X]$ by 
 $$
 \tilde\Psi_{m;a}(X) := \begin{cases} X^{m}-1  \quad\quad\quad\quad\quad\quad\quad\quad\quad\quad\, {\rm if} \quad 2\nmid m, \quad a=0\\
 1  \quad\quad\quad\quad\quad\quad\quad\quad\quad\quad\quad\quad\quad {\rm if} \quad 2\mid m, \quad a=0\\
 X^m -1  \quad\quad\quad\quad\quad\quad\quad\quad\quad\quad\, {\rm if} \quad 6\nmid m, \quad a=1\\ 
 1+X^6+X^{12}+\cdots + X^{m-6} \quad\, {\rm if} \quad 6\mid m, \quad a=1\\
 X^m -1  \quad\quad\quad\quad\quad\quad\quad\quad\quad\quad\, {\rm if} \quad 3\nmid m, \quad a=-1\\
 1+X^3+X^{6}+\cdots + X^{m-3} \quad\;\; {\rm if} \quad 3\mid m, \quad a=-1\\
 1+X+X^2+\cdots +X^{m-1}  \quad\quad {\rm if} \quad a=2\\
 X^{m}-1  \quad\quad\quad\quad\quad\quad\quad\quad\quad\quad\, {\rm if} \quad 2\nmid m, \quad a=-2\\
 -1+X-X^2+\cdots +X^{m-1}  \quad\; {\rm if} \quad 2\mid m, \quad a=-2\\
 X^{m}-1  \quad\quad\quad\quad\quad\quad\quad\quad\quad\quad\, {\rm if} \quad |a|>2 . \end{cases}
 $$
 
\begin{Le} \label{L11}
 	Let $(m, a)\in \mathbb N \times \mathbb{Z}$. Then the following properties hold.
 	
 	\medskip
 	{\bf (1)} 
 	We have $\tilde{\Psi}_{m;a}(X)=\Psi_m(X)$ iff $\Delta_{m; a}\neq 0$.
 	
 	\smallskip
 	
 	{\bf (2)} 
 	We have $\Psi_m(X)=(X-1)\tilde{\Psi}_{m;2}(X)$.
 	
 	\smallskip
 	
 	{\bf (3)} 
 	If $2\mid m$, then $\Psi_m(X)=(X+1)\tilde{\Psi}_{m;-2}(X)$.		
 	
 	\smallskip
 	
 	{\bf (4)} 
 	If $6\mid m$, then $\Psi_m(X)=(X^6-1)\tilde{\Psi}_{m;1}(X)$.
 	
 	\smallskip
 	
 	{\bf (5)} 
 	If $3\mid m$, then $\Psi_m(X)=(X^3-1)\tilde{\Psi}_{m;-1}(X)$.
 \end{Le}
 \begin{proof} From the very definition of $\tilde{\Psi}_{m;a}(X)$ we get that parts (2) to (5) hold and that we have $\tilde{\Psi}_{m;a}(X)=x^m-1=\Psi_m(X)$ iff
 	$$
 	(m, a)\notin \bigl(6\mathbb N \times \{1\}\bigr)\cup \bigl(3\mathbb N \times \{-1\}\bigr)\cup \bigl(\mathbb N \times \{2\}\bigr)\cup \bigl(2\mathbb N \times \{-2\}\bigr)\cup
 	\bigl(2\mathbb N \times \{0\}\bigr)
 	$$
and thus iff $\Delta_{m; a}\neq 0$ by Theorem \ref{TH1}(3). So part (1) also holds.\end{proof}

\begin{Le}\label{L12}
For $(m, a)\in \mathbb{N}\times \mathbb{Z}$ let $\tilde\Delta_{m;a}:=|\Res\bigl(\tilde{\Psi}_{m;a}(X),\Psi_m(a-X)\bigr)|$. Then the following properties hold.

\medskip
{\bf (1)} We have inequalities $0<\tilde\Delta_{m;a}<(|a|+2)^{m^2}$.

\smallskip
{\bf (2)} If $m\ge 2$, then the number of prime divisors of $\tilde\Delta_{m;a}$ is bounded by $m^2\log_3(|a|+2)$.\footnote{The asymptotic bound of the number of primes dividing a fixed $N\in\llbracket1,(|a|+2)^{m^2}-1\rrbracket$ is $\frac{m^2\log(|a|+2)}{\log\bigl(m^2\log(|a|+2)\bigr)}$ (see \cite{HW}, Section 22.10, Page 471), for large $m$.}
\end{Le}
\begin{proof} For each pair $(m, a)\in (\mathbb{N}\times \mathbb{Z})\setminus (2\mathbb N\times\{0\})$ we define a subset of $\llbracket1, m\rrbracket$ as follows:
$$
\mathbb J_{m;a} := \begin{cases} \llbracket1, m\rrbracket \quad\quad\quad\quad\;\;\; {\rm if} \quad 2\nmid m, \: a\in\{0,-2\} \quad {\rm or }\quad 6\nmid m, \: a=1 \quad {\rm or }\quad 3\nmid m, \: a=-1\\
	 \llbracket1, m\rrbracket\setminus (\frac{m}{6}\mathbb{N}) \quad\;\, {\rm if} \quad 6\mid m, \quad a=1\\
	 \llbracket1, m\rrbracket\setminus (\frac{m}{3}\mathbb{N}) \quad\;\, {\rm if} \quad 3\mid m, \quad a=-1\\
	 \llbracket1, m-1\rrbracket \quad\quad\quad\, {\rm if} \quad a=2\\
	 \llbracket1, m\rrbracket\setminus\{\frac{m}{2}\} \quad\quad {\rm if} \quad 2\mid m, \quad a=-2\\
	 \llbracket1, m\rrbracket, \quad\quad\quad\quad\;\, {\rm if} \quad |a|>2. \end{cases}
$$

For $m\in 2\mathbb N$, let $\tilde\Delta_{m;0}:=1$. For $(m, a)\in\mathbb{N}\times \mathbb{Z}$ we have:
$$
\tilde\Delta_{m;a} = |\underset{i\in \mathbb J_{m;a}}{\prod} \bigl((a-\zeta_m^i)^m-1\bigr)|\leq \underset{i\in \mathbb J_{m;a}}{\prod} |((a-\zeta_m^i)^m-1)|\leq \bigl((|a|+1)^m+1\bigr)^{|\mathbb J_{m;a}|}\leq (|a|+2)^{m^2}.$$
Thus, for part (1) it suffices to show that $\tilde\Delta_{m;a}\neq 0$. 

If $m$ is odd, then Lemma \ref{L2.5}(11) implies that $\tilde\Delta_{m;0}\neq 0$.

If $a=1$, then, as $X^2-X+1 \mid X^6-1$, Lemmas \ref{L2.5}(7) and \ref{L11}(4) imply that $\tilde\Delta_{m;1}\neq 0$.

If $a=-1$, then, as $X^2+X+1 \mid X^3-1$, Lemmas \ref{L2.5}(8) and \ref{L11}(5) imply that $\tilde\Delta_{m;-1}\neq 0$.

If $a=2$, then Lemmas \ref{L2.5}(9) and \ref{L11}(2) imply that $\tilde\Delta_{m; 2}\neq 0$.

If $a=-2$, then Lemmas \ref{L2.5}(10) and \ref{L11}(3) imply that $\tilde\Delta_{m; -2}\neq 0$. So part (1) holds.

For part (2), we note that $m^2\log_3(|a|+2)\ge m^2\log_3(2)\ge 4\log_3(2)>2$. For each $l\in\mathbb N\setminus\{1,2\}$, the product of the smallest $l$ primes is greater than $3^l$. Thus, if $\tilde\Delta_{m;a}$ has at least $m^2\log_3(|a|+2)$ distinct prime divisors, then $\tilde\Delta_{m;a}>3^{m^2\log_3(|a|+2)}=(|a|+2)^{m^2}$, which contradicts part (1). So part (2) holds.
\end{proof}

\begin{Le}\label{L13} Let $m\in \mathbb N$ and let $q$ be a prime number. Suppose there exists an element $\alpha\in\overline{\mathbb F_q}$ such that $\alpha^m=1$.
	Then the following properties hold.
	
	\medskip 
	
	{\bf (1)} 
	If $(2-\alpha)^m=1$ and $\alpha \neq 1$, then $q \mid \tilde{\Delta}_{m;2}$.
	
	\smallskip
	
	{\bf (2)} 
	If $2\mid m$, $(-2-\alpha)^m=1$ and $\alpha \neq -1$, then $q \mid \tilde{\Delta}_{m;-2}$.
	
	\smallskip
	
	{\bf (3)} 
	If $6\mid m$, $(1-\alpha)^m=1$ and $\alpha$ is not a $6$-th root of unity, then $q \mid \tilde{\Delta}_{m;1}$.
	
	\smallskip
	
	{\bf (4)} 
 If $3\mid m$, $(-1-\alpha)^m=1$ and $\alpha$ is not a $3$-rd root of unity, then $q \mid \tilde{\Delta}_{m;-1}$.
\end{Le}
\begin{proof} For part (1) we use Lemma \ref{L11}(2) which states that $\Psi_m(X)=(X-1)\tilde{\Psi}_{m;2}(X)$. As $\alpha \neq 1$ and $\Psi_m(\alpha)=0$, we must have that $\tilde{\Psi}_{m;2}(\alpha)=0$.
	It follows that $\alpha$ is a common root of $\tilde{\Psi}_{m;2}(X)$ and $\Psi_m(2-X)$, hence $q\mid\tilde\Delta_{m;2}$. So part (1) holds. 
	
	Parts (2) to (4) are proved similarly using Lemma \ref{L11}(3) to (5) (respectively).
	\end{proof}

\begin{Pro}\label{P10}
	For $(m,a)\in \mathbb{N}\times\mathbb Z$ we define a subset of prime numbers as follows: 
	$$
	P_{m;a}:=\{p \: | \: p\: \text{is a prime and} \; p\mid\tilde\Delta_{m;a}\} .
	$$
	Let $N\in \mathbb{N}$ and let $p$ and $q$ be two distinct primes dividing $N$. Then $N$ does not satisfy the property $(\natural)$ with respect to $(p,q,m;a)$ provided one of the following conditions holds. 
	
	\medskip
	{\bf (1)} 
	We have $|a|\ge 3$, $\Delta_{m;a}\neq 0$, and $q\notin P_{m;a}$. 
	
	\smallskip
	
	{\bf (2)} 
	We have $p>2$, $2\mid m$, and $a=0$. 
	
	\smallskip
	
	{\bf (3)}
	We have $q\notin P_{m;2}$ and $a=2$.
	
	\smallskip 
	
	{\bf (4)}
	We have $m\in 2\mathbb N$, $q\notin P_{m;-2}$, and $a=-2$.
	
	\smallskip 
	
	{\bf (5)} 
	We have $6 \mid m$, $q\notin P_{m;1} \cup \{7\}$, and $a=1$. 
  
\smallskip 
{\bf (6)} We have $3 \mid m$, $q\notin P_{m;-1}$, and $a=-1$.\end{Pro}

\begin{proof} We can assume that $\gcd(p,m)=1$. Let $r\in \llbracket1, m-1\rrbracket$ be such that $m\mid p-r$. Let $r'\in \mathbb{N}$ be such that $m\mid rr'-1$. It suffices to show that the assumption that there exists a pair $(X_1, Y_1)$ of $m$-th roots of unity in $\overline{\mathbb F_q}$ such that $X_1^r+Y_1^r=a$ and $X_1+Y_1\neq a$ leads to a contradiction, with $a$ in parts (2), (3), (4), (5), and (6) being $0$, $2$, $-2$, $1$, and $-1$ (respectively). 

For part (1), as $X_1^r$ and $a-X_1^r$ are $m$-th roots of unity, it follows that $q\mid\Delta_{m;a}$. Thus $q \mid |\Delta_{m;a}|=|\tilde\Delta_{m;a}|$, a contradiction. So part (1) holds.
	
For part (2), as $r$ is odd and $m$ is even, $r'$ must be odd. Then, as $X_1^r=-Y_1^r$, we have $X_1^{rr'}=-Y_1^{rr'}$. So $X_1=-Y_1$, a contradiction.
	
	For part (3), as $X_1^r$ and $2-X_1^r$ are $m$-th roots of unity and $q\nmid\tilde\Delta_{m;2}$, we have $X_1^{r}=1$ by Lemma \ref{L13}. By symmetry, we also have $Y_1^r=1$. As $\gcd(r, m)=1$, we derive that $X_1=Y_1=1$, a contradiction to $X_1+Y_1\neq 2$.
	
	For part (4), as $X_1^r$ and $-2-X_1^r$ are $m$-th roots of unity and $q\nmid\tilde\Delta_{m;-2}$, we have $X_1^{r}=-1$ by Lemma \ref{L13}. By symmetry, we also have $Y_1^r=-1$. As $r$ is odd and $m$ is even, $r'$ must be odd. Then $X_1^{rr'}=-1$ and $Y_1^{rr'}=-1$. So $X_1=Y_1=-1$, a contradiction 
	to $X_1+Y_1\neq -2$.
	
 For part (5), as $X_1^r$ and $1-X_1^r$ are $m$-th roots of unity and $q\nmid\tilde\Delta_{m;1}$, we have $X_1^{6r}=1$ by Lemma \ref{L13}. As $X_1^{6r}=1=X_1^{6m}$ and $\gcd (r, m)=1$, we get that $X_1^6=1$. By symmetry, we also have $Y_1^6=1$. If $r\equiv 1 \pmod{6}$, then from $X_1^r+Y_1^r =1$ together with $X_1^6=Y_1^6=1$,
 we get that $X_1+Y_1 =1$, a contradiction. 
 
 Assume now that $r \not\equiv 1 \pmod{6}$. As $6 \mid m$ and $\gcd (r, m)=1$, we get that $r\equiv 5 \pmod{6}$. Thus $X_1^5+Y_1^5=1$, hence $X_1^{-1}+Y_1^{-1}=1$. Multiplying this by $X_1Y_1$, we get $X_1+Y_1=X_1Y_1$, in particular $X_1Y_1\neq 1$.
 If $X_1=1$ or $Y_1=1$, then $X_1+Y_1=X_1Y_1$ implies that $1=0$, a contradiction. If $X_1=-1$, then $Y_1^{-1}=1-X_1^{-1}=1+1=2$. 
 Then we get that $2^6=1$ in $\mathbb{F}_q$, hence $q\mid 63$. As $q\neq 7$, we must have $q=3$. So $Y_1=2^{-1}=2$ and $X_1+Y_1=-1+2=1$, a contradiction. So $X_1\neq -1$.
 By symmetry $Y_1\neq -1$. We also have that $X_1\neq Y_1$, because if $X_1=Y_1$, then the equation $X_1^{-1}+Y_1^{-1}=1$ implies that $X_1=2$. Then $2^6=1$, hence $q=3$,
 and we get $X_1+Y_1=2+2=1$, a contradiction. Thus, $1$, $X_1$, $Y_1$, and $X_1Y_1$ are $4$ distinct $6$-th roots of unity. Therefore, it is
 not possible both $X_1$ and $Y_1$ to be $3$-rd roots of unity. Without loss of generality, we can assume that $X_1^3\neq 1$. As $X_1\neq -1$, $X_1$ must be a primitive $6$-th
 root of unity. Then $X_1^2-X_1+1=0$, hence $X_1^3=-1$. So $X_1^2 =-X_1^{-1}$. Thus we have $X_1=X_1^2+1=1-X_1^{-1}=Y_1^{-1}$. We compute $1= X_1^{-1}+Y_1^{-1}=Y_1+X_1$, a contradiction. So part (5) holds.
	
	For part (6), as $X_1^r$ and $-1-X_1^r$ are $m$-th roots of unity and $q\nmid\tilde\Delta_{m;-1}$, we have $X_1^{3r}=1$ by Lemma \ref{L13}. As $X_1^{3r}=1=X_1^{3m}$ and $\gcd (r, m)=1$, we get that $X_1^3=1$. By symmetry, we also have $Y_1^3=1$. Then $X_1^r+Y_1^r=-1$, $X_1+Y_1\neq -1$, and $\gcd(3, r)=1$ imply that $X_1^2+Y_1^2=-1$. 
	Suppose that $X_1=1$. Then $Y_1^2=-1-X_1^2=-2$, so in $\overline{\mathbb{F}_q}$ we have that $1=Y_1^6=-2^3=-8$. This implies that $q=3$. 
	As $Y_1^3=1$ and $3$ does not divide the order of the multiplicative group of $\mathbb{F}_3(Y_1)$, we get that $Y_1=1$. Then $X_1+Y_1=1+1=-1$, a contradiction.
	Thus $X_1\neq 1$. By symmetry, we also have that $Y_1\neq 1$. In this case both $X_1$ and $Y_1$ are primitive $3$-rd roots of unity. So $X_1$ and $Y_1$ are solutions
	to the equation $X^2+X+1=0$. If $X_1\neq Y_1$, by the Vieta's formula $X_1+Y_1=-1$, a contradiction. If $X_1=Y_1$, then we have that $-1=X_1^2+Y_1^2=2X_1^2$. 
	Multiplying $X_1^2+X_1+1=0$ by $2$ and substituting $-1$ for $2X_1^2$, we get that $2X_1 =-1$. By cubing this equation we get that in $\mathbb{F}_q$ we have $8=-1$,
	so $q=3$. Then $2X_1=-1$ implies that $X_1=1$, a contradiction. So part (6) holds.\end{proof}

\begin{Rem}\normalfont\label{R5}
	We have $2^5+2^5=64\equiv 1 \pmod{7}$ and $2+2 \not\equiv 1 \pmod{7}$. Thus, if $p$ is a prime such that $p\equiv 5 \pmod 6$ and $N\in 7p\mathbb N$, then $N$ has property $(\natural)$ with respect to $(p, 7, 6;1)$. 
\end{Rem}

\begin{Le}\label{L14}
	Let $(N,m,a)\in \mathbb{T}\times\mathbb{N}\times \mathbb{Z}$ and $\Psi^0_{m; a}(X) :=\Psi_m(X)/\Upsilon_{m; a}(X)$ with $\Upsilon_{m; a}(X)$ as in Lemma \ref{L2.5}. Then the following properties hold for 
	$$\Delta^0_{m;a} :=|\Res\bigl(\Psi^0_{m;a}(X),\Psi^0_{m;a}(a-X)\bigr)|.$$ 
	
	{\bf (1)} 
	We have $0< \Delta^0_{m;a}<(a+2)^{m^2}$.
	
	\smallskip
	{\bf (2)} 
	If $m\ge 2$, then the number of prime divisors of $\Delta^0_{m;a}$ is bounded by $m^2\log_3(|a|+2)$.
	
	\smallskip
	{\bf (3)} 
	If $N$ has property $(\sharp_{m;a})$, then $\gcd (N, \Delta^0_{m;a})>1$. 
\end{Le}

\begin{proof} For part (1), as $\gcd\bigl(\Psi_m(X), \Psi_m(a-X)\bigr)=\Upsilon_{m;a}(X)$, $\Psi^0_{m;a}(X)$ and $\Psi^0_{m;a}(a-X)$ have no common roots and $\Upsilon_{m;a}(a-X)=(-1)^{\deg(\Upsilon_{m;a})}\Upsilon_{m;a}(X)$. Thus $0< \Delta^0_{m;a}$. 
	
	We have:
	$$
	\Res\bigl(\Psi^0_{m;a}(X), \Psi_m(a-X)\bigr )=\Res\bigl(\Psi^0_{m;a}(X),\Psi^0_{m;a}(a-X)\bigr)\Res\bigl(\Psi^0_{m;a}(X),(-1)^{\deg(\Upsilon_{m;a})}\Upsilon_{m;a}(X)\bigr).
	$$
	As $\Psi_m(X)=\Psi^0_{m;a}(X)\Upsilon_{m;a}(X)$ is a separable polynomial, $\Psi^0_{m;a}(X)$ and $\Upsilon_{m;a}(X)$ have no common roots in $\mathbb C$, so
	$\Res\bigl(\Psi^0_{m;a}(X),\Upsilon_{m;a}(X)\bigr)\in \mathbb{Z}\setminus\{0\}$. This implies that
	$$
	|\Res\bigl(\Psi^0_{m;a}(X),\Psi^0_{m;a}(a-X)\bigr)| \leq |\Res\bigl(\Psi^0_{m;a}(X), \Psi_m(a-X)\bigr )| .
	$$
	From this and the following estimate 
	$$
	|\Res\bigl(\Psi^0_{m;a}(X), \Psi_m(a-X)\bigr )| = 
	|\underset{k\in\llbracket1,m\rrbracket, \: \Upsilon_{m;a}(\zeta_m^k)\neq 0 }{\prod} \bigl((a-\zeta_m^k)^m-1\bigr)|\leq \underset{k\in\llbracket1,m\rrbracket}{\prod} ((|a+1)^m+1)|\leq (|a|+2)^{m^2}
	$$
we get that part (1) holds.
	
	Part (2) follows from part (1) by the same argument as in the proof of Lemma \ref{L12}(2).
	
	For part (3), it suffices to show that the assumption that $\gcd (N, \Delta^0_{m;a})=1$ leads to a contradiction. As $\Upsilon_{m;a}(a-X)=(-1)^{\deg(\Upsilon_{m;a})}\Upsilon_{m;a}(X)$, we get
	$$
	\bigl( \Psi_{m;a}(X), \Psi_{m;a}(a-X), N\bigr)= \Bigl(\Upsilon_{m;a}(X) \bigl( \Psi^0_{m;a}(X), \Psi^0_{m;a}(a-X)\bigr), N\Bigr).
	$$
	As $\Delta^0_{m;a}(X)\in \bigl( \Psi^0_{m;a}(X), \Psi^0_{m;a}(a-X)\bigr)$ and $\gcd(\Delta^0_{m;a}(X), N)=1$, we get that
	$$ \Bigl(\Upsilon_{m;a}(X) \bigl( \Psi^0_{m;a}(X), \Psi^0_{m;a}(a-X)\bigr), N\Bigr)= \bigl(\Upsilon_{m;a}(X), N\bigr),$$
a contradiction to $N$ has property $(\sharp_{m;a})$ as $\Upsilon_{m;a}(X)$ is monic.
	So part (3) holds. 
	\end{proof}

\begin{Co}\label{C5}
	Let $(c_1, c_2)\in (0, \infty)^2$. For an integer $T\ge 4$ let $\mathcal S_T$ be the non-empty set of square free semiprimes whose prime divisors are less than $T$ and we consider its subsets
 $$\mathcal{S}^1_{T, c_1, c_2}:=\{N\in \mathcal S_T|\exists (m,a)\in\mathbb N\times\mathbb Z, N\; \textup{has property}\; (\sharp_{m;a}), m<T^{c_1}, |a|<T^{c_2}\},$$ 
 $$\mathcal{S}^2_{T, c_1, c_2}:=\{N\in \mathcal S_T|\exists (m,a)\in\mathbb N\times\mathbb Z, N\; \textup{has property}\; (\natural_{m;a}), m<T^{c_1}, |a|<T^{c_2}\}.$$ 
 Then the following properties hold.
	
	\medskip
	 {\bf (1)} 
	 If $3c_1+c_2<1$, then both limits $\lim_{T\to\infty} \frac{|\mathcal S^1_{T, c_1, c_2}|}{|\mathcal S_{T}|}$ and $
	 \lim_{T\to\infty} \frac{|\mathcal S^2_{T, c_1, c_2}|}{|\mathcal S_{T}|}$ exist and are $0$.
	 
	 {\bf (2)} 
	 There exists no constant $C\in (0,\infty)$ such that each sufficiently large square free semiprime $N$ has property $(\sharp_{m;a})$ or $(\natural_{m;a})$ for a pair $(m,a)\in\mathbb N\times\mathbb Z$ with $\max(m,|a|)<(\log N)^C$.\end{Co}

\begin{proof} For large $T$, as the number of primes bounded by $T$ is equal to $\frac{T}{\log T}+ o(\frac{T}{\log T})$ (see \cite{T}, Theorem 3.7),
	we have:
	$$
	|\mathcal S_{T}|=\frac{1}{2}\Bigl( \frac{T}{\log T}+ o\Bigl(\frac{T}{\log T}\Bigr)\Bigr)^2 -\frac{1}{2}\Bigl( \frac{T}{\log T}+ o\Bigl(\frac{T}{\log T}\Bigr)\Bigr)= \frac{T^2}{2 (\log T)^2} + o\Bigl(\frac{T^2}{(\log T)^2}\Bigr) .
	$$
	
Let $\iota\in\{1,2\}$. Let $P^{\iota}_{m;a}$ be $\{p \: | \: p\: \text{is a prime and} \; p\mid\Delta^0_{m;a}\}$ if $\iota=1$ and be the set $P_{m,a}$ of Proposition \ref{P10} if $\iota=2$. The number of primes in the union 
$$\underset{(m, a)\in\llbracket1,\lfloor T^{c_1}\rfloor\rrbracket\times\llbracket -\lfloor T^{c_2}\rfloor,\lfloor T^{c_2}\rfloor\rrbracket}{\cup} P^{\iota}_{m;a},$$ is bounded by the following sum 
$$\underset{|a|\in\llbracket -\lfloor T^{c_2}\rfloor,\lfloor T^{c_2}\rfloor\rrbracket}{\sum}\:\underset{m\in\llbracket1,\lfloor T^{c_1}\rfloor\rrbracket}{\sum}m^2\log_3(|a|+2)$$ 
by Lemma \ref{L14}(2) if $\iota=1$ and by Lemma \ref{L12}(2) if $\iota=2$ and thus by 
	$3T^{3c_1}T^{c_2}\log_3(T^{c_2}+2)$; here a plus $1$ is added to the factor $3$ for each sub-sum with $a<0$ or $a=0$ or $a>0$.  So, for sufficiently large $T$, by Lemma \ref{L12}(3) if $\iota=1$ and by Proposition \ref{P10} if $\iota=2$ we have that $|\mathcal S^{\iota}_{T, c_1, c_2}|$ is strictly less than
\begin{equation*}
	 \Bigl( \frac{T}{\log T}+ o\Bigl(\frac{T}{\log T}\Bigr)\Bigr) 3T^{3c_1+c_2}\log_3(T^{c_2}+2) =
	\frac{3T^{3c_1+c_2+1}\log_3(T^{c_2}+2)}{\log T} + o\bigl(T^{3c_1+c_2+1}\bigr).
\end{equation*}
So part (1) follows from the inequalities
	$$
	0\le \lim_{T\to\infty} \frac{|\mathcal S^{\iota}_{T, c_1, c_2}|}{|\mathcal S_{T}|}\le \lim_{T\to\infty} 6T^{3c_1+c_2-1} \log_3(T^{c_2}+2) \log T = 0. 
	$$

Part (2) follows from part (1).\end{proof}

\begin{Rem}\normalfont\label{R5.5}
Let $\epsilon\in (0,1)$. All odd semiprimes in $\mathcal S_T$ are contained in $\mathcal S^1_{T,\epsilon,1+\epsilon}$ (resp.\ $\mathcal S^2_{T,1,\epsilon}$) by Proposition \ref{P8.1}(1) (resp.\  \ref{P6}(3)). Hence both limits $\lim_{T\to\infty} \frac{|\mathcal S^1_{T,\epsilon,1+\epsilon}|}{|\mathcal S_{T}|}$ and $
\lim_{T\to\infty} \frac{|\mathcal S^2_{T, 1,\epsilon}|}{|\mathcal S_{T}|}$ exist and are $1$. Thus if the pair $(\varepsilon_1,\varepsilon_2)\in (0,\infty)^2$ is such that for every pair $(c_1,c_2)\in (0,\infty)^2$ with $\varepsilon_1 c_1+\varepsilon_2c_2<1$, both limits $\lim_{T\to\infty} \frac{|\mathcal S^1_{T,c_1,c_2}|}{|\mathcal S_{T}|}$ and $
\lim_{T\to\infty} \frac{|\mathcal S^2_{T,c_1,c_2}|}{|\mathcal S_{T}|}$ exist and are $0$, then $\min(\varepsilon_1,\varepsilon_2)\ge 1$.
\end{Rem}	

\begin{Rem}\normalfont\label{R6}
	Let $(n, m, r)\in \mathbb{N}^3$ with $\gcd(r, m)=1$. Then, unlike positive characteristics, in an algebraic closure $\overline{\mathbb{Q}}$ 
	of the field $\mathbb{Q}$ of rational numbers, the equation $\sum_{i=1}^n \al_i X_i^r =0$, with $(\al_1, \ldots, \al_n)\in \mathbb{Q}^n$, has no solution 
	$(A_1, \ldots, A_n)\in \overline{\mathbb{Q}}^n$ such that $A_1^m=\cdots=A_n^m=1$ and $\sum_{i=1}^n \al_i A_i \neq 0$.
	Indeed, $\zeta_m$ and $\zeta_m^r$
	are roots of the irreducible cyclotomic polynomial $\Phi_m(X)$ in $\mathbb{Q}[x]$. Let $\sigma : \mathbb{Q}(\zeta_m) \rightarrow \mathbb{Q}(\zeta_m)$ be the automorphism with $\sigma(\zeta_m)=\zeta_m^r$. If $A_1^m=\cdots=A_n^m=1$, then for each $1\leq i\leq n$ there exists $r_i\in \llbracket0,m-1\rrbracket$
	such that $A_i=\zeta_m^{r_i}$. Thus,
	$$
	\sigma\Bigl(\sum_{i=1}^n \al_i A_i\Bigr)=\sigma\Bigl(\sum_{i=1}^n \al_i \zeta_m^{r_i}\Bigr)=\sum_{i=1}^n \al_i \sigma(\zeta_m)^{r_i}
	=\sum_{i=1}^n \al_i \zeta_m^{rr_i}=\sum_{i=1}^n \al_i A_i^{r}.
	$$
	This implies that $\sum_{i=1}^n \al_i A_i=0$ iff $\sum_{i=1}^n \al_i A_i^r=0$. 
\end{Rem}

\begin{De}\label{D4}
	Let $N\in\mathbb T$. 
	
	\medskip
	{\bf (1)} By the first cyclotomic set of $N$ we mean the set $\mathbb E_1(N)$ formed by all pairs of integers $(m,a)\in\mathbb N\times \mathbb Z$ such that $N$ has property $(\sharp_{m; a})$.
	
	\smallskip
	{\bf (2)} By the second cyclotomic set of $N$ we mean the set $\mathbb E_2(N)$ formed by all pairs of integers $(m,a)\in\mathbb N^2$ such that $N$ has property $(\natural_{m; a})$.
	
	\smallskip
	{\bf (3)} By the cyclotomic complexity of $N$ we mean $$\C_N=\C_N^{\mathbb Q}:=\min\left(\frac{\log(m+|a|)}{\log\log N}\Big|(m,a)\in\mathbb E_1(N)\cup\mathbb E_2(N)\right).$$

	
\end{De}

The cyclotomic complexity of $N\in\mathbb T$ measures the complexity of finding a proper divisor of $N$ greater than $1$.

For $N\in \mathbb{T}$ we check that for each prime divisor $p$ of $N$ we have $(1, p+2)\in\mathbb E_1(N)$.
The Division Algorithm of Proposition \ref{P8} applied modulo $N$ to the pair $\bigl(g_{-1}(X), g_0(X)\bigr)$ equal to $\bigl(X-1+(N),p+1-X+(N)\bigr)$ produces the prime divisor $p$ of $N$. As $(1, p+2)\in \mathbb E_1(N)$, it follows that $\C_N$ exists and we have inequalities $\C_N\le \frac{\log\bigl(\min(p|p\in\Div(N))+3\bigr)}{\log\log N}< \frac{\log(\sqrt{N}+3)}{\log\log N}$. For each $C\in (0, 1)$, there exists $N\in \mathbb T$ such that $\frac{(\log N)^C}{\log\log N}< C_N$ by Corollary \ref{C5}(2). If $N$ is odd, then $\bigl(\min(p|p\in\Div(N))-1,2\bigr)\in \mathbb E_2(N)$ by Corollary \ref{C4.5}(2) and by Corollary \ref{C4.5}(1) and (4) we have
$$\C_N\le \frac{\log\bigl(2\min(p|p\in\Div(N)\bigr)}{\log\log N}\le\frac{\log\bigl(2(\sqrt{N+1}-1)\bigr)}{\log\log N}<\frac{\log(2\sqrt{N})}{\log\log N}=\frac{1}{2}\frac{\log 4N}{\log\log N}.$$

We have $\lim_{T\to\infty} \max(C_N|N\in\mathbb N, N\le T)=\infty$ by Corollary \ref{C5}(2).	

\begin{Rem}\normalfont\label{R7}			
	Let $(m, a)\in\mathbb N\times\mathbb Z$ be such that $N$ has property $(\sharp_{m:a})$ or $(\natural_{m;a})$ and we have $C_N=\frac{\log (m+|a|)}{\log \log N}$. This implies that $m\leq (\log N)^{C_N}$ and $|a|\leq (\log N)^{C_N}$. If $N$ has property $(\sharp_{m;a})$, then following the first part of the proof of Theorem \ref{TH3}, the complexity of finding a proper divisor of $N$ is bounded by
	$O\Big((\log N)^{4C_N}+(\log N)^{3C_N} (\log N) M(\lfloor \log_2 N) \rfloor +1)\Big)$.
	Similarly, if $N$ has property $(\natural_{m;a})$, then by Theorem \ref{TH3} the complexity of finding a proper divisor of $N$ is bounded by
	$O\Big((\log N)^{6C_N}+(\log N)^{5C_N} (\log N)^2 M(\lfloor \log_2 N) \rfloor +1)\Big)$.			
\end{Rem}

\begin{Ex}\normalfont\label{EX5}
Let $p$ and $q$ be two distinct odd primes such that one of them is a Mersenne prime $2^l-1$; so $l\in\mathbb N$ is a prime. The set $O:=\{o_p(2),o_q(2)\}$ has two elements and contains $l$. Let $o:=\min(O)$ and $N:=pq$; we have $o\le l\le \log_2(\frac{N}{3}+1)$. The semiprime $N$ has property $(\sharp_{o,4})$ by Corollary \ref{C4.5}(1) and therefore $C_N\le\frac{\log\bigl( \log_2(\frac{N}{3}+1)+4\bigr)}{\log\log N}$.
\end{Ex}

\begin{Rem}\normalfont\label{R8}
For $l\in\mathbb N$, let $\omega(l)\in\mathbb N\cup\{0\}$ be the number of distinct prime factors of $l$. Let $N\in\mathbb T$ be odd. Let $p$ and $q$ be distinct prime divisors of $N$ with $p<q$. The Generalized Riemann Hypothesis (GRH) implies that 
$$g(q)\ll \omega(q-1)^4\Bigl[\log\bigl(\omega(q-1)\bigr)+1\Bigr]^4\log (q)^2$$ 
(see \cite{Sh}, Theorem 1.3; see also \cite{MTT}, Theorem 2). We have $\omega(l)=O(\frac{\log(l)}{\log (\log l)})$ by \cite{HW}, Section 22.10, Page 471. As there exist pairs $\bigl(o,2g(q)\bigr)\in\mathbb E_1(N)$ by Corollary \ref{C4.5}(3) with $o\mid p-1$, it follows that GRH implies that there exist pairs $(m,a)\in\mathbb E_1(N)$ such that $m\le\sqrt{N}$ and $2\le a\ll O\bigl(\log(N)\bigr)^3$. Based on this and analogs of \cite{Sh}, Theorem 1.3 for primitive roots in finite fields (see \cite{Sh}, Theorems 1.1 and 1.2), one would like to identify quadruples $(C_1,C_2,C_3,C_4)\in (0,\infty)^3\times (0, 1]$ such that $C^K_N\le C_1\frac{(\log C_3N)^{C_4}}{\log\bigl(\log(N)\bigr)^{C_2}}$ for all odd $N\in\mathbb T$, where $K$ is a number field and $C^K_N$ would be defined similarly to $C_N^{\mathbb Q}$ but with the role of $\mathbb Z$ replaced by the ring of integers of $K$ (to be compared with the number field sieve).
\end{Rem}

\section{Algorithms}\label{S8}

First we have the Dichotomic Euclidean Algorithm in one variable modulo $N$ (see Proposition \ref{P8}).

\begin{framed}\label{ALG1}
	\noindent
	\textbf{Algorithm 1 (in 8 Steps).}
	
	\medskip
	
	\noindent
	\textbf{Input:} $N\in\mathbb N\setminus\{1\}$ and pair $\bigl(f(x), g(x)\bigr)\in \mathbb{Z}[x]^2$ with $\deg(g)\leq \deg(f)$ and $g$ monic.
	
	\begin{enumerate}[label=\arabic*.]
		
		\item Replace each coefficient of $f(x)$ and $g(x)$ by its residue modulo $N$.
		
		\item Using division, Divide $f(x)$ by $g(x)$ to obtain
		\[
		f(x) = q(x) g(x) + h(x), \quad \bigl(q(x),h(x)\bigr)\in \mathbb{Z}[x]^2, \quad \deg(h) < \deg(g).
		\]
		
		\item If $h(x) \equiv 0 \pmod{N}$, then output $h(x):=g(x)$ and STOP.
		
		\item If $h(x) \not\equiv 0 \pmod{N}$, then let $l$ be the leading coefficient of $h(x)$. Replace $l$ by its residue modulo $N$.
		
		\item If $1 < \gcd(l, N)$, then output $d:=\gcd(l, N)$ and STOP.
		
		\item If $\gcd (l, N)=1$, then compute $l'\in\llbracket1,N-1\rrbracket$ such that $l l' \equiv 1 \pmod{N}$.

		\item Set $f(x) := g(x)$. 
		
		\item Set $g(x) := l' h(x)$ and go to Step 1. 
		\end{enumerate}
\noindent		
	\textbf{Output:} either a polynomial $h(X)\in \mathbb{Z}[X]$ with $\bigl(f(X), g(X), N\bigr)=\bigl(h(X), N\bigr)$ and coefficients in $\llbracket 0,N-1\rrbracket$ or a $d\in\Div(N)\setminus\{1,N\}$. 	 
	
\end{framed}

The next algorithm searches for pairs $(m,a)\in\mathbb E_1(N)\cup\mathbb E_2(N)$ with $\frac{m+a}{N}$ `small'.

 \begin{framed}\label{ALG3}
	\noindent
	\textbf{Algorithm 2 (in 21 Steps).}
	
	\medskip
	\noindent
	\textbf{Input:} $N\in \mathbb{T}$ (composite integer $N$ with at least $2$ prime divisors).
	
	\begin{enumerate}[label=\arabic*.]
		
		\item Set $m := 3$ and $B := 2\log N$. 	
		
		\item If $m>B$, then STOP with no output. If $m\leq B$, then compute $e_m :=\gcd(m,N)$. If $e_m>1$, then output $d:=e_m$ and STOP.
		If $e_m=1$, then set $r:=2$, $a :=1$.
		
		\item If $r>m-1$, then set $m := m+1$ and go to Step 2. If $\gcd(r, m)=1$, then compute $r'\in\llbracket1,m-1\rrbracket$ such that $r r'=1 \pmod m$, and go to Step 5.
		
		\item If $\gcd(r, m)>1$, then set $r:= r+1$ and go to Step 3. 
		
		\item Set $h(x):= (x+ (a-x^r)^{r'}-a)^N \in \mathbb{Z}[x]$. 
		
		\item Reduce $h(x)$ modulo $x^m-1$ using fast exponentiation (binary exponentiation). 
		
		\item Replace each coefficient of $h(x)$ by its residue modulo $N$.
		
		\item Set $f(x):= (a-x^r)^{m}-1 \in \mathbb{Z}[x]$, $g(x):= x^{m}-1\in \mathbb{Z}[x]$.
		
		\item Replace each coefficient of $f(x)$ and $g(x)$ by its residue modulo $N$. 
		
		\item Divide $f(x)$ by $g(x)$ to obtain
		\[
		f(x) = q(x) g(x) + g_1(x), \quad \bigl(q(x), g_1(x)\bigr)\in \mathbb{Z}[x]^2, \quad \deg(g_1) < \deg(g).
		\]
		
		\item If $g_1(x) \equiv 0 \pmod{N}$, set $\theta(x):= g(x)$ and go to Step 15. If $g_1(x) \not\equiv 0 \pmod{N}$, then replace each coefficient of $g_1(x)$ by its residue modulo $N$.
		
		\item Extract the leading coefficient $l$ of $g_1(x)$. If $1 < \gcd(l, N)$, then output $(d, m, a, 1)$ with $d:=\gcd(l, N)$ and STOP.
		
		\item If $\gcd(l, N)=1$, then set $f(x) := g(x)$ and compute $l'\in\llbracket1,N-1\rrbracket$ such that $N\mid ll'-1$.
		
		\item Set $g(x) := l'g_1(x)$ and go to Step 9.
		
		\item Divide $h(x)$ by $\theta (x)$ to obtain
		\[
		h(x) = \vartheta(x) \theta (x) + \eta (x), \quad \bigl(\vartheta(x), \eta(x)\bigr)\in \mathbb{Z}[x]^2, \quad \deg(\eta) < \deg(\theta).
		\]
		
		\item Replace each coefficient of $\eta(x)$ by its residue modulo $N$.
		
		\item For $j\in\llbracket 0,\deg(\eta)\rrbracket$, extract the coefficient $c_j$ of $x^j$ in $\eta(x)$.
		
		\item For each $j\in\llbracket 0,\deg(\eta)\rrbracket$, set $b_j := \gcd (c_j, N)$.
		
		\item If there exists $j\in\llbracket 0,\deg(\eta)\rrbracket$ with $b_j >1$, then output $(d, m, a, 2)$ with $d:=b_j$ and STOP. If $b_j=1$ for each $j\in\llbracket 0,\deg(\eta)\rrbracket$, then continue. 
		
		\item If $a>0$, then set $a := -a$ and go to Step 5. If $a<0$, then set $a:=|a|+1$.
		
		\item If $a < B$, then go to Step 5. If $a> B$, then set $r := r+1$, $a :=1$ and go to Step 3. 
	\end{enumerate}
	
\noindent		
	\textbf{Output:} either nothing or a $d$ or a quadruple $(d, m, a, i)$, where $d\in\Div(N)\setminus\{1,N\}$, $i\in\{1,2\}$, and $(m, a)\in \mathbb E_i(N)$.
\end{framed}
 
For the input $\bigl(f(X), g(X), N\bigr)=\bigl(X^3-1, (10-X)^3-1, 77\bigr)$, Algorithm 1 outputs the polynomial $h(X)=X^2+67X+23$. This implies that $(3, 10)\notin \mathbb E_1(77)$.
 On the other hand $X^{11}+Y^{11} = 10$ together with $X+Y\neq 10$ has solution $(X, Y ) = (1, 4)$ in $\mathbb F_7$. So $(3, 10) \in \mathbb E_2(77)$.
 
For the input $N$ equal to $77$ (or in $\{143,187,209\}$ or equal to $221$), Algorithm 2 outputs $(d, 3, 1, 2)$ with $d=7$ (or $d=11$ or $d=13$). 

Next we exemplify the calculations of Algorithm 2.

\begin{Ex}\normalfont\label{EX6}
Let $N=77$. 

Step 1. Set $m=3$ and $B=2\log 77 \approx 8.6876$.

Step 2. As $m=3<B$, compute $e_m= \gcd(3, 77)=1$. No factor is found. Set $r=2$ and $a=1$.

Step 3. Check $r=2\leq m-1=2$. Also $\gcd(2, 3)=1$. Compute $r'<3$ such that $2r'=1 \pmod 3$. Thus $r'=2$.

Step 4. $\gcd(2, 3)=1$ so continue.

Steps 5-7. Set $h(X)=(X+(1-X^2)^2-1)^{77}$; reduce $h(X)$ modulo $x^3-1$ and modulo $77$. 

Step 8. Define $f(X)=(1-X^2)^3-1=-3X^2+3X^4-X^6$ and $g(X)=X^3-1$.

Step 9. Replace each coefficient of $f(X)$ and $g(X)$ by its residue modulo $77$. We obtain $f(X)=76X^6+3X^4+74X^2$ and $g(X)=X^3+76$.

Step 10. Divide $f(X)$ by $g(X)$. The reminder is $g_1(X)=74X^2-228X+438976$.

Step 11. $g_1(X)\not \equiv 0 \pmod {77}$, so continue.

Step 12. The leading coefficient of $g_1(X)$ is $74$. Compute $\gcd(74, 77)=1$. 

Step 13. Compute $l'\in\llbracket1,76\rrbracket$ such that $l l'=1 \pmod {77}$. Therefore $l'=51$ and we set $f(X)=X^3+76$. 

Step 14. Replace $g(X)$ by $51(74X^2-228X+438976)$. Thus $g(X)=X^2+76X+26$ modulo $77$. Return to Step 9.

Steps 9-11. We get $g_1(X)=5750X+2052 \not\equiv 0 \pmod {77}$.

Step 12. Leading coefficient $l=5750$. Compute $\gcd(5750, 77)=1$.

Step 13. Compute $l'\in\llbracket1,76\rrbracket$ such that $l l'=1 \pmod {77}$. Therefore $l'=40$ and we set $f(X)=X^2+76X+26$.

Step 14. Replace $g(X)$ by $40(5750X+2052)$. Thus $g(X)=X+75$ modulo $77$. Return to Step 9.

Steps 9-11. We get $g_1(X)=-49 \not\equiv 0 \pmod {77}$.

Step 12. Leading coefficient $l=-49$. Compute $\gcd(-49, 77)=7>1$.

Output: $(d, m, a, i)=(7, 3, 1, 1)$ so the algorithm discovers the nontrivial factor $7$ of $77$.
\end{Ex}

{\bf Acknowledgments.} Second author would like to thank SUNY Binghamton for good working conditions.

\end{document}